\pdfoutput=1
\documentclass[cup7b]{cupbook}
\usepackage{amsmath}
\usepackage{amssymb}
\usepackage{graphicx} 
\pdfoutput=1
\newcommand{\Mn}{{\rm gl}_n}
\newcommand{\GL}{{\rm GL}}
\newcommand{\SL}{{\rm SL}}
\newcommand{\gl}{{\rm gl}}
\newcommand{\sll}{{\rm sl}}
\newcommand{\Hom}{{\rm Hom}}
\newcommand{\Cx}{{\mathbb{C}}}
\newcommand{\Zx}{{\mathbb{Z}}}
\newcommand{\Qx}{{\mathbb{Q}}}
\newcommand{\Nx}{{\mathbb{N}}}
\newcommand{\Dx}{{\mathbb{D}}}
\newcommand{\Rx}{{\mathbb{R}}}

\newcommand{\Px}{{\mathbb{P}}}
\newcommand{\bfone}{{\mathbf{1}}}

\newcommand{\Aut}{{\rm Aut}}
\newcommand{\Mon}{{\rm Mon}}
\newcommand{\DGal}{{\rm DGal}}
\newcommand{\der}{\partial}
\newcommand{\calA}{{\mathcal A}}
\newcommand{\calB}{{\mathcal B}}
\newcommand{\calS}{{\mathcal S}}
\newcommand{\calR}{{\mathcal R}}
\newcommand{\calG}{{\mathcal G}}
\newcommand{\calD}{{\mathcal D}}
\newcommand{\calO}{{\mathcal O}}
\newcommand{\calM}{{\mathcal M}}
\newcommand{\calT}{{\mathcal T}}
\newcommand{\calL}{{\mathcal L}}
\newcommand{\calP}{{\mathcal P}}
\newcommand{\calC}{{\mathcal C}}
\newcommand{\calH}{{\mathcal H}}
\newcommand{\ord}{{\mathrm ord}}

\newcommand{\dd}{{\partial}}
\newcommand{\Soln}{{\rm Soln}}

\newcommand{\Perm}{{\rm Perm}}
\newcommand{\Sets}{{\rm Sets}}
\newcommand{\Rep}{{\rm Rep}}
\newcommand{\Vect}{{\rm Vect}}
\newcommand{\Sym}{{\rm Sym}}
\def\miprod{{\scriptstyle \bigcirc\hspace{-0.23cm}\raisebox{0.018cm}{${\scriptstyle s}$}}\hspace{0.1cm}}

\author[Michael F. Singer]{Michael F. Singer\\Department of Mathematics\\
           North Carolina State University\\
           Raleigh, NC 27695-8205\\
           singer@math.ncsu.edu}

\begin{document}

\newtheorem{thm}{Theorem}[section]
\newtheorem{lem}[thm]{Lemma}
\newtheorem{cor}[thm]{Corollary}
\newtheorem{prop}[thm]{Proposition}
\newtheorem{defin}[thm]{Definition}
\newtheorem{ex}[thm]{Example}
\newtheorem{exs}[thm]{Examples}
 
\chapter[Galois Theory of Linear  Differential Equations]{Introduction to the Galois Theory of Linear  Differential Equations}
\author {Michael F. Singer\\Department of Mathematics\\
           North Carolina State University\\
           Raleigh, NC 27695-8205\\
           singer@math.ncsu.edu}

\section{Introduction} This paper is an expanded version of the 10 lectures I gave as the  2006 London Mathematical Society Invited Lecture Series at the Heriot-Watt University, 31 July - 4 August 2006\footnote{I would like to thank the London Mathematical Society for inviting me to present these lectures, the Heriot-Watt University for hosting them  and the International Centre for the Mathematical Sciences for sponsoring a mini-programme on the Algebraic Theory of Differential Equations to complement these talks.  Thanks also go to the Edinburgh Mathematical Society and the Royal Society of Edinburgh, who provided support for some of the participants and to Chris Eilbeck,  Malcolm MacCallum, and Alexandre Mikhailov for organizing these events.  This material is based upon work supported by the National Science Foundation under Grant No. 0096842 and 0634123}. 
My goal was to give the audience an introduction to  the algebraic, analytic and algorithmic aspects of the Galois theory of linear differential equations by focusing on   some of the main ideas and philosophies and on examples.   There are several texts   (\cite{beukers, kaplansky, DAAG, magid, PuSi2003} to name a few) that give detailed expositions and I hope that the taste offered here will encourage the reader to dine more fully with one of these.

The rest of the paper is organized as follows. In Section~\ref{mfssec1}, {\em What is a Linear Differential Equation?}, I  discuss three ways to think about linear differential equations: scalar equations, linear systems and differential modules.  Just as it is useful to think of linear maps in terms of linear equations, matrices and associated modules, it will be helpful in future sections to go back and forth between the different ways of presenting linear differential equations.  

In Section~\ref{mfssec2}, {\em Basic Galois Theory and Applications}, I will give the basic definitions and describe the Galois correspondence.  In addition I will describe the notion of monodromy and its relation to the Galois theory.  I will end by giving several applications and ramifications, one of which will be to answer the question {\em Although $y = \cos x$ satisfies $y'' + y = 0$, why doesn't $\sec x$ satisfy a linear differential equation?}

In Section~\ref{mfssec3}, {\em Local Galois Theory}, I will discuss the formal solution of a linear differential equation at a singular point and describe the asymptotics which  allow one to associate with it an analytic solution in a small enough sector at that point.  This will involve a discussion of Gevrey asymptotics,  the Stokes phenomenon and its relation to the Galois theory. A motivating question (which we will answer in this section) is {\em In what sense does the function }
\[ f(x) = \int_0^\infty \frac{1}{1+\zeta} e^{-\frac{\zeta}{x}}d\zeta\]
{\em represent the divergent series}
\[\sum_{n\geq 0}(-1)^nn!x^{n+1} \ ?\]

In Section~\ref{mfssec4}, {\em Algorithms}, I turn to a discussion of  algorithms that allow us to determine properties of a linear differential equation and its Galois group.  I will show how category theory, and in particular the tannakian formalism, points us in a direction that naturally leads to algorithms. I will also discuss algorithms to find ``closed form solutions'' of linear differential equations. 

In Section~\ref{mfssec5}, {\em Inverse Problems}, I will consider the problem of which groups can appear as Galois groups of linear differential equations.   I will show how monodromy and the ideas in Section~\ref{mfssec3} play a role as well as ideas from Lie theory.

In Section~\ref{mfssec6}, {\em Families of Linear Differential Equations},  I will look at linear differential equations that contain  parameters and ask {\em How does the Galois group change as we vary the parameters?}.  This will lead to a discussion of a a generalization of the above theory to a Galois theory of parameterized linear differential equations.  

\section{What is a Linear Differential Equation?}\label{mfssec1}
I will develop an algebraic setting for the study of linear differential equations.  Although there are many interesting questions concerning differential equations in characteristic $p$ \cite{ mvdp1,mvdp2,vdp_charp,vdp_reduction},  we will restrict ourselves throughout this paper, without further mention,  to fields of characteristic $0$.
I begin with some basic definitions. 
\begin{defin}\label{def1} 1) A {\em differential ring} $(R,\Delta)$ is a ring $R$ with a  set $\Delta = \{\dd_1, \ldots , \dd_m\}$ of maps ({\em  derivations}) $\dd_i: R \rightarrow R$, such that 
\begin{enumerate}
\item $\dd_i(a+b) = \dd_i(a)+\dd_i(b), \ \ \dd_i(ab) = \dd_i(a)b+a\dd_i(b)$ for all $a,b\in R$, and
\item $\dd_i\dd_j = \dd_j \dd_i$ for all $i,j$.
\end{enumerate}
\noindent 2) The ring $C_R = \{ c \in R \ | \ \dd(c) = 0 \  \forall \ \dd \in \Delta\}$ is called the {\em ring of constants of $R$}.
\end{defin}

When $m = 1$, we say $R$ is an {\em  ordinary differential ring} $ (R,\dd)$. We frequently use the notation $a'$ to denote $\dd(a)$ for $a \in R$. A differential ring that is also a field is called a {\em differential field}.  If $k$ is a differential field, then $C_k$ is also a field.

\begin{exs}{\rm 
1) $(C^{\infty}(\Rx^m), \Delta = \{\frac{\dd}{\dd x_1}, \ldots , \frac{\dd}{\dd x_m}\}$)
 = infinitely differentiable functions on $\Rx^m$.\\[0.1in]
2)  $(\Cx(x_1, \ldots , x_m),  \Delta = \{\frac{\dd}{\dd x_1}, \ldots , \frac{\dd}{\dd x_m}\}$) = field of rational functions\\[0.1in]
3) $(\Cx[[x]], \frac{\dd}{\dd x})$  = ring of formal power series \\[0.05in] \hspace*{.5in}$\Cx((x)) =  \mbox{ quotient field of }\Cx[[x]] =\Cx[[x]][\frac{1}{x}]$\\[0.1in]
4) $(\Cx\{\{x\}\}, \frac{\dd}{\dd x})$  = ring of germs of convergent series \\[0.05in] \hspace*{.5in}$\Cx(\{x\}) =  \mbox{ quotient field of }\Cx\{\{x\}\} =\Cx\{\{x\}\}[\frac{1}{x}]$\\[0.1in]
5)  $(\calM_{\calO}, \Delta = \{\frac{\dd}{\dd x_1}, \ldots , \frac{\dd}{\dd x_m}\}$) = field of functions meromorphic on $\calO^{\rm open, connected} \subset \Cx^m$}
\end{exs}
The following result of Seidenberg \cite{sei58, sei69} shows that many examples reduce to Example 5) above:

\begin{thm} Any differential field $k$, finitely generated over $\Qx$, is isomorphic to a differential subfield of some $\calM_\calO$.
\end{thm}
We wish to consider and compare three different versions of the notion of a linear differential equation.
\begin{defin} Let $(k,\dd)$ be a differential field. \begin{enumerate}
 \item A {\em scalar linear differential equation} is an equation of the form 
 \[L(y) = a_ny^{(n)} + \ldots + a_0y = 0 , \ a_i \in k.\]
 \item A {\em matrix linear differential equation} is an equation of the form
 \[Y' = AY, \ A \in \Mn(k)\]
 where $\Mn(k)$ denotes the ring of $n\times n$ matrices with entries in $k$.
 \item A {\em differential module} of dimension $n$ is an $n$-dimensional $k$-vector space $M$ with a map  $\dd:M\rightarrow M$ satisfying 
\[\dd(fm) = f'm+f\dd m \mbox{ for all } f \in k, m \in M.\]
\end{enumerate}
\end{defin}

We shall show that these three notions are equivalent and give some further properties.\\[0.1in]

\noindent \underline{\bf From scalar to matrix equations:}  Let $L(y) = y^{(n)} + a_{n-1} y^{(n-1)} + \ldots + a_0 y = 0$. If we let
 $y_1 = y, y_2 = y', \ldots y_n = y^{(n-1)}$, then we have
\[\left(\begin{array}{c} y_1\\y_2\\ \vdots \\ y_{n-1}\\y_n \end{array}\right)' = \left(\begin{array}{ccccc}0&1&0&\ldots& 0\\ 0& 0 & 1& \ldots & 0 \\ \vdots &\vdots &\vdots &\vdots &\vdots \\ 0&0&0&\ldots& 1\\-a_0& -a_1& -a_2 & \ldots & -a_{n-1}\end{array}\right) \left(\begin{array}{c} y_1\\y_2\\ \vdots \\ y_{n-1}\\y_n \end{array}\right) \]
We shall write this last equation as $Y' = A_LY$ and refer $A_L$ as the companion matrix of the scalar equation and the matrix equation as the companion equation.  Clearly any solution of the scalar equation yields a solution of the companion equation and {\em vice versa}.\\[0.1in]
\noindent \underline{\bf From matrix equations to differential modules (and back):}\\[0.1in]
Given $Y' = AY, \ A \in \gl_n(k)$,  we construct a differential module in the following way:  Let $M = k^n, \ e_1, \ldots, e_n$ the usual basis.  Define $\dd e_i = -\sum_j a_{j,i} e_j, $ {\em i.e.,} $\dd e = -A^t e$.  Note that if $m = \sum_if_ie_i$ then $\dd m = \sum_i(f_i' - \sum_ja_{i,j}f_j)e_i$.  In particular, we have that $\dd m = 0$ if and only if 
\[\left(\begin{array}{c} f_1\\ \vdots \\ f_n\end{array}\right)' = A  \left(\begin{array}{c} f_1\\ \vdots \\ f_n\end{array}\right)\]
It is this latter fact that motivates the seemingly strange definition of this differential module, which we denote by $(M_A, \dd)$.\\[0.1in]
Conversely, given a differential module $(M, \dd)$, \underline{select} a basis $e = (e_1, \ldots ,e_n)$.  Define $A_M \in \gl_n(k)$ by $\dd e_i = \sum_j a_{j,i}e_j$.  This yields a matrix equation $Y' = AY$.  If $\bar{e} = (\bar{e}_1,  \ldots ,\bar{e}_n)$ is another basis, we get another equation $Y' = \bar{A}Y$.  If $f$ and $\bar{f}$ are vectors with respect to these two bases and $f = B\bar{f}, \ B \in \GL_n(k)$, then 
\[\bar{A} = B^{-1}AB - B^{-1}B' \ . \]
\begin{defin}\label{diffmodule} Let $(M_1, \dd_1)$ and $(M_2,\dd_2)$ be differential modules.\\
1) A {\em differential module homomorphism} $\phi:M_1 \rightarrow M_2$ is a $k$-linear map $\phi$ such that $\phi(\dd_1(m)) = \dd_2(\phi(m))$ for all $m \in M_1$.\\
2) The differential modules $(M_1, \dd_1)$ and $(M_2, \dd_1)$ are {\em isomorphic} if there exists a bijective differential homomorphism $\phi:M_1 \rightarrow M_2$.\\
3) Two differential equations $Y' = A_1Y$ and $Y' = A_2Y$ are {\em equivalent} if the differential modules $M_{A_1}$ and $M_{A_2}$ are isomorphic
\end{defin}
Instead of equivalent, some authors use the term ``gauge equivalent'' or ``of the same type''. \\[0.1in]
Differential modules offer us an opportunity to study linear differential equations in a basis-free environment.  Besides being of theoretical interest, it is important in computations to know that a concept is independent of bases, since this allows one to then select a convenient basis in which to compute.  \\[0.1in]
\noindent Before we show how one can recover a scalar equation from a differential module, we show that the standard constructions of linear algebra can be carried over to the context of differential modules. Let $(M_1, \dd_1)$ and $(M_2, \dd_2)$ be differential modules and assume that for certain bases these correspond to the equations $Y' = A_1Y$ and $Y' = A_2Y$.\\[0.1in]  
\noindent  In Table~\ref{tab1} we list  some standard linear algebra constructions and how they generalize to differential modules. In this table, $I_{n_i}$ represents the $n_i\times n_i$ identity matrix and for two matrices $A = (a_{i,j})$ and $B$, $A\otimes B$ is the matrix where the $i,j$-entry of $A$ is replaced by $a_{i,j}B$.  Also note that if $f \in \Hom_k(M_1,M_2)$ then $\dd(f) = 0$ if and only if $f(\dd m_1) = \dd_2(f(m_1))$, that is, if and only if $f$ is a differential module homomorphism. \\[0.1in]
 \noindent Referring to the table,  it is not hard to show that $\Hom_k(M_1, M_2) \simeq M_1 \otimes M_2^*$ as differential modules.  Furthermore, given $(M,\dd)$ with $\dim_k(M) = n$, corresponding to a differential equation $Y' = AY$, we have that $M \simeq \oplus_{i=1}^n \bfone_k$if and only if there exist $y_1, \ldots , y_n$ in $k^n$, linearly independent over $k$ such that $y_i' = Ay_i$.  \\[0.1in]

\setcounter{table}{0}
\begin{table*}\label{table1}
\caption{}
\label{tab1}
$\begin{array}{|l|l|l|}\hline
\mbox{Construction} & \hspace{.6in}\dd &\mbox{Matrix Equation}\\\hline
(M_1\oplus M_2, \dd) & \dd(m_1\oplus m_2) =  & Y' = \left(\begin{array}{cc}A_1&0\\0&A_2\end{array}\right)Y\\ 
& \ \dd_1m_1\oplus\dd_2 m_2 & \\\hline
(M_1\otimes M_2, \dd) &\dd(m_1\otimes m_2) =  & Y' =  \\
 & \  \dd_1m_1\otimes m_2 + m_1 \otimes \dd m_2 & \ \ (A_1\otimes I_{n_2} + I_{n_1}\otimes A_2)Y\\ \hline
(\Hom_k(M_1,M_2), \dd)& \dd(f)(m) =  &  Y' = YA_2^T-A_1^TY\\
 & \   f(\dd_1m) - \dd_2(f(m)) & \\ \hline 
  \bfone_k =  (k, \dd) =  \mbox{trivial} & \dd \equiv 0 & Y' = 0\\
 \mbox{  differential module} & & \\ \hline
 (M^*, \dd)= & \dd(f)(m) = f(\dd(m))& Y' = -A^TY\\
 \Hom_k(M, \bfone_k) & &  \\ \hline
 \end{array}$
 \end{table*}

\noindent \underline{\bf From matrix to scalar linear differential equations:}
Before I discuss the relationship between matrix and scalar linear differential equations, I need one more concept.

\begin{defin} Let $k$ de a differential field.  The {\em ring of differential operators over $k$} = $k[\dd]$ is the ring of noncommutative polynomials $\{a_n \dd^n + \ldots +a_1 \dd + a_0 \ | \ a_i \in k\}$ with coeffcients in $k$, where multiplication is determined by $\dd \cdot a = a' + a \dd$ for all $a \in k$.
\end{defin}

\noindent We shall refer to the degree of an element $L \in k[\dd]$ as its order $\ord L$. The following properties are not hard to check (\cite{PuSi2003}, Chapter 2.1):
\begin{lem}\label{euclid} Let $L_1 \neq 0, L_2 \in k[\dd]$.\\[0.05in]
\noindent 1) There exist unique $Q, R \in k[\dd]$  with $\ord R < \ord L_1$ such that $L_2 = QL_1+R$.\\[0.05in]
2) Every left ideal of $k[\dd]$ is of the form $k[\dd]L$ for some $L \in k[\dd]$.
\end{lem}

\noindent We note that any differential module $M$ can be considered a left $k[\dd]$-module and conversely, any left $k[\dd]$-module that is finite dimensional as a $k$-vector space is a differential module.    In fact we have a stronger statement:
\begin{thm} (Cyclic Vector Theorem)  Assume there exists an $a \in k$ such that $a' \neq 0$.  Every differential module is of the form $k[\dd]/k[\dd]L$ for some $L \in k[\dd]$. \end{thm}
This result has many proofs (\cite{churchill-kovacic,cope1,cope2, deligne_LNM,  jacobson, katz}  to name a few), two of which can be found in Chapter 2.1 of \cite{PuSi2003}.

\begin{cor} (Systems to Scalar equations) Let $k$ be as above.  Every system $Y' = AY$ is equivalent to a scalar equation $L(y) = 0$. \end{cor}
\begin{proof} Let $A^* = -A^t$.  Apply the Cyclic Vector Theorem to $M_{A^*}$ to find an $L \in k[\dd]$ such that $M_{A^*} = k[\dd]/k[\dd]L$.  If $A_L$ is the companion matrix of $L$, a calculation shows that $M_A \simeq M_{A_L}$.\end{proof}

\noindent We note that the hypothesis that $k$ contain an element $a$ with $a' \neq 0$ is necessary. If the derivation is trivial on $k$, then two matrix equations $Y' = A_1Y$ and $Y' = A_2Y$ are equivalent if and only if the matrices are similar.  There exist many examples of matrices not similar to a companion matrtix.\\[0.05in]
Before we leave (for now) our discussion of the ring $k[\dd]$, I wish to make two remarks.\\[0.05in]
First, we can define a map $i :k[\dd] \rightarrow k[\dd]$  by $ i(\sum a_j\dd^j) = \sum (-1)^j \dd^j a_j,$.  This map is is an involution ($i^2 = id$). Denoting $i(L) = L^*$, we have that $(L_1L_2)^* = L_2^*L_1^*$.  The operator $L^*$ is referred to as the {\em adjoint} of $L$.  Using the adjoint one sees that there is right euclidean division as well and that every right ideal of $k[\dd]$ is also principal.  \\[0.05in]
Second,  Lemma~\ref{euclid}.2 allows us to define 
\begin{defin} Let $L_1, L_2 \in k[\dd]$.\\[0.05in]
1) The {\em least common left multiple LCLM($L_1,L_2$)} of $L_1$ and $L_2$ is the monic generator of $k[\dd]L_1\cap k[\dd]L_2$.\\
2) The {\em greatest common right divisor GCRD($L_1, L_2$)} of $L_1$ and $L_2$ is the monic generator of $k[\dd]L_1+ k[\dd]L_2$.
\end{defin} 
A simple modification of the usual euclidean algorithm allows one to find these objects. One can also define least common right multiples and greatest common left divisors using right ideals.  \\[0.1in]
\noindent \underline{\bf Solutions} I will give properties  of solutions of scalar and matrix linear differential equations and define the notion of the solutions of a differential module.
 Let $(k,\dd)$ be a differential field with constants $C_k$ (see Definition~\ref{def1}). 
 \begin{lem} \label{dependence} Let $v_1, \ldots v_r \in k^n$ satisfy $v_i' = Av_i$, $A \in \gl_n(k)$. If $v_1, \ldots , v_r$ are $k$-linearly dependent,  then they are $C_k$-linearly dependent.
 \end{lem}
 \begin{proof} Assume, by induction,  that $v_2, \ldots, v_{r}$ are $k$-linearly independent and $v_1 = \sum_{i=2}^r a_iv_i, a_i \in k$. We then have
\[ 0 = v_1' - Av_1 = \sum_{i=2}^r a_i'v_i +\sum_{i=2}^r a_i(v_i'-Av_i) = \sum_{i=2}^r a_i'v_i \]
so by assumption, each $a_i' = 0$.\end{proof}
\begin{cor} \label{wronskiancor}Let $(k,\dd)$ be a differential field with constants $C_k$, $A\in \gl_n(k)$ and $L \in k[\dd]$.\\[0.05in]
\noindent 1) The solution space $\Soln_k(Y'=AY)$ of $Y'=AY$ in $k^n$ is a $C_k$-vector space of dimension at most $n$.\\[0.05in]
2) The elements $y_1, \ldots , y_r \in k$ are linearly dependent over $C_k$ if and only if the wronskian determinant 
\[wr(y_1, \ldots , y_r) = \det \left(\begin{array}{ccc} y_1&\ldots &y_r\\ \vdots & \vdots & \vdots \\ y_1^{(r-1)} & \ldots &y_r^{(r-1)} \end{array}\right)\]
is zero.\\[0.05in]
3) The solution space $\Soln_k(L(y)=0) = \{y \in k \ | \ L(y) = 0 \}$ of $L(y)=0$ in $k$ is a $C_k$-vector space of dimension at most $n$.
\end{cor}
\begin{proof} 1) This follows immediately from Lemma~\ref{dependence}.\\[0.05in]
2) If $\sum_{i=1}^r c_iy_i=0$ for some $c_i \in C_k$, not all zero,  then $\sum_{i=1}^r c_iy_i^{(j)}=0$  for all $j$.  Therefore, $ wr(y_1, \ldots  , y_r) =0$.  Conversely if $wr(y_1, \ldots , y_r) =0$, then there exists a nonzero vector $(a_0, \ldots , a_{r-1})$ such that 
\[(a_0, \ldots , a_{r-1})\left(\begin{array}{ccc} y_1&\ldots &y_r\\ \vdots & \vdots & \vdots \\ y_1^{(r-1)} & \ldots &y_r^{(r-1)} \end{array}\right)=0\]
Therefore each $y_i$ satisfies the scalar linear differential equation $L(y) = a_{r-1}y^{(r-1)} + \ldots + a_0y = 0$ so each vector $v_i = (y_i, y_i^\prime, \ldots , y_i^{(r-1)})^t$ satisfies $Y' = A_LY$, where $A_L$ is the companion matrix. Lemma~\ref{dependence} implies that the $v_i$ and therefore the $y_i$ are linearly dependent over $C_k$.\\[0.05in]
3) Apply Lemma~\ref{dependence} to $Y' = A_LY$. \end{proof}
In general, the dimension of the solutions space of a linear differential equation in a field is less than $n$.  In the next section we will construct a field such that over this field the solutions space has dimension $n$.  It will be useful to have the following definition
\begin{defin}\label{fundmatrixdef} Let $(k,\dd)$ be a differential field,  $A\in \gl_n(k)$ and $R$ a differential ring containing $k$.  A matrix $Z \in \GL_n(R)$ such that $Z' = AZ$ is called a {\em fundamental solution matrix} of $Y' = AY$.
\end{defin}
Note that if $R$ is a field and $Z$ is a fundamental solution matrix, then the columns of $Z$ form a $C_R$-basis of the solution space  of $Y' = AY$ in $R$ and that this solution space has dimension $n$.\\[0.1in]
\noindent Let $(M,\dd)$ be a differential module over $k$.  We define the solution space $\Soln_k(M)$ of $(M,\dd)$ to be the kernel $\ker_M\dd$ of $\dd$ on $M$.   As we have noted above, if $\{e_i\}$ is a basis of $M$ and $Y' = AY$ is the associated matrix differential equation in this basis, then $\sum_iv_ie_i \in \ker \dd$ if and only if $v' = Av$ where $v = (v_1, \ldots , v_n)^t$.  If $K \supset k$ is a differential extension of $k$, then $K\otimes_kM$ can be given the structure of a $K$-differential module, where $\dd(a\otimes m) = a'\otimes m+ a \otimes \dd m$.  We then define $\Soln_K(M)$ to be the kernel $\ker(\dd, K\otimes_kM )$ of $\dd$ on $K\otimes_kM$.

\section{Basic Galois Theory and Applications}\label{mfssec2}

\noindent \underline{\bf Galois theory of polynomials} The idea behind the Galois theory of polynomials is to associate to a polynomial a group (the group of symmetries of the roots that preserve all the algebraic relations among these roots) and deduce properties of the roots from properties of this group (for example, a solvable group implies that the roots can be expressed in terms of radicals). This idea can be formalized in the following way.\\[0.1in]
Let $k$ be a field (of characteristic $0$ as usual) and let $P(X) \in k[X]$ be a polynomial  of degree $n$ without repeated roots (i.e., ${\rm  GCD}(P.P') = 1$).  Let \[S =  k[X_1, \ldots ,X_n, \frac{1}{\prod(X_i-X_j)}]\] (the reason for the term $\frac{1}{\prod(X_i-X_j)}$ will be explained below) and let $I = (P(X_1), \ldots , P(X_n)) \triangleleft  S$ be the ideal generated by the $P(X_i)$.  The ring $S/I$ is generated by $n$ distinct roots of $P$ but does not yet reflect the possible algebraic relations among these roots.  We therefore consider any maximal ideal $M$ in $S$ containing $I$.  This ideal can be thought of as a maximally consistent set of algebraic relations among the roots.  We define
\begin{defin} 1) The {\em splitting ring} of the polynomial $P$ over $k$ is the ring
\[R = S/M = k[X_1, \ldots ,X_n, \frac{1}{\prod(X_i-X_j)}]/M \ .\]
2) The {\em Galois group} of $P$ (or of $R$ over $k$) is the group of automorphisms $\Aut(R/k)$.
\end{defin}
Note that since $M$ is a maximal ideal, $R$ is actually a field.  Furthermore, since $S$ contains $\frac{1}{\prod(X_i-X_j)}$, the images of the $X_i$ in $R$ are {\em distinct} roots of $P$.  So, in fact, $R$ coincides with the usual notion of a splitting field (a field generated over $k$ by the distinct roots of $P$) and as such, is unique up to $k$-isomorphism.  Therefore, $R$ is independent of the choice of maximal ideal $M$ containing $I$.  We will follow a similar  approach with linear differential equations.\\[0.1in]
\underline{\bf Galois theory of linear differential equations} Let $(k,\dd)$ be a differential field and $Y' = AY, A\in \gl_n(k)$ a matrix differential equation over $k$. We now want the Galois group to be the group of symmetries of solutions preserving  all algebraic {\em and differential} relations.  We proceed in a way similar to the above.\\[0.1in]
Let 
\[ S = k[Y_{1,1}, \ldots , Y_{n,n}, \frac{1}{\det (Y_{i,j})}]\]
where $Y = (Y_{i,j})$ is an $n\times n$ matrix of indeterminates.  We define a derivation on $S$ by setting $Y' = AY$.  The columns of $Y$ form $n$ independent solutions of the matrix linear differential equation $Y' =AY$ but we have not yet taken into account other possible algebraic and differential relations. To do this, let $M$ be any maximal {\em differential} ideal and let $R = S/M$.  We have now constructed a ring that satisfies the following definition
\begin{defin} Let $(k,\dd)$ be a differential field and $Y' = AY, A\in \gl_n(k)$ a matrix differential equation over $k$. A {\em Picard-Vessiot ring (PV-ring)} for  $Y' = AY$ is a differential ring $R$ over $k$ such that
\begin{enumerate} \item $R$ is a simple differential ring ({\em i.e.,} the only differential ideals are $(0)$ and $R$).
\item There exists a fundamental matrix $Z \in \GL_n(R)$ for the equation $Y' = AY$.
\item R is generated as a ring by $k$, the entries of $Z$ and $\frac{1}{\det Z}$.
\end{enumerate}
\end{defin} 
One can show that a PV-ring must be a domain and, if $C_k$ is algebraically closed, then any PV-rings for the same equation are $k$-isomorphic as differential rings and $C_k = C_R$.  Since a PV-ring is a domain we define a {\em PV-field} $K$ to be the quotient field of a PV-ring.  Assuming that $C_k$ is algebraically closed this is the same as saying that $K$ is generated over $k$ by the entries of a fundamental solution matrix and that $C_K = C_k$.  These facts are proven in Chapter 1.2 of \cite{PuSi2003}.  {\em From now on, we will always assume that $C_k$is algebraically closed} (see \cite{dyckerhoff, hendriks_vdp, DAAG, ulmer_prime} for results  when this is not the case) .

\begin{exs}\label{exPV}1) $k = \Cx(x) \ \ x' = 1 \ \ \alpha \in \Cx \ \ Y' = \frac{\alpha}{x}Y$
$S = k[Y,\frac{1}{Y}] \ \ \ \ Y' = \frac{\alpha}{x}Y$\\[0.1in]
\underline{Case 1: $\alpha = \frac{n}{m}, \ \ GCD(n,m) = 1$.}  I claim that   $\calR = k[Y,\frac{1}{Y}]/[Y^m - x^n] = k(x^{\frac{m}{n}})$.  To see this note that  the ideal $(Y^m - x^n) \triangleleft k[Y,\frac{1}{Y}]$ is a maximal ideal and closed under differentiation since $(Y^m-x^n)' = \frac{n}{x}(Y^m - x^n)$.\\[0.1in]
\underline{Case 2: $\alpha \notin \Qx$.} In this case, I claim that $(0)$ is a maximal differential  ideal and so $R = k[Y,\frac{1}{Y}]$.  To see this, first note that for any $f \in \Cx(x)$,
\[\frac{f'}{f} = \sum \frac{n_i}{(x-\alpha_i)}, n_i \in \Zx, \alpha_i \in \Cx\]
and this can never equal $\frac{N\alpha}{x}$.  Therefore $Y' = \frac{N\alpha}{x}Y$ has no nonzero solutions in $\Cx(x)$.  Now assume $I$ is a proper nonzero differential ideal in $K[Y, \frac{1}{Y}]$.  One sees that $I$ is of the form $I = (f)$ where  $f = Y^N+a_{N-1}Y^{N-1} + \ldots + a_0$, $N>0$.  Since $Y$ is invertible, we may assume that $a_0 \neq 0$.  Since $f' \in (f)$,  by comparing leading terms we have that $f' = \frac{N\alpha}{x}f$ and so $a_0' = \frac{N\alpha}{x}a_0$, a contradiction.\\[0.05in]
2) $k = \Cx(x) \ \ x' = 1  \ \ Y' = Y$.  Let $S = k[Y,\frac{1}{Y}]$.  An argument similar to Case 2 above shows that $(0)$ is the only proper differential ideal so the PV-ring is $k[Y, \frac{1}{Y}]$.
\end{exs}

As expected the differential Galois group is defined as
\begin{defin} Let $(k,\dd)$ be a differential field and $R$ a PV-ring over $k$. The {\em differential Galois group of $R$ over $k$}, $\DGal(R/k)$ is the group $\{ \sigma: R\rightarrow R \ | \ \sigma \mbox{ is a differential $k$-isomorphism}\}$. 
\end{defin}
If $R$ is the PV-ring associated with $Y' = AY, A\in \gl_n(k)$, we sometimes denote the differential Galois group by $\DGal(Y' = AY)$.  If $K$ is a PV-field over $k$, then the $k$-differential automorphisms of $K$ over $k$ are can be identified with the differential $k$-automorphisms of $R$ over $k$ where $R$ is a PV-ring whose quotient field is $K$. \\[0.1in]
Let $R$ be a PV-ring for the equation $Y'=AY$ over $k$ and $\sigma \in\DGal(R/k)$. Fix a fundamental solution matrix $Z$ and let $\sigma \in \DGal(R/k)$. We then have that  $\sigma(Z)$ is again a fundamental solution matrix of $Y'=AY$ and, a calculation shows that $(Z^{-1}\sigma(Z))' = 0$.  Therefore, $\sigma(Z) = Zm_\sigma$ for some $m_\sigma\in \GL_n(C_k)$.  This gives an injective group homomorphism $\sigma \mapsto m_\sigma$ of $\DGal(R/k)$ into $\GL_n(C_\sigma)$. If we select a different fundamental solution matrix, the images of the resulting maps would be conjugate.  The image has a further property.  
\begin{defin}A subgroup $G$ of $\GL_n(C_k)$ is said to be a {\em linear algebraic group} if it is  the zero set in $\GL_n(C_k)$ of a system of polynomials over $C_k$ in $n^2$ variables.
\end{defin}

With the identification above $\DGal(R/k) \subset \GL_n(C_k)$ we have 
\begin{prop} $\DGal(R/k) \subset \GL_n(C_k)$ is a linear algebraic group.\end{prop}
\begin{proof} Let $R = k[Y,\frac{1}{\det Y}]/M, \ M = (f_1, \ldots, f_m)$.  We may assume that the $f_i \in k[Y_{1,1}, \ldots , Y_{n,n}]$.  We may identify $\DGal(R/k)$ with the subgroup of $\sigma \in \GL_n(C_k)$ such that $\sigma(M) \subset M$.  Let $W_N$ be the $C_k$-space of polynomials in $k[Y_{1,1}, \ldots , Y_{n,n}]$ of degree at most $N$.  Note that $\GL_n(C_k)$ acts on $k[Y_{1,1}, \ldots , Y_{n,n}]$ by substitution and  leaves $W_N$ invariant.  If $N$ is the maximum of the degrees of the $f_i$ then $V_N = W_N \cap M$ generates $M$ and we may identify $\DGal(R/k)$ with the group of $\sigma \in \GL_n(C_k)$ such that $\sigma(V_N) \subset V_N$.  Let $\{e_\alpha\}_\calA$ be a $C$-basis of $W_N$ with  $\{e_\alpha\}_{\alpha \in \calB\subset \calA}$ a $C$-basis of $V_N$.  For any $\beta \in \calA \mbox{ there exist polynomials }p_{\alpha \beta}(x_{ij})$ such that for any $\sigma = (c_{ij})\in \GL_n, \ \sigma(e_{\beta}) = \sum_\alpha p_{\alpha \beta} (c_{ij}) e_\alpha$.  We then have that $\sigma = (c_{ij}) \in \DGal(R/k)$ if and only if $p_{\alpha \beta}(c_{ij}) = 0$ for all $ \beta \in \calB$ and  $ \alpha \notin \calB.$
\end{proof}
\begin{exs} 1) $k = \Cx(x) \ \ x'= 1 \ \ \  Y' = \frac{\alpha}{x} Y\ \ DGal \subset \GL_1(\Cx)$.  Form  the above description of the differential Galois group as the group laving a certain ideal invariant, we see (using the information of Example~\ref{exPV}):\\[0.05in]
\noindent 
\underline{Case 1: $\alpha = \frac{n}{m},  GCD(n,m) = 1$}. We then have  $R = \Cx[Y,\frac{1}{Y}]/(Y^m-x^n) = \Cx(x)[x^\frac{n}{m}]$. Therefore  $\DGal = \Zx/m\Zx = \{(c)\ | \ c^n-1=0\}\subset \GL_1(\Cx)$\\[0.05in]
\noindent
\underline{Case 2: $\alpha \notin \Qx$}. we have $R = \Cx(x)[Y,\frac{1}{Y}]/(0)$. Therefore $\DGal= \GL_1(\Cx)$.\\[0.1in]
2) $k = \Cx(x) \ \ x'= 1  \ Y' = Y$.  As above, one sees that $\DGal= \GL_1(\Cx)$
\end{exs}

\noindent In general, it is not easy to compute Galois groups (although there is a complete algorithm due to Hrushovski  \cite{Hrushovski}).  We will give more examples in Sections \ref{mfssec3}, \ref{mfssec4}, and \ref{mfssec5}.\\[0.1in]
\noindent One can define the differential Galois group of a differential module in the following way. Let $M$ be a differential module.  Once we have selected a basis, we may assume that $M \simeq M_A$, where $M_A$ is the differential module associated with $Y' = AY$.  We say that $Y'=AY$ is a  matrix linear differential equation associated with $M$.
\begin{defin} Let $M$ be a differential module over $k$.  The {\em differential Galois group of $M$} is the differential Galois group of an associated matrix linear differential equation.
\end{defin}
To show that this definition makes sense we must show (see the discussion before Definition~\ref{diffmodule}) that equivalent matrix differential equations  have the same differential Galois groups. Let $Y' = A_1Y$, $Y' = A_2Y$ be equivalent equations and let $A_1 = B^{-1}A_2B - B^{-1}B', \ \ B\in \GL_n(k)$.  If $R = k[Z_1, \frac{1}{Z_2}]$ is a PV-ring for $Y' = A_1Y$, then a computation shows that $Z_2=BZ_1$ is a fundamental solution matrix for $Y' = A_2Y$ and so $R$ will also be a PV-ring for $Y' = A_2Y$.  The equation $Z_2 = BZ_1$ furthermore shows that the action of the differential Galois group (as a matrix group) on $Z_1$ and $Z_2$ is the same.

\noindent We now state
\begin{thm}{\bf (Fundamental Theorem of Differential Galois Theory)} Let $k$ be a differential field with algebraically closed constants $C_k$. Let $K$ be a Picard-Vessiot field with differential Galois group $G$.\\[0.05in]
1) There is a bijective correspondence between Zariski-closed subgroups $H \subset G$ and differential subfields $F$, $k \subset F\subset K$ given by
\begin{eqnarray*} 
H \subset G  &\mapsto & K^H  = \{ a \in K \ | \ \sigma(a) = a \mbox{ for all } \sigma \in H \}\\
k \subset F \subset K &\mapsto& \DGal(K/F) = \{ \sigma \in G \ | \ \sigma(a) = a \mbox{ for all } a \in F\}
\end{eqnarray*}
2) A differential subfield $k\subset F \subset K$ is  a Picard-Vessiot extension of $k$ if and only if  $\DGal(K/F)$ is a normal subgroup of $G$, in which case $\DGal(F/k) \simeq\DGal(K/k) \DGal(K/F)$.
\end{thm}
The Zariski closed sets form the collection of closed sets of a topology on $\GL_n$ and one can speak of closures, components, etc. in this topology. In this  topology, each linear algebraic group $G$ may be written (uniquely) as the finite disjoint union of connected closed subsets (see Appendix A of \cite{PuSi}  or \cite{springer} for a further discussion of linear algebraic groups).  The connected component containing the identity matrix is denoted by $G^0$ and is referred to as the identity component. It is a normal subgroup of $G$.  As a consequence of the Fundamental Theorem, one can show  the folowing:
\begin{cor}\label{fundcor} Let $k, K$ and $G$ be as above.\\[0.05in]
1) For $a \in K$, we have that  $\ a\in k \Leftrightarrow \sigma(a) = a \ \forall \sigma \in DGal(K/k)$\\[0.05in]
2) For $H \subset DGal(K/k),$ the Zariski closure $\overline{H}$ is  $\DGal(K/k)$ if and only if  $K^H = k$\\[0.05in]
3) $K^{G^0}$ is precisely the algebraic closure $\tilde{k}$ of $k$ in $K$. 
\end{cor}
 I refer the reader to Chapter 1.3 of \cite{PuSi2003} for the proof of the Fundamental Theorem and its corollary.
 
 \begin{ex}$ y' = \frac{\alpha}{x}y \ \ \ \alpha \notin \Qx , \ \  k = \Cx(x) \ \ K = k(x^\alpha)  \ \ DGal = \GL_1 = \Cx^*$

\vspace{.1in}
\noindent
The Zariski-closed subsets of $\Cx^*$ are finite so the closed subgroups are finite and cyclic. We have the following correspondence 
\begin{center}
\begin{tabular}{ccc}
Groups & \hspace{.75in} & Fields\\
$\{1\}$ &$\Leftrightarrow$ & $k(x^\alpha)$\\
$\cap$& & $\cup$\\
$\{1, e^{2\pi i/n}, \ldots , e^{2(n-1)\pi i/n}\}$ & $\Leftrightarrow$& $k(x^{n\alpha})$\\
$\cap$& & $\cup$\\
$\Cx^*$&$\Leftrightarrow$& $k$
\end{tabular}
\end{center}
\end{ex}
The Fundamental Theorem also follows from a deeper fact giving geometric meaning to  the Picard-Vessiot ring and relating this to the Galois group.  \\[0.1in]
\noindent{\bf \underline{Galois theory and torsors}}  Let $k$ be a field (not necessarily a differential field) and $\bar{k}$ its algebraic closure. I wish to refine the notion of a Zariski closed set in order to keep track of the field over which these objects are defined. 
\begin{defin} 1) A Zariski closed subset $V \subset \bar{k}^n$ is called a {\em variety defined over $k$ or $k$-variety} if it the set of common zeroes of polynomials in $k[X_1, \ldots X_n]$.\\[0.05in]
2) If $V$ is a variety defined over $k$, the {\em $k$-ideal of $V$, $I_k(V),$} is $\{ f\in k[X_1, \ldots ,X_n] \  | \ f(a) = 0 \mbox{ for all } a \in X\}$.\\[0.05in]
3) If $V$ is a variety defined over $k$, the {\em $k$-coordinate ring, $k[V]$, of $V$} is the ring $k[X_1, \ldots , X_n]/I_k(V)$.\\[0.05in]
4) If $V\subset\bar{k}^n$ and $W\subset\bar{k}^n$, a {\em $k$-morphism $f:V\rightarrow W$} is an $n$-tuple of polynomials $f = (f_1, \ldots ,f_n)$ in $k[X_1, \ldots , X_n]$ such that $f(a) \in W$ for all $a \in V$.\\[0.05in]
5) We say two $k$-varieties $V$ and $W$ are {\em $k$-isomorphic} if there are $k$-morphisms $f:V\rightarrow W$ and $g:W\rightarrow V$ such that $fg$ and $gf$ are the identity maps.
\end{defin}

Note that the coordinate ring is isomorphic to the ring of $k$-morphisms from $V$ to $k$.  Also note that a $k$-morphism $f:V\rightarrow W$ induces a $k$-algebra homomorphism $f^*:k[W] \rightarrow k[V] $ given by $f^*(g)(v) = g(f(v))$ for $v \in V$ and $g \in k[W]$.

\begin{exs} 1) The affine spaces $k^n$ are $\Qx$-varieties and, as such, their coordinate rings are just the rings of polynomials $\Qx[X_1, \ldots , X_n]$. \\[0.05in]
2) The group of invertible matrices $\GL_n(\bar{k})$ is also a $\Qx$-variety. In order to think of it as a closed set, we must embed each
invertible matrix $A$ in the $n+1 \times n+1$ matrix $\left(\begin{array}{cc}A&0\\0&\det(A)\end{array}\right)$. The defining ideal in $\Qx[ X_{1,1}, \ldots, X_{n+1, n+1}]$ is generated by 
\[X_{n+1, 1} , \ldots , X_{n+1, n}, X_{1, n+1}, \ldots ,X_{n, n+1}, (\det X) - X_{n+1, n+1}\]
where $X$ is the $n\times n$ matrix with entries $X_{i,j}, 1\leq i \leq n, 1\leq j \leq n$. 
The coordinate ring is the ring $\Qx[X_{1,1}, \ldots ,X_{n,n}, \frac{1}{\det X}]$
\end{exs}
\begin{exs}1) Let $k = \Qx$ and $V = \{a \ | \ a^2 - 2 = 0\} \subset \bar{\Qx}$. We have that $I_\Qx(V) = (X^2-2) \subset \Qx[X]$ and $k[V] = \Qx[X]/(X^2-2) = \Qx(\sqrt{2})$. Note that in this setting ``$\sqrt{2}$'' is a function on $V$, defined over $\Qx$ with values in $\bar{\Qx}$  and not a number!\\[0.05in]
2) Let $k = \Qx$ and $W = \{a \ | \ a^2 - 1 = 0\} = \{\pm 1\} \subset \bar{\Qx}$. We have that $I_\Qx(W) = (X^2 - 1)$ and $k[W] = \Qx[X]/(X^2-1) = \Qx\oplus \Qx$.  \\[0.05in]
3) Note that $V$ and $W$ are not $\Qx$-isomorphic since such an isomorphism would induce an isomorphism of their $\Qx$-coordinate rings.  If we consider $V$ and $W$ as $\bar{\Qx}$-varieites, then the map $f(x) = (\sqrt{2}(x-1)+\sqrt{2}(x+1))/2$ gives a $\bar{\Qx}$-isomorphism of $W$ to $V$.
\end{exs}  
We shall be interested in certain $k$-subvarieties of $\GL_n(\bar{k})$ that have a group acting on them.  These are described in the  following definition (a more general definition of torsor is given in Appendix A of \cite{PuSi2003} but the definition below suffices for our present needs).

\begin{defin} Let $G \subset \GL_n(\bar{k})$ be a linear algebraic group defined over $k$ and $V\subset \GL_n(\bar{k})$ a $k$-variety. We say that $V$ is a {\em $G$-torsor (or $G$-principal homogeneous space) over $k$} if for all $v, w \in V$ there exists a unique $g\in G$ such that $vg=w$. Two $G$-torsors $V, W \in \GL_n(\bar{k})$ are said to be {\em $k$-isomorphic} if there is a $k$-isomorphism between the two commuting with the action of $G$.
\end{defin}

\begin{exs} 1) $G$ itself is always a  $G$-torsor.\\[0.05in]
2) Let $k = \Qx$ and $V = \{a \ | \ a^2 - 2 = 0\} \subset \bar{\Qx}$.  We have that $V \subset \bar{\Qx}^* =  \GL_1(\bar{\Qx})$. Let $G = \{\pm 1\} \subset \GL_1(\bar{\Qx})$. It is easy to see that $V$ is a $G$-torsor. As noted in the previous examples, the torsors $V$ and $G$ are not isomorphic over $\Qx$ but are isomorphic over $\bar{\Qx}$.
\end{exs}

We say a $G$-torsor over $k$ is {\em trivial} if it is $k$-isomorphic to $G$ itself. One has the following criterion.
\begin{lem} A $G$-torsor $V$ over $k$ is trivial if and only if it contains a point with coordinates in $k$. In particular, any $G$-torsor over an algebraically closed field is trivial.
\end{lem}
\begin{proof} If $V$ contains a point $v$ with coordinates in $k$ then the map $g\mapsto xg$ is defined over $k$ and gives a $k$-isomorphism between $G$ and $V$ (with inverse $v \mapsto x^{-1}v$).  If $V$ is trivial and $f:G\rightarrow V$ is  an isomorphism defined over $k$, then $f(id)$ has coordinates in $k$ and so is a $k$-point of $V$.\end{proof}

For connected groups $G$, there are other fields for which all torsors are trivial.
\begin{prop} Let $G$ be a {\em connected} linear algebraic group defined over $k$ and  $V$ a $G$-torsor over $k$. If $k$ is an algebraic extension of 
$\Cx(x), \ , \Cx((x)),$or $ \Cx(\{x\})$, then $V$ is a trivial $G$-torsor.
\end{prop}
\begin{proof} This result is due to Steinberg, see \cite{serre}, Chapter III.2.3
\end{proof}

Another way of saying that $V$ is a $G$ torsor is to say that $V$ is defined over $k$ and is a left coset of $G$ in $\GL_n(\bar{k})$.  Because of this one would expect that $V$ and $G$ share many geometric properties. The one we will be interested in concerns the dimension. We say that a $k$-variety $V$ is {\em irreducible} if it is not the union of two proper $k$-varieties.  This is equivalent to saying that $I_k(V)$ is a prime ideal or, equivalently, that $k[V]$ is an integral domain.  For an irreducible $k$-variety $V$,  one defines the {\em dimension of $V$, $\dim_k(V)$} to be the transcendence degree of $k(V)$ over $k$, where $k(V)$ is the quotient field of $k[V]$.  For example, if $k = \Cx$ and $V$ is a $k$-variety that is nonsingular as a complex manifold, then the dimension in the above sense is just the dimension as a complex manifold (this happens for example when $V$ is a subgroup of $\GL_n(\Cx)$). In general, any $k$-variety $V$ can be written as the irredundant union of irreducible $k$-varieties, called the irreducible components of $V$  and the dimension of $V$ is defined to be the maximum of the dimensions of its irreducible components. One can show that if  $k \subset K$, then $\dim_k(V) = \dim_K(V)$ and in particular, $\dim_k(V) = \dim_{\bar{k}}V$ (see \cite{PuSi2003}, Appendix A for more information about $k$-varieties). 
\begin{prop} If $V$ is a $G$-torsor, then $\dim_k(V) = \dim_k(G)$.
\end{prop}
\begin{proof} As noted above, any $G$-torsor over $k$ is trivial when considered as a $\bar{k}$ torsor.  We therefore have $\dim_k(V) = \dim_{\bar{k}}(V) = \dim_{\bar{k}}(G)= \dim_k(G)$.\end{proof}
The reason for introducing the concept of a $G$-torsor in this paper is the following proposition.  We use the fact that if $V$ is a $G$-torsor over $k$ and $g \in G$, then the isomorphism $\rho_g:v\mapsto vg$ induces an isomorphism $\rho_g^*:k[V] \rightarrow k[V]$.
\begin{prop} Let $R$ be a Picard-Vessiot extension with differential Galois group $G$.  Then $R$ is the coordinate ring of a $G$-torsor $V$ defined over $k$. Furthermore, for $\sigma \in G$, the Galois action of $\sigma$ on $R = k[V]$ is the same as $\rho_\sigma^*$.
\end{prop}
\begin{cor} \label{torsorcor}1) If $K$ is a Picard-Vessiot field over $k$ with Galois group $G$, then $\dim_{C_k} G = tr.deg.(K/k)$.\\[0.05in]
2) If $R$ is a Picard-Vessiot ring over $k$ with Galois group $G$ where $k$ is either algebraically closed,  $\Cx(x), \  \Cx((x)),$ or $ \Cx(\{x\})$, then $R \simeq k[G]$.
\end{cor}
We refer to Chapter 1.3 of \cite{PuSi2003} for a proof of the proposition and note that the Fundamental Theorem of Galois Theory is shown  there to follow from this proposition. \\[0.2in]
{\bf \underline{Applications and ramifications.}} I shall discuss the following applications and ramifications of the  differential Galois theory
\begin{itemize}
\item Monodromy
\item Factorization of linear differential operators
\item Transcendental numbers
\item Why doesn't $\sec x$ satisfy a linear differential equation?
\item Systems of linear partial differential equations
\end{itemize}
Another application concerning   solving linear differential equations in terms of special functions and algorithms to do this will be discussed in  Section~\ref{mfssec4}.  The Picard-Vessiot theory has also been used to identify obstructions to complete integrability of Hamiltonian systems (see \cite{audin} or \cite{morales}) \\[0.2in]
\underline{\em Monodromy.}  We will now consider differential equations over the complex numbers, so let $k = \Cx(x), \ x' = 1$ and 
\begin{eqnarray}\label{eqnmono}\frac{dY}{dx} &=& A(x) Y, \ A(x) \in \gl_n(\Cx(x)) \end{eqnarray} 
\begin{defin} A point $x_0 \in \Cx$ is an {\em ordinary point} of equation~(\ref{eqnmono}) if all the entries of $A(x)$ are analytic at $x_0$, otherwise $x_0$ is a {\em singular point}. \end{defin}
One can also consider the point at infinity, $\infty$.  Let $t = \frac{1}{x}$.  We then have that $\frac{dY}{dt} = -\frac{1}{t^2}A(\frac{1}{t}) Y$.  The point $x=\infty$ is an ordinary or singular point precisely when $t=0$ is an ordinary or singular point of the transformed equation.
\begin{ex} $\frac{dy}{dx} = \frac{\alpha}{x} y \Rightarrow \frac{dy}{dt} = -\frac{\alpha}{t} y$ so the singular points are $\{0, \infty\}$.
\end{ex}
Let $S = \Cx\Px^1$ be the Riemann Sphere and let $X  = S^2 - \{ \mbox{singular points of (\ref{eqnmono})}\}$ and let $x_0 \in X$.  From standard existence theorems, we know that in a small neighborhood $\calO$ of $x_0$, there exists a fundamental matrix $Z$ whose entries are analytic in $\calO$.  Let $\gamma$ be a closed curve in $X$ based at $x_0$.  Analytic continuation along $\gamma$ yields a new fundamental solution matrix $Z_\gamma$ at $x_0$ which must be related to the old fundamental matrix as
\begin{eqnarray*} Z_\gamma &=& Z D_\gamma\end{eqnarray*}
where $D_\gamma \in \GL_n(\Cx)$.  One can show that $D_\gamma$ just depends on the homotopy class of $\gamma$.  This defines  a homomorphism
\begin{eqnarray*}
\Mon: \pi_1(X, x_0) &\rightarrow& \GL_n(\Cx)\\
    \gamma &\mapsto& D_\gamma
\end{eqnarray*}
 \begin{defin} The homomorphism $\Mon$ is called the {\em monodromy map} and its image is called the {\em monodromy group}.
 \end{defin}
 Note that the monodromy group depends on the choice of $Z$ so it is only determined up to conjugation.  Since analytic continuation preserves analytic relations, we have that the monodromy group is contained in the differential Galois group.
 \begin{exs} 1) $ y' = \frac{\alpha}{x} y \ \ (y = e^{\alpha \log x} = x^\alpha)$\ \ Singular points = $\{0, \infty\}$
\[\Mon(\pi_1(X,x_0)) = 
\{(e^{2\pi i \alpha})^n \ | \ n \in \Zx \}\]
If $\alpha \in \Qx$ this image is finite and so equals the differential Galois group.  If $\alpha \notin \Qx$, then this image is infinite and so Zariski dense in $\GL_1(\Cx)$.  In either case the monodromy group is Zariski dense in the differential Galois group. \\[0.05in]
2)  $ y'=y \ \ (y = e^x)$,  \ \  Singular points = $\{\infty\}$\[\Mon(\pi_1(X,x_0)) = \{1\} \]  
In this example the differential Galois group is $\GL_1(\Cx)$ and the monodromy group is not Zariski dense in this group. 
 \end{exs}
The second example shows that the monodromy group is not necessarily Zariski dense in the differential Galois group but the first example suggests that under certain conditions it may be.  To state these conditions we need the following  defintions.
\begin{defin} \label{prelimreg}1) An open sector $\calS = \calS(a,b,\rho)$ at a point $p$ is the set of complex numbers $z \neq p$ such that $\arg(z-p) \in (a,b)$ and $|z-p| < \rho(\arg(z-p))$ where $\rho:(a,b)\rightarrow \Rx_{>0}$ is a continuous function. \\[0.1in]
2) A singular point $p$ of $Y' = AY$ is a {\em  regular singular point} if for any open sector $\calS(a,b,\rho)$ at $p$ with $|a-b|<2\pi$, there exists a fundamental solution matrix $Z = (z_{ij})$, analytic in $\calS$ such that for some  positive constant  $c$ and  integer $N$,  $|z_{ij}|< c |x^N|$ as $x\rightarrow p$ in $\calS$.
\end{defin}
Note that in the first example above, the singular points are regular while in the second they are not.  For scalar linear differential equations $L(y) = y^{(n)} + a_{n-1}y^{(n-1)} + \ldots +a_0y = 0$,  the criterion of Fuchs states that $p$ is a regular singular point if and only if each $a_i$ has a pole of order $\leq n-i$ at $p$ (see Chapter 5.1 of \cite{PuSi2003} for a further discussion of regular singular points). For equations with only regular singular points we have
\begin{prop} \label{schlesinger}(Schlesinger) If $Y' = AY$ has only regular singular points, then $\Mon(\pi_1(X,x_0))$ is Zariski dense in $DGal$.
\end{prop} 
\begin{proof} Let $K$ = PV extension field  of $k$ for $Y' = AY$. We have $Mon(\pi_1(X,x_0)) \subset DGal(K/k)$.  By Corollary~\ref{fundcor} it is enough to show that if $f \in K$ is left fixed by $Mon(\pi_1(X,x_0))$ then $f \in k = \Cx(x)$. If $f$ is left invariant by analytic continuation on $X$, then $f$ is a single valued meromorphic function on $X$.  Let $Z = (z_{ij}), z_{ij} \in K$ be a fundamental solution matrix and let $f \in \Cx(x, z_{ij})$. Regular singular points imply that $f$ has polynomial growth near the singular points of $Y'=AY$ and is at worst meromorphic at other points on the Riemann Sphere. Cauchy's Theorem implies that $f$ is meromorphic on the Riemann Sphere and Liouville's Theorem implies that $f$ must be a rational function. \end{proof}

\noindent In Section~\ref{mfssec3}, we shall see what other analytic information one needs to determine the differential Galois group in general.\\[0.2in]

\noindent \underline{\em Factorization of linear differential operators.} In Section \ref{mfssec4}, we will see that factorization properties of linear differential operators and differential modules play a crucial role in algorithms to calculate properties of Galois groups.  Here I shall relate factorization of operators to the Galois theory.  \\[0.1in]
Let $k$ be a differential field, $k[\dd]$ the ring of differential operators over $k$ and $L \in k[\dd]$.  Let $K$ be the PV-field over $k$ corresponding to $L(y) = 0$, $G$ its Galois group and $V = \Soln_K(L(y) = 0)$. 
\begin{prop}\label{invsubprop} There is a correspondence among the following objects \begin{enumerate}
\item Monic $L_0 \in k[\dd]$ such that $L = L_1\cdot L_0$.
\item Differential submodules of $k[\dd]/k[\dd]L$.
\item $G$-invariant subspaces of $V$.
\end{enumerate}
\end{prop}
\begin{proof} Let $W \subset V$ be  $G$-invariant and  $B = \{b_1, \ldots ,b_t\}$ a $C_k$-basis of $W$.  For $\sigma \in G$, let $[\sigma]_B$ be the matrix of $\sigma|_W$ with respect to $B$,
 
\[\sigma(\left[\begin{array}{ccc}b_1&\ldots&b_t\\ \vdots & \vdots &\vdots\\b_1^{(t-1)}&\ldots&b_t^{(t-1)}\end{array}\right]) = \left[\begin{array}{ccc}b_1&\ldots&b_t\\ \vdots & \vdots &\vdots\\b_1^{(t-1)}&\ldots&b_t^{(t-1)}\end{array}\right] [\sigma]_B \ .\]

\noindent Taking determinants we have
$\sigma(wr(b_1, \ldots , b_t)) = wr(b_1, \ldots , b_t)\det([\sigma]_B)$. Define
$L_W(Y) \stackrel{def}{=} \frac{wr(Y,b_1, \ldots, b_t)}{wr(b_1, \ldots , b_t)}$. Using the above, we have that the coefficients of $L_W$ are left fixed by the differential Galois group and so lie in $k$. Therefore $ L_W\in k[\dd]$.  Furthermore,  Corollary~\ref{wronskiancor} implies that $L_W$ vanishes on $W$.  If we write $L = L_1\cdot L_W + R$ where $\ord(R) < \ord(L_W)$, one sees that $R$ also vanishes on $W$.  Since the dimension of $W = \ord(L_W)$, Corollary~\ref{wronskiancor} again implies that $R= 0$.  Therefore $W$ is the solution space of the right factor $L_W$ of $L$.\\[0.05in]
Let $L_0 \in k[\dd]$ with $L = L_1\cdot L_0 $. Let $W = Soln_K(L_0(y) = 0)$. We have that $ L_0$ maps $V$ into  $Soln_K(L_1(y) = 0)$. Furthermore we have the following implications
\begin{itemize}
\item $Im(L_0) \subset  Soln_K(L_1(y) = 0) \Rightarrow \dim_{C_k}(Im(L_0))\leq \ord (L_1)$
\item $Ker(L_0) = Soln_K(L_0(y) = 0) \Rightarrow \dim_{C_k}(Ker(L_0))\leq \ord (L_0)$
\item $n = \dim_{C_k}(V)=\dim_{C_k}(Im(L_0)) + \dim_{C_k}(Ker(L_0)) \leq $\\$ \hspace*{2in} \ord(L_1) + \ord(L_0) =n$
\end{itemize}
Therefore $dim_{C_k}Soln_K(L_0(y)=0) = \ord (L_0)$.  This now implies that the association of $G$-invariant subspaces of $V$ to right factors of $L$ is a bijection. \\[0.05in]
I  will now describe the correspondence between right factors of $L$ and submodules of $k[\dd]/k[\dd]L$. One easily checks that if $L = L_1 L_0$, then $k[\dd]L  \subset k[\dd]L_0$ so $M = k[\dd]L_0/ k[\dd]L\subset k[\dd]/k[\dd]L$.  Conversely let $M$ be a differential submodule of $k[\dd]/k[\dd]L$. Let $\pi: k[\dd] \rightarrow k[\dd]/k[\dd]L$ and $\overline{M} = \pi^{-1}(M)$.  $\overline{M} = k[\dd]L_0$ for some $L_0 \in k[\dd]$.  Since $L \in \overline{M}, L = L_1L_0$. \end{proof}
\noindent 
By induction on the order, one can show that any differential operator can be written as a product of irreducible operators.  This need not be unique (for example $\dd^2 = \dd \cdot \dd = (\dd +\frac{1}{x})(\dd - \frac{1}{x}))$ but a version of the Jordan-H\"older Theorem implies that if $L \ L_1 \cdot \ldots \cdot L_r = \tilde{L}_1\cdot\ldots \cdot \tilde{L}_s$ the $r=s$ and, after possibly renumbering, $k[\dd]/k[\dd]L_i \simeq k[\dd]/k[\dd]\tilde{L}_i$ (\cite{PuSi2003}, Chapter 2).\\[0.1in]
One can also characterize submodules of differential modules in terms of corresponding  matrix linear differential equations.  Let $Y' = AY$ correspond to the differential module $M$ of dimension $n$. One then has that $M$ has a proper submodule of dimension $t$ if and only if $Y' = AY$ is equivalent to $Y' = BY, B = \left(\begin{array}{cc} B_0 & B_1 \\ 0 & B_2\end{array} \right) \ B_0 \in {\rm gl}_t(k)$.\\[0.2in]
\underline{\em Transcendental numbers.}  Results  of Siegel and Shidlovski (and subsequent authors) reduce the problem of showing that the values of certain functions are algebraically independent to showing that these functions are themselves algebraically independent.  The Galois theory of linear differential equations can be used to do this.
\begin{defin}  $f = \sum^\infty_{i=0} a_n\frac{z^n}{n!} \in \Qx[[z]]$ is an {\em E-function} if $\exists \ c \in \Qx$ such that, for all $n$, $|a_n| \leq c^n$ and the least  common denominator of $a_0, \ldots , a_n \leq c^n$.
\end{defin} 
\begin{exs}$e^x = \sum \frac{z^n}{n!}$, \ \ \ \ \ $J_0(z) = \sum \frac{z^{2n}}{(n!)^2}$
\end{exs}
\begin{thm} (Siegel-Shidlovski, {\em c.f.,} \cite{Lang_trans}): If $f_1, \ldots  ,f_s$ are E-functions, \underline{\em algebraically independent} over $\Qx(z)$ and satisfy a linear differential equation
\[L(y) = y^{(m)} + a_{m-1}y^{(n-1)} + \ldots + a_0y , \ \ \  a_i \in \Qx(z)\]
then for $\alpha \in \Qx, \alpha \neq 0$ and $\alpha $ not a pole of the $a_i$, the numbers $f_1(\alpha), \ldots , f_s(\alpha)$ are  algebraically independent over $\Qx$.\end{thm}

The differential Galois group of $L(y) = 0$ measures the algebraic relations among solutions.  In particular, Corollary~\ref{torsorcor} implies that if $y_1, \ldots y_n$ are a $\bar{\Qx}$-basis of the solution space of $L(y) = 0$ then the transcendence degree of $\bar{\Qx}(x)(y_1, \ldots y_n, y_1', \ldots y_n', \ldots , y_1^{(n-1)}, \ldots , y_n^{(n-1)})$ over $\bar{\Qx}(x)$ is equal to the dimension of the differential Galois group of $L(y) = 0$.
\begin{ex} The differential Galois group of the Bessel equations $y''+\frac{1}{z} y' + (1 - \frac{\lambda^2}{z^2})y = 0, \ \lambda - \frac{1}{2} \notin \Zx$ over $\bar{\Qx}(x)$ can be shown to be $\SL_2$ (\cite{kolchin_ostrowski}) and examining the recurrence defining a power series solution of this equation, one sees that the solutions are E-functions.   Given a  fundamental set of solutions $\{J_\lambda, Y_\lambda\}$ we have that the transcendence degree of $\bar{\Qx}(x)( J_\lambda, Y_\lambda, J_\lambda', Y_\lambda')$ over $\bar{\Qx}(x)$ is $3$ (one can show that $ J_\lambda Y_\lambda' -  J_\lambda', Y_\lambda \in \bar{\Qx}(x)$.  Therefore one can apply the Siegel-Shidlovski Theorem to show that certain values of $J_\lambda, Y_\lambda$ and $ J_\lambda'$ are algebraically independent over $\bar{\Qx}$.
\end{ex}
\vspace{.1in}
\noindent \underline{\em Why does $\cos x$ satisfy a linear differential equation over $\Cx(x)$, while} \linebreak \underline{\em  $\sec x$ does not?} There are many reasons for this.  For example, $\sec x = \frac{1}{cos x}$ has an infinite number of poles which is impossible for a solution of a linear differential equation over $\Cx(x)$. The following result, originally due to Harris and Sibuya \cite{harris_sibuya}, also yields this result and can be proven using   differential Galois theory \cite{singer_relations}.
\begin{prop}Let  $k$ be a differential  field with algebraically closed constants and $K$ a PV -field of $k$. Assume $y\neq 0,1/y \in K$  satisfy linear differential equations over $k$.  The $\frac{y'}{y}$ is algebraic over $k$.
\end{prop}
\begin{proof} (Outline) We may assume that $k$ is algebraically closed and that $K$ is the quotient field of a PV-ring $R$.  Corollary~\ref{torsorcor} implies that $R = k[G]$ where $G$ is an algebraic group whose $C_k$ points form  the Galois group of $R$ over $k$.  Since both  $y$ and $\frac{1}{y}$ satisfy linear differential equations over $k$, their orbits under the action of the Galois group span a {\em finite dimensional} vector space over $C_k$.  This in turn implies the orbit of that $y$ and $\frac{1}{y}$  induced by the the action of $G(k)$ on $G(k)$ by right multiplication  is finite dimensional.  The general theory of linear algebraic groups \cite{springer} tells us that this implies that   these elements must lie in $k[G]$. A result of Rosenlicht  \cite{magid_finite, toroid} implies that $y = a f$ where $a \in k$ and $f \in k[G]$ satisfies $f(gh) = f(g)f(h)$ for all $g, h \in G$.
In particular, for $\sigma \in G(C_k) = \Aut(K/k)$, we have $\sigma(f(x)) = f(x\sigma) = f(x)f(\sigma)$.  If $L \in k[\dd]$ is of minimal positive order such that $L(f) = 0$,  then $0 = \sigma(L(f)) = L(f\cdot f(\sigma))$ and a calculation shows that this implies that $f(\sigma) \in C_k$. \\[0.05in]
We therefore have that, for all $\sigma \in \Aut(K/k)$ there exists a $c_\sigma \in C_k$ such that  $\sigma(y) = c_\sigma y$.  This implies that $\frac{y'}{y}$ is left fixed by $\Aut(K/k)$ and so must be in $k$.  \end{proof}
If $y = \cos x$ and $\frac{1}{y} = \sec x$ both satisfy linear differential equations over $\bar{\Qx}(x)$, then the above result implies that $\tan x$ would be algebraic over $\bar{\Qx}(x)$.  One can show that this is not the case by noting that  a periodic function that is algebraic over $\bar{\Qx}(x)$ must be constant, a contradiction. Another proof can be given by startng with a putative minimal polynomial $P(Y) \in \Cx(x)[Y]$ and noting that $\frac{dP}{dx} + \frac{dP}{dY} Y' = \frac{dP}{dx} + \frac{dP}{dY} (Y^2+1)$ must be divisible by $P$.  Comparing coefficients, one sees that the coefficients of powers of $Y$ must be constant and so $\tan x$ would be constant.\\[0.2in]
\underline{\em Systems of linear partial differential equations.} Let $(k, \Delta = \{\dd_1, \ldots , \dd_m\})$ be a  partial differential field. One can define the ring of linear partial differential operators $k[\dd_1, \ldots , \dd_m]$ as the noncommutative polynomial ring in $\dd_1, \ldots , \dd_m$ where $\dd_i\cdot a = \dd_i(a) + a \dd_i$ for all $a \in k$. In this context, a differential module is a finite dimensional $k$-space that is a left $k[\dd_1, \ldots , \dd_m]$-module.  We have 
\begin{prop} Assume $k$ contains an element $a$ such that $\dd_i(a) \neq 0$ for some $i$.  There is a correspondence between \begin{itemize}
\item  differential modules,
\item  left ideals $I \subset k[\dd_1, \ldots , \dd_m]$ such that $\dim_k k[\dd_1, \ldots , \dd_m]/I < \infty$, and 
\item systems $\dd_i Y = A_iY, A\in \gl_n(k), i = 1, \ldots m $ such that  
\[ \dd_i(A_j) +A_iA_j = \dd_j(A_i) + A_jA_i\] i.e., integrable systems
\end{itemize}
\end{prop}
If \[\dd_i Y = A_iY, A\in \Mn(k), i = 1, \ldots m\] is an integrable system, the solution space over any differential field is finite dimensional over constants.  Furthermore, one can define   PV extensions, differential Galois groups, etc. for such systems 
and develop a Galois theory whose groups are linear algebraic groups. For further details see \cite{DAAG, CaSi} and Appendix D of \cite{PuSi2003}.

\section{Local Galois Theory}\label{mfssec3} In the previous section, we showed that for equations with only regular singular points the analytically defined monodromy group determines the differential Galois group.  In this section we shall give an idea of what analytic information determines the Galois group for general singular points. We shall use the following notation:
\begin{itemize}
\item
${ \Cx[[x]]} = \{\sum_{i=0}^\infty a_i x^i \ | \ a_i \in \Cx\}, \ \ x'=1$
\item
${ \Cx((x))}$ = quotient field of $\Cx[[x]]$ = $\Cx[[x]][x^{-1}]$
\item${\Cx\{\{x\}\} } = \{f \in \Cx[[x]] \ | \ f \mbox{ converges in a neighborhood $\calO_f$ of $0$}\}$
\item 
${\Cx(\{x\}) }$ = quotient field of $\Cx\{\{x\}\} $ = $\Cx\{\{x\}\}[x^{-1}]$
\end{itemize}
Note that \[ \Cx(x) \subset \Cx(\{x\}) \subset \Cx((x))\]
We shall consider equations of the form
\begin{eqnarray} \label{eqnlocal1}
Y' = AY 
\end{eqnarray}
with  $ A \in \Mn(\Cx(\{x\})) \mbox{ or }\Mn(\Cx((x)))$
and discuss the form of solutions of such equations as well as the Galois groups relative to the base fields $\Cx(\{x\})$ and $ \Cx((x))$.  We start with the simplest case:\\[0.2in]
{\bf \underline{Ordinary points.}} We say that $x = 0$ is an {\em ordinary point} of equation~(\ref {eqnlocal1}) if $A \in \Mn(\Cx[[x]])$. The standard existence theory (see \cite{poole} for a classical introduction or \cite{hartman_book}) yields the following:
\begin{prop} If $A \in \Mn(\Cx[[x]])$ then there exists a unique $Y \in \GL_n(\Cx[[x]]$ such that $Y' = AY$ and $Y(0)$ is the identity matrix.  If $A \in \Mn(\Cx(\{x\}))$ then we furthermore have that $Y \in \GL_n(\Cx(\{x\}))$.
\end {prop}
This implies that for ordinary points with $A \in \Mn(k)$ where $k = \Cx(\{x\})$ or $ \Cx((x))$, the Galois group of the associated Picard-Vessiot extension over $k$ is trivial.   The next case to consider is \\[0.2in]
{\bf{\underline{Regular singular points.}} } From now on we shall consider equations of the the form~(\ref{eqnlocal1}) with $A \in \Cx(\{x\})$ (although some of what we say can be adapted to the case when $A \in \Cx((x))$). We defined the notion of a regular singular point in Definition~\ref{prelimreg}. The following proposition gives equivalent conditions (see Chapter 5.1.1 of \cite{PuSi2003} for a proof of the equivalences).
\begin{prop}Assume $A \in \Mn(\Cx(\{x\}))$. The following are equivalent:
\begin{enumerate}
\item $Y'=AY$ is equivalent over $\Cx((x))$ to $Y' = \frac{A_0}{x}Y, A_0 \in \Mn(\Cx)$.
\item $Y'=AY$ is equivalent over $\Cx(\{x\})$ to $Y' = \frac{A_0}{x}Y, A_0 \in \Mn(\Cx)$.
\item For any small sector $\calS(a,b,\rho), \ |a-b|<2\pi,$ at $0$, $\exists$ a fundamental solution matrix $Z=(z_{ij})$, analytic in $\calS$ such that $|z_{ij}| < |x|^N$ for some integer $N$ as $x\rightarrow 0$ (regular growth of solutions).
\item  $Y' = AY$ has a fundamental solution matrix   of the form $Y = x^{A_0}Z, A_0 \in \Mn(\Cx), Z\in \GL_n(\Cx(\{x\}))$.
\item  $Y'=AY$ is equivalent to $y^{(n)} + a_{n-1}y^{(n-1)}+ \ldots +a_0y$, $a_i \in \Cx(\{x\})$, where $ord_{x=0}a_i\geq i-n$.
\item $Y'=AY$ is equivalent to $y^{(n)} + a_{n-1}y^{(n-1)}+ \ldots +a_0y$, $a_i \in \Cx(\{x\})$, having $n$ linearly independent solutions of the form  \linebreak $x^{\alpha}\sum_{i=0}^rc_i(x)(\log x)^i, \alpha\in \Cx, c_i(x) \in \Cx\{\{x\}\}$.
\end{enumerate}
\end{prop}

We now turn to the Galois groups.  We first consider the Galois group of a regular singular equation over $\Cx(\{x\})$. Let $\calO$ be a connected open neighborhood of $0$ such that the entries of $A$ are regular in $\calO$ with at worst a pole at $0$. Let $x_0 \neq 0$ be a point of $\calO$ and $\gamma$ a simple loop around $0$ in $\calO$ based at $x_0$.  There  exists a fundamental solution matrix $Z$ at $x_0$ and the field $K= k(Z,\frac{1}{\det Z})$, $k = \Cx(\{x\})$ is a Picard-Vessiot extension of $k$.  Analytic continuation around $\gamma$ gives a new fundamental solution matrix $Z_\gamma = ZM_\gamma$ for some $M_\gamma \in \GL_n(\Cx)$.  The matrix $M_\gamma$ depends on the choice of $Z$ and so is determined only up to conjugation.  Nonetheless, it is referred to as the {\em monodromy matrix at $0$}.  As in our discussion of monodromy in Section~\ref{mfssec4},  we see that $M_\gamma$ is in the differential Galois group $\DGal(K/k)$  and, arguing as in Proposition~\ref{schlesinger}, one sees that $M_\gamma$ generates a cyclic group that is Zariski dense in $\DGal(K/k)$. From the theory of linear algebraic groups, one can show that a linear algebraic group have a Zariski dense cyclic subgroup must be of the form $(\Cx^*,\times)^r \times (\Cx,+)^s \times \Zx/m\Zx, \ r\geq 0 , \ s\in \{0,1\}$ and that any such group occurs as a differential Galois group of a regular singular equation over $k$. (Exercise 5.3 of \cite{PuSi2003}).  This is not too surprising since property (vi) of the above proposition shows that $K$ is contained in a field of the form $k(x^{\alpha_1}, \ldots , x^{\alpha_r}, \log x)$ for some $\alpha_i \in \Cx$.\\[0.1in]
Since $\Cx(\{x\})\subset \Cx((x)) = \tilde{k}$, we can consider the the field $\tilde{K} = \tilde{k}( Z,\frac{1}{\det Z})$ where $Z$ is as above.  Again, using (vi) above, we have that $\tilde{K} \subset F=\tilde{k}(x^{\alpha_1}, \ldots , x^{\alpha_r}, \log x)$ for some $\alpha_i \in \Cx$.  Since we are no longer dealing with analytic functions, we cannot appeal to analytic continuation but we can define a {\em formal} monodromy, that is a differential automorphism of $F$ over $\tilde{k}$ that sends $x^\alpha\mapsto x^\alpha e^{2\pi i \alpha}$ and $\log x \mapsto \log x + 2\pi i$.  One can show that the only elements of $F$ that are left fixed by this automorphism lie in $\tilde{k}$.  This automorphism restricts to an automorphism of $\tilde{K}$. Therefore the Fundamental Theorem implies that the group generated by the formal monodromy is Zariski dense in the differential Galois group $\DGal(\tilde{K}/\tilde{k})$.  The restriction of the formal monodromy to $K$ (as above) is just the usual monodromy, so one sees that  the {\em convergent Galois group} $\DGal(K/k)$ equals the {\em formal Galois group} $\DGal(\tilde{K}/\tilde{k})$. \\[0.1in]
We note that the existence of $n$ linearly independent solutions of the form described in (vi) goes back to Fuchs in the $19^{th}$ Century and can be constructed using the Frobenius method (see \cite{gray} for an historical account of the development of linear differential equations in the $19^{th}$ Century).\\[0.2in]
{\bf{\underline{Irregular singular points.}} } We say $x=0$ is an irregular singular point if it is not ordinary or regular. Fabry, in the $19^{th}$ Century constructed a fundamental set of formal solutions of a scalar linear differential equation at an irregular singular point.  In terms of linear systems, every equation of the form $Y' = AY, \ A \in \Mn(\Cx((x)))$ has a fundamental solution matrix of the form
\[ \hat{Y} = \hat{\phi}(t) x^L e^{Q(1/t)}\]
where \begin{itemize}
\item $t^\nu = x, \ \nu \in \{1,2,3,\ldots, n! \}$
\item $\hat{\phi}(t) \in \GL_n(\Cx((t)))$
\item $L\in \Mn(\Cx)$
\item $Q = {\rm diag}(q_1, \ldots , q_n), q_i \in \frac{1}{t}\Cx[\frac{1}{t}]$
\item $L$ and $Q$ commute.
\end{itemize}
One refers to the integer $\nu$ as the {\em ramification} at $0$ and the polynomial $q_i$ as the {\em eigenvalues} at $0$ of the equation.  We note that even if $A \in \Mn(\Cx(\{x\}))$, the matrix $\hat{\phi}(t) $ does not necessarily lie in $\Mn(\Cx(\{t\}))$.  There are various ways to make this representation unique but we will not deal with this here.   We also note that by adjusting the term $x^L$, we can always assume that no ramification occurs in $\hat{\phi}(t) $, that is, $\hat{\phi} \in \GL_n(\Cx((x)))$ (but one may lose the last commutativity property above).     As in the regular case, there are algorithms to determine a fundamental matrix of this form.  We refer to Chapter 3 of \cite{PuSi2003} for a detailed discussion of the existence of these formal solutions and algorithms to calculate them and  p.~98 of \cite{PuSi2003} for historical references and references to some recent work. \\[0.1in]
I will now  give two examples (taken from \cite{loday_richaud95}).
\begin{ex}\label{eulerex}(Euler Eqn) ${ x^2y' + y = x}$ \\[0.05in]
This inhomogeneous equation has a solution $\hat{y} = \sum_{n=0}^\infty (-1)^n n! x^{n+1}$. By applying the operator $x\frac{d}{dx} - 1$ to both sides of this equation we get a second order homogeneous equation $x^3 y^{\prime\prime} + (x^2+x)y^{\prime} - y = 0$ which, in matrix form becomes

\[\frac{dY}{dx}= \left(\begin{array}{cc} 0&1\\ \frac{1}{x^3}&-(\frac{1}{x}+\frac{1}{x^2})\end{array}\right)Y, \ \ \ Y = \left(\begin{array}{c} y\\y'\end{array}\right)\]
This matrix equation has a fundamental solution matrix of the form  $\hat{Y} = \hat{\phi}(x) e^Q$, where \\[0.15in]
$Q= \left(\begin{array}{cc} \frac{1}{x}&0\\ 0&0\end{array}\right) \ \ \hat{\phi}(x) = \left(\begin{array}{cc} 1&\hat{f}\\ -\frac{1}{x^2}&\hat{f}'\end{array}\right) \ \ \ \hat{f} = \sum_{n=0}^\infty (-1)^n n!x^{n+1}$\\ There is no ramification here (i.e., $\nu = 1$).
\end{ex}
\begin{ex} (Airy Equation) \label{airyex}${y^{\prime \prime} = xy}$\\[0.05in] The singular point  of this equation is at $\infty$.  If we let 
$z =\frac{1}{x}$ the equation becomes $ z^5 y^{{\prime}{\prime}} + 2z^4 y^{\prime} - y = 0$ or in matrix form
\[ Y' = \left(\begin{array}{cc} 0&1\\ \frac{1}{z^5}&-\frac{2}{z}\end{array}\right) Y\]
This equation has a fundamental matrix of the form  $\hat{Y} = \hat{\phi}(z) z^JU e^{Q(t)}$ where
\[\hat{\phi}(z) = \ldots , \ \ U = \left(\begin{array}{cc} 1&1\\ 1&-1\end{array}\right) , \ \ J = \left(\begin{array}{cc} \frac{1}{4}&0\\ 0&-\frac{3}{4}\end{array}\right)\] 
\[ t = z^{\frac{1}{2}} , \ \ \ Q = \left(\begin{array}{cc} -\frac{2}{3t^3}&0\\ 0&\frac{2}{3t^3}\end{array}\right)\]
The ramification is $\nu = 2$. The precise form of $\hat{\phi}(z)$ is not important at the moment but we do note that it lies in $\GL_n(\Cx((z)))$ (see \cite{loday_richaud95} for the precise form).
\end{ex}
We now turn to the Galois groups over $\Cx((x))$ and $\Cx(\{x\})$.\\[0.2in]
\underline{\em Galois Group over $\Cx((x))$: the Formal Galois Group at $x=0$}. We consider the differential equation $Y' = AY,  A \in  \Mn(\Cx((x)))$. To simplify the presentation while capturing many of the essential ideas we shall assume that there is no ramification in $\hat{\phi}$, that is, the  equation has a fundamental solution matrix $\hat{Y} = \hat{\phi}(x) x^L e^{Q(1/t)}$, with $\underline{\hat{\phi}(x)\in \Mn(\Cx((x)))}$. The map $Y_0 \mapsto \hat{\phi}(x) Y_0$ defines an equivalence \underline{over $\Cx((x)) $} between the equations 
\[ Y_0' = A_0Y_0 \mbox{  and  } Y' = AY\]
where $A_0 = \hat{\phi}^{-1}A\hat{\phi} - \hat{\phi}^{-1} \hat{\phi}'$.
\begin{defin}The equation $Y' = A_0Y$ is called a {\em Normal Form} of $Y' = AY$. It has $Y_0 = x^Le^{Q(1/x)}$ as a fundamental solution matrix.
\end{defin}

\begin{ex}(Euler Equation)\label{eulerex2} For this equation (Example~\ref{eulerex}), the normal form is $\frac{dY_0}{dx}= A_0Y_0, \ A_0 = \left(\begin{array}{cc} -\frac{1}{x^2} &0\\ 0&0\end{array}\right) \ {Y_0} =  e^{\left(\begin{array}{cc} \frac{1}{x} &0\\ 0&0\end{array}\right)}$\\[0.15in]
\end{ex}
\begin{ex} (Airy Equation)\label{airyex2} Using the solution in Example~\ref{airyex}, we get a normal form $\frac{dY_0}{dx}= A_0Y_0$ with solution 
$Y_0 = x^LUe^{Q(1/t)}$ with $L,Q,t$ as in Example~\ref{airyex}.
\end{ex}
The map $Y_0 \mapsto \hat{\phi}(x) Y_0$ defines an isomorphism between the  solutions spaces of   $Y_0' = A_0Y_0 \mbox{  and  } Y' = AY$ and a calculation shows that   $F=\hat{\phi}(x)$ satisfies the differential equation \[\frac{dF}{dx} = AF-FA_0.\]
Since the normal form is equivalent over $\Cx((x))$ to the original equation, its differential Galois group over $\Cx((x))$ is the same as that of the original equation.\\[0.1in]
Let $K$ be the Picard-Vessiot extension of $\Cx((x))$ for $\frac{dY_0}{dx}= A_0Y_0$ (and therefore also for $\frac{dY}{dx}= AY$).  We have that \[K \subset E = \Cx((x))(x^\frac{1}{\nu}, x^{\alpha_1},\ldots , x^{\alpha_n},\log x, e^{q_1(x^{1/\nu})}, \ldots , e^{q_n(x^{1/\nu)})}).\]
We shall describe some differential automorphisms of $E$ over $\Cx((x))$:\\[0.1in]
(1) \underline{Formal Monodromy $\gamma$:}  This is a differential  automorphism of $E$ given by the following conditions
\begin{itemize}
\item $\gamma(x^r) = e^{2\pi ir} x^r$
\item $\gamma(\log x) = \log x + 2\pi i$
\item $\gamma(e^{q(x^{1/\nu})}) = e^{q(\zeta x^{1/\nu})}, \ \zeta=e^\frac{2\pi i}{\nu}$
\end{itemize}
(2) \underline{Exponential Torus $\calT$:} This is the group of differential automorphisms $ \calT = \DGal(E/F)$ where  $E = F(e^{q_1(x^{1/\nu})}, \ldots , e^{q_n(x^{1/\nu)})})$ and \newline
$F = \Cx((x))(x^\frac{1}{\nu}, x^{\alpha_1},\ldots , x^{\alpha_n},\log x)$. Note that for  $\sigma \in \calT$:
\begin{itemize}
\item $\sigma(x^r) =  x^r$
\item $\sigma(\log x) = \log x $
\item $\sigma(e^{q(x^{1/\nu})}) = c_\sigma e^{q( x^{1/\nu})}$
\end{itemize}
so we may identify $\calT$ with a subgroup of $(\Cx^*)^n$.  One can show (Proposition 3.25
of \cite{PuSi2003}) 
\begin{prop} If $f\in E$ is left fixed by $\gamma$ and $\calT$, then $f \in \Cx((x))$. \end{prop}
The formal monodromy and the elements of the exponential torus restrict to differential automorphisms of $K$ over $\Cx((x))$ and are again called by these names.  Therefore the above proposition and the Fundamental Theorem imply 
\begin{thm}\label{formalgalois} The differential Galois group of $Y' = AY$ over $\Cx((x))$ is the Zariski closure of the group generated by the formal monodromy and the exponential torus. \end{thm}
One can furthermore show  (Proposition 3.25
of \cite{PuSi2003}) that $x=0$ is a regular singular point if and only if the exponential torus is trivial.

\begin{ex} (Euler Equation)\label{eulerex3} We continue using the notation of Examples~\ref{eulerex} and ~\ref{eulerex2}. One sees that $K = E = \Cx((x))(e^\frac{1}{x})$ in this case.  Furthermore, the formal monodromy is trivial.  Therefore  $\DGal(K/\Cx((x)))$ is equal to the exponential torus $\calT = \{\left(\begin{array}{cc} c &0\\ 0&1\end{array}\right)  \ | \ c \in \Cx^*\}$
\end{ex}
\begin{ex}\label{airyex3} (Airy Equation) We continue using the notation of Examples~\ref{airyex} and ~\ref{airyex2}. In this case $F = \Cx((z))(z^\frac{1}{4})$ and  $E = F(e^\frac{2}{3z^{3/2}})$.  The formal monodromy $\gamma = \left(\begin{array}{cc} 0&\sqrt{-1}\\ \sqrt{-1}&0\end{array}\right)$ and the
exponential torus $\calT$ = $\{\left(\begin{array}{cc} c&0\\ 0&c^{-1}\end{array}\right) \ | \ c\neq 0 \}$. Therefore 
\[DGal(K/\Cx((x))) = D_\infty = \{\left(\begin{array}{cc} c&0\\ 0&c^{-1}\end{array}\right) \ | \ c\neq 0 \} \cup \{\left(\begin{array}{cc} 0&c\\ -c^{-1}&0\end{array}\right) \ | \ c\neq 0 \}\]
\end{ex}

\vspace{.2in}
\noindent \underline{\em Galois Group over $\Cx(\{x\})$: the Convergent Galois Group at $x=0$}.  We now assume that we are considering a differential equation with $\frac{dY}{dx} = AY$ with \underline{$A \in \Cx(\{x\})$} and we let $K$ be the Picard-Vessiot extension of $\Cx(\{x\})$ for this equation.  Since $\Cx(\{x\}) \subset \Cx((x))$, we have that the formal Galois group  $\DGal(K/\Cx((x)))$ is a subgroup of the convergent Galois group $\DGal(K/\Cx(\{x\}))$. If $x=0$ is a regular singular point then these Galois groups are the same but in general the convergent group is larger than the formal group. I will describe analytic objects that measure the difference.  The basic idea is the following.  We have a formal solution $\hat{Y} = \hat{\phi}(x) x^L e^{Q(1/t)}$ for the differential equation. For each small enough sector $\calS$ at $x=0$ we can give the terms $x^L$ and $e^{Q(1/t)}$ meaning in an obvious way.  The matrix $\hat{\phi}(x)$ is a little more problematic but it has been shown that one can  find  a matrix $\Phi_\calS$ of functions, analytic is $\calS$, ``asymptotic'' to $\hat{\phi}(x)$ and such that $\Phi_\calS(x) x^L e^{Q(1/t)}$ satisfies the differential equation.  In overlapping sectors $\calS_1$ and $\calS_2$ the corresponding $\Phi_{\calS_1}$ and $\Phi_{\calS_2}$ may not agree but they will satisfy $\Phi_{\calS_1} = \Phi_{\calS_2} St_{12}$ for some $St_{12} \in \GL_n(\Cx)$.  The $St_{i,j}$ ranging over a suitable choice of small sectors $\calS_\ell$ together with the formal Galois group should generate a group that is Zariski dense in the convergent Galois group.  There are two (related) problems with this idea.  The first is what do we mean by asymptotic and the second is how do we select the sectors and the $\Phi_\ell$ in some canonical way?  We shall first consider these questions for fairly general formal series and then specialize to the situation where the series are entries of matrices $\hat{\phi}(x)$ that arise in the formal solutions of linear differential equations.\\[0.1in]
\noindent\underline{\em Asymptotics and Summability.} In the $19^{th}$ century, Poincar\'e proposed the following notion of an asymptotic expansion.
\begin{defin} Let $f$ be analytic on a sector $\calS=\calS(a,b,\rho)$ at $x=0$. The series $\sum_{n\geq n_0} c_nx^n$ is an {\em  asymptotic expansion} of $f$ if $\forall N \geq 0$ and every closed sector $W\subset \calS$  there exists a constant $ C(N,W)$ such that 
\[|f(x) - \sum_{n_0\leq n \leq N-1} c_n x^n | \leq C(N,W) |x|^N \ \ \forall x\in W.\]
\end{defin}

\noindent A function $f$ analytic on a sector can have at most one asymptotic expansion on that sector and we write $J(f)$ for this series if it exists. We denote by $\calA(\calS(a,b,\rho))$ the set of  functions analytic on  $\calS(a,b,\rho)$ having asymptotic expansions on this sector.  We define $\calA(a,b)$ to be  the direct limit of the $\calA(\calS(a,b,\rho))$ over all $\rho$ (here we say $f_1 \in 
\calA(\calS(a,b,\rho_1))$ and $f_2 \in \calA(\calS(a,b,\rho_2))$ are equivalent if the agree on the intersection of the sectors). If one thinks of $a, b \in S^1$, the unit circle,  then one can see that the sets $\calA(a,b)$ form a sheaf on $S^1$. We have the following facts (whose proofs can be found in Chapters 7.1 and 7.2 of \cite{PuSi2003} or in \cite{balser}).  
\begin{itemize}
\item $\calA(a,b)$ is a differential $\Cx$-algebra.
\item $\calA(S^1) = \Cx(\{x\})$.
\item (Borel-Ritt) For all $(a,b) \neq S^1$, the map $J: \calA(a,b) \rightarrow \Cx((x))$ is surjective.\end{itemize}

Given a power series in $\Cx((x))$ that satisfies a linear differential equation, one would like to say that it is the asymptotic expansion of an analytic function that also satisfies the differential equation.  Such a statement is contained in the next result, originally due to Hukuhara and Turrittin, with generalizations and further developments  concerning  nonlinear equations due to Ramis and Sibuya (see Chapter 7.2 of \cite{PuSi2003} for a proof and references).
\begin{thm} Let $A \in \Mn(\Cx(\{x\})),$ $ w \in (\Cx(\{x\}))^n$ and $\hat{v} \in (\Cx((x)))^n$ such that \[\hat{v}' -A\hat{v} = w \ .\]
For any direction $\theta \in S^1$ there exist $a,b, \ a<\theta < b$ and $v \in (\calA(a,b))^n$ such that $J(v) = \hat{v}$ and \[v'-Av = w\] \end{thm}
In general,  there may be many functions $v$ satisfying the conclusion of the theorem.  To guarantee uniqueness we will refine the notion of asymptotic expansion and require that this be valid on large enough sectors.  The refined notion of asymptotics is given by
\begin{defin} (1) Let $k$ be a  positive real number, $\calS$ an open sector.  A function $f \in \calA(\calS)$ with $J(f) = \sum_{n\geq n_0} c_n x^n$ is {\em Gevrey of order $k$} if for all closed $W\subset \calS, \ \exists A>0, c>0$ such that $\forall N\geq 1,  \ x \in W, \ |x| \leq c$ one has
\[|f(z) - \sum_{n_0\leq n \leq N-1}c_nx^n |\leq A^N(N!)^{\frac{1}{k}}|x|^N\]
(One usually uses $\Gamma(1+\frac{N}{k})$ instead of $(N!)^\frac{1}{k}$ but this makes no difference.)\\[0.1in]
(2) We denote by $\calA_{\frac{1}{k}}(\calS)$  the ring of  all Gevrey functions of order $k$ on $\calS$ and let $\calA_{\frac{1}{k}}^0(\calS)$ = $\{f \in \calA_{\frac{1}{k}}(\calS) \ | \ J(f) = 0\}$ 
\end{defin}
The following is a useful criterion to determine if a function is in $\calA_{\frac{1}{k}}^0(\calS)$.
\begin{lem} $f \in \calA_{\frac{1}{k}}^0(\calS)$ if and only if $ \forall \mbox{ closed } W \subset \calS, \exists \ A, B >0$ such that 
\[|f(x)| \leq A \exp {(-B|x|^{-k})} \ \ \forall x\in W\] 
\end{lem}
\begin{ex} $f(x) = e^{-\frac{1}{x^k}} \in \calA_{\frac{1}{k}}^0(\calS),$ for $ \calS = \calS(-\frac{\pi}{2k},\frac{\pi}{2k})$ but {\em not on a larger sector. This is a hint of what is to come.}
\end{ex}
Corresponding to the notion of a Gevrey function is the notion of a Gevrey series.

\begin{defin} $\sum_{n\geq n_0} c_nx^n$ is a {\em Gevrey series of order $k$} if $\exists A > 0$ such that  $\forall n> 0 \ |c_n|\leq A^n (n!)^\frac{1}{k}.$ \\[0.1in]
(2) ${ \Cx((x))_{\frac{1}{k}}}$= Gevrey series of order $k$.
\end{defin}
Note that if $k < \ell$ then $\Cx((x))_{\frac{1}{\ell}} \subset \Cx((x))_{\frac{1}{k}}$.  Using these definitions one can prove the following key facts:
\begin{itemize}
\item (Improved Borel-Ritt) If $|b-a|<\frac{\pi}{k}$, the map $J: \calA_{\frac{1}{k}}(a,b) \rightarrow\Cx((x))_{\frac{1}{k}}$ is surjective but not injective. 
\item (Watson's Lemma) If $|b-a|>\frac{\pi}{k}$, the map $J: \calA_{\frac{1}{k}}(a,b) \rightarrow\Cx((x))_{\frac{1}{k}}$ is injective but not surjective.
\end{itemize}
Watson's lemma motivates the next definition.
\begin{defin} (1) $\hat{y} \in \Cx((x))$ is {\em $k$-summable in the direction $d \in S^1$} if there exists an $f \in \calA_{\frac{1}{k}}(d-\frac{\alpha}{2}, d+\frac{\alpha}{2})$ with $J(f) = \hat{y}, \ \alpha>\frac{\pi}{2}$.\\[0.1in]
(2) $\hat{y} \in \Cx((x))$ is {\em $k$-summable} if it is $k$-summable in all but a finite number of directions 
\end{defin}

A $k$-summable function has unique analytic ``lifts'' in many large sectors (certainly enough to cover a deleted neighborhood of the origin).  I would like to say that the entries of the matrix $\hat{\phi}(x)$ appearing in the formal solution of a linear differential equation are $k$-summable for some $k$ but regretably this is not exactly true and the situation is more complicated (each is a sum of $k_i$-summable functions for  several possible $k_i$).  Before I describe what does happen, it is useful to have  criteria for deciding if a series $\hat{y} \in \Cx((x))$ is $k$-summable.  Such criteria can be given in terms of Borel and Laplace transforms.

\begin{defin} The {\em Formal Borel Transform $\calB_k$ of order $k$}:
\[ \calB_k : \Cx[[x]] \rightarrow \Cx[[\zeta]]\]
\[\calB_k(\sum c_nx^n) = \sum\frac{c_n}{\Gamma(1+\frac{n}{k})} \zeta^n\]
\end{defin}
\underline{Note:} If $\hat{y} = \sum c_n x^n \in \Cx((x))_{\frac{1}{k}} \Rightarrow \calB(\hat{y})$ has a nonzero radius of convergence.

\begin{ex}\label{borelex} $\calB_1(\sum(-1)^n n! x^{n+1}) = \sum (-1)^n \frac{\zeta^{n+1}}{n+1} = \log (1+\zeta)$
\end{ex}

\begin{defin} The {\em Laplace Transform $\calL_{k, d}$ of order $k$ in the direction $d$}: Let $f$ satisfy $|f(\zeta)| \leq A\exp({B|\zeta|^k})$ along the ray $r_d$ from $0$ to $\infty$ in direction $d$. Then
\[\calL_{k,d}f (x) = \int_{r_d} f(\zeta)\exp(-(\frac{\zeta}{x})^k) d((\frac{\zeta}{x})^k)\]
\end{defin}

\noindent \underline{Note}: $\calL_{k,d}\circ\calB_k(x^n) = x^n$ and $\calL_{k,d}\circ\calB_k(f) = f$ for $f\in \Cx\{\{x\}\}$. 

\begin{ex} \label{laplaceex} For any ray $r_d$ not equal to the negative real axis, the function $f(\zeta) = \log(1+\zeta)$ has an analytic continuation  along $r_d$.  Furthermore, for such a ray, we have
\[\calL_{1,d}(\log(1+\zeta))(x) =  \int_d\log(1+\zeta)e^{-\frac{\zeta}{x}}d(\frac{\zeta}{x}) = \int_d\frac{1}{1+\zeta} e^{-\frac{\zeta}{x}}d\zeta\]
\end{ex}

\noindent  Note that there are several slightly different definitions of the Borel and Laplace transforms scattered in the literature but for our purposes these differences do not affect the final results.  I am following the definitions in \cite{balser} and \cite{PuSi2003}. The following gives useful criteria for a Gevrey series of order $k$ to be $k$-summable (Chapter 3.1 of \cite{balser} or Chapter 7.6 of \cite{PuSi2003}).
\begin{prop} $\hat{y} \in \Cx[[x]]_\frac{1}{k}$ is $k$-summable in direction $d$\\[0.05in] \hspace*{1.5in} $\Updownarrow$\\[0.05in]
$\calB_k(\hat{y})$ has an analytic continuation $h$ in a full sector $\{\zeta \ | 0<|\zeta| < \infty, \ |\arg \zeta - d|<\epsilon\}$ and has exponential growth of order $\leq k$ at $\infty$ in this sector ($|h(\zeta)| \leq A\exp (B|\zeta |^k$).  In this case, $f = \calL_{k,d}(h)$ is its $k$-sum.
\end{prop}

\begin{ex}\label{eulerex4} Consider 
 $\hat{y} = \sum(-1)^nn!x^{n+1}$, a power series of   Gevrey order 1 that is also a solution of the Euler equation.   We have seen that $f(\zeta) = \calB_1(\hat{y}) = \sum (-1)^n \frac{\zeta^{n+1}}{n+1} = \log(1+\zeta) \mbox{ for } |\zeta|<1$.  As we have noted in Example~\ref{laplaceex}, this function can be continued outside the unit disk in  any direction other than along the negative real axis.  On any ray except this one, it has exponential growth of order at most 1 and its Laplace transform is
 \[y_d(x) = \calL_{1,d}(\log(1+\zeta))(x) = \int_d\log(1+\zeta)e^{-\frac{\zeta}{x}}d(\frac{\zeta}{x}) = \int_d\frac{1}{1+\zeta} e^{-\frac{\zeta}{x}}d\zeta . \]
We note that this function again satisfies the Euler equation.  To see this note that $y(x) = Ce^{\frac{1}{x}}$ is a solution of $x^2y'+y=0$ and so variation of constants gives us that $f(x) = \int_0^x  e^{\frac{1}{x} - \frac{1}{t}} \frac{dt}{t}$ is a solution of $x^2y' + y = x$.  for convenience let $d=0$ and define  a new variable $\zeta$ by $\frac{\zeta}{x} = \frac{1}{t} - \frac{1}{x}$ gives $f(x) = \int_{d=0}\frac{1}{1+\zeta} e^{-\frac{\zeta}{x}}d\zeta$. 
Therefore, for any ray except the negative real axis, we are able to canonically associate a function, analytic in a large sector around this ray, with the divergent series $\sum(-1)^nn!x^{n+1}$ and so this series is $1$-summable. Furthermore these functions will again satisfy the Euler equation. This is again a hint for what is to come.  \end{ex}

\vspace{.1in}
\noindent \underline{\em Linear Differential Equations and Summability.} It has been known since the early $20^{th}$ Century that a formal series solution of an analytic differential equation $F(x, y,y', \ldots, y^{(n)}) = 0$ must Gevrey of order $k$ for some $k$ (\cite{maillet}).  Regretably, even for linear equations, these formal solutions need not be $k$ summable (see \cite{loday_richaud95}). Nonetheless, if $\hat{y}$ is a vector of formal series that satisfy a linear differential equation it can be written as a sum $\hat{y} = \hat{y_1} + \ldots + \hat{y_t}$ where each $\hat{y_i}$ is $k_i$ summable for some $k_i$.  This result is the culmination of the work of several authors including Ramis, Malgrange,  Martinet, Sibuya and \'Ecalle (see \cite{loday_richaud95}, \cite{ramis93} and Chapter 7.8 of \cite{PuSi2003}). The $k_i$ that occur as well as the directions along which the $\hat{y_i}$ are not summable can be seen from the normal form of the differential equation.  We need the following definitions.  Let $Y' = AY, \ A\in \Mn(\Cx(\{x\}))$ and let $\hat{Y} = \hat{\phi}(x) x^L e^{Q(1/t)}$ be a formal solution with $t^\nu =x$ and $Q = {\rm diag}(q_1, \ldots , q_n), q_i \in \frac{1}{t}\Cx[\frac{1}{t}]$.

\begin{defin} (1) The rational number $k_i = \frac{1}{\nu}(\mbox{the degree of }q_i \mbox{ in } \frac{1}{t})$ is called a {\em slope} of $Y' = AY$.\\[0.05in]
(2)  A {\em Stokes direction $d$} for a $q_i = cx^{-k_i}+ \ldots$ is a direction such that $Re(cx^{-k_i}) = 0$, i.e., where $e^{q_i}$ changes behavior.\\[0.05in]
(3) Let $d_1, d_2$ be consecutive Stokes directions.  We say that $(d_1, d_2)$ is a {\em negative  Stokes pair} if $Re(cx^{-k_i}) <0$ for $\arg(x) \in (d_1,d_2)$.\\[0.05in]
(4) A {\em  singular direction} is the bisector of a negative Stokes pair.
\end{defin}
\begin{ex} For the Euler equation we have that $\nu = 1$ and $Q(x) = diag(\frac{1}{x},0)$.  Therefore the slopes are $\{0, 1\}$.  The Stokes directions are $\{\frac{\pi}{2}, \frac{3\pi}{2}\}$.  The pair $(\frac{\pi}{2}, \frac{3\pi}{2})$ is a negative Stokes pair and so $\pi$ is the only singular direction. \end{ex}
With these definitions we can state
\begin{thm} (Multisummation Theorem) Let $1/2<k_1<\ldots <k_r$ be the slopes of $Y'=AY$ and let $\hat{Y} \in  (\Cx((x)))^n$ satisfy this differential equation.  For all directions $d$ that are not singular directions, there exist $\hat{Y}_i \in \Cx((x))$ such that $\hat{Y} = \hat{Y}_1+\ldots+\hat{Y}_r$, each $\hat{Y}_i$ is $k_i$-summable in direction $d$ and, for $y_i $= the $k_i$-sum of $\hat{Y}_i $, $y_d = y_1+\ldots+y_r$ satisfies $y'=Ay$.
\end{thm}
The condition that $1/2 <k_1$ is a needed technical assumption but, after a suitable change of coordinates, one can always reduce to this case.
We shall refer to the element $y_d$ as above as the {\em multisum of $\hat{Y}$ in the direction $d$}.

\vspace{.1in}

\noindent \underline{\em Stokes Matrices and the Convergent Differential Galois Group}. We are now in a position to describe the analytic elements which, together with the formal Galois group determine the analytic Galois group. Once again, let us assume that we have a differential equation $Y' = AY$ with $A \in \Mn(\Cx(\{x\}))$ and assume this has a formal solution $\hat{Y} = \hat{\phi}(x) x^L e^{Q(1/t)}$ with $\hat{\phi}(x)\in \GL_n(\Cx((x)))$.  The matrix $F=\hat{\phi}(x)$ is a formal solution of the $n^2\times n^2$ system 
\begin{eqnarray}\label{stokeseqn}\frac{dF}{dx} = AF-FA_0
\end{eqnarray}
where $Y' = A_0Y$ is the normal form of $Y' = AY$.
We will want to associate functions, analytic in sectors, to $\hat{\phi}(x)$ and therefore will  apply the Multisummation Theorem to \underline{equation~\ref{stokeseqn}} (and not to $Y'=AY$). One can show that the eigenvalues of equation~(\ref{stokeseqn}) are of the form $q_i-q_j$ where the $q_i,q_j$ are eigenvalues of $Y' = AY$.\\[0.1in]
\parbox{1.3in}{\includegraphics[width =.4\textwidth]{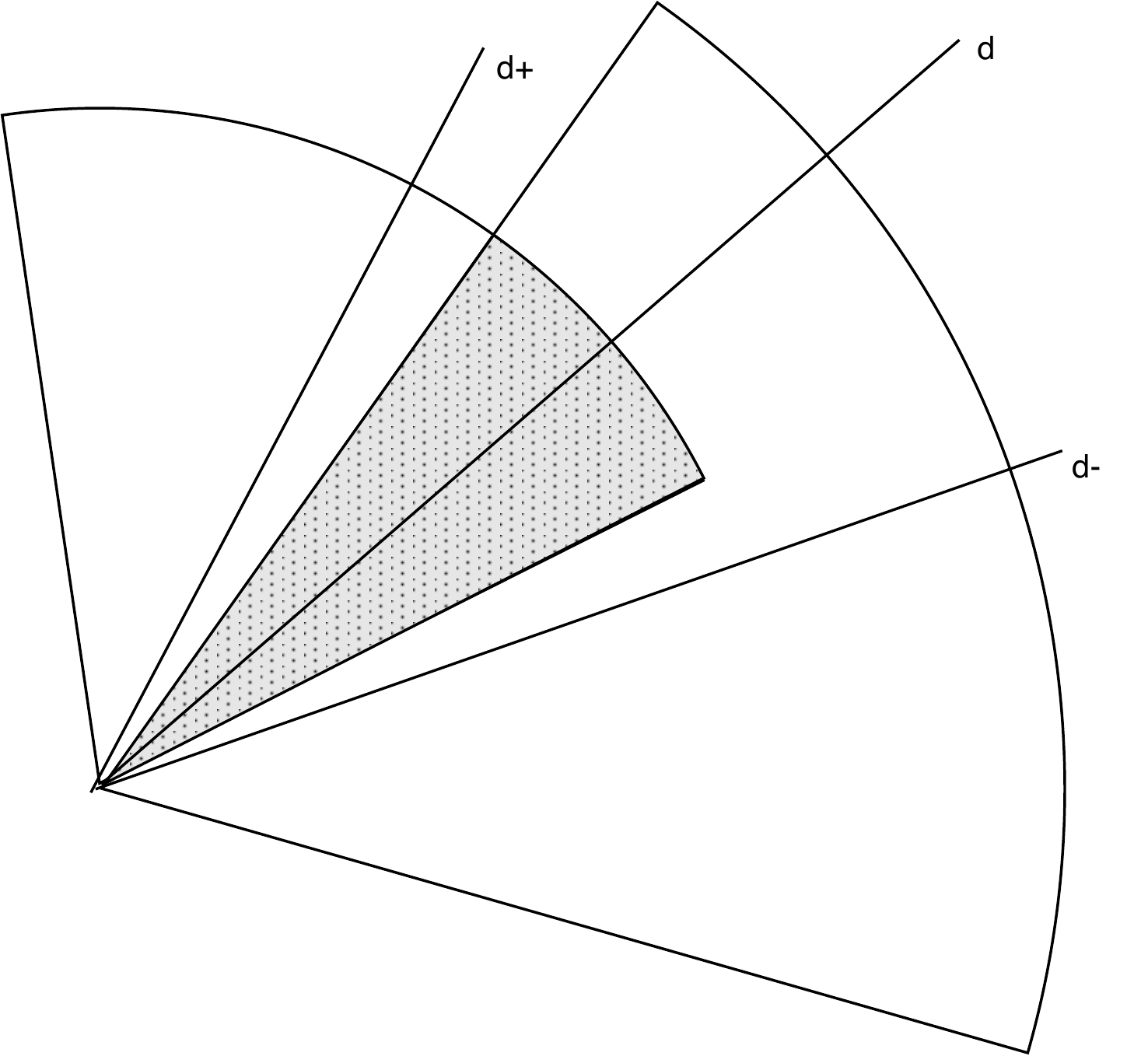}} \hspace{.4in} \parbox{2.5in}{Let  $d$ be a singular direction of $Y' = AY- YA_0$. Select $d^+$ and $d^-$ nearby so that the associated sectors overlap and contain $d$. Let $\phi^+, \phi^-$ be the associated multisums.  We must have $\phi^+ x^L e^{Q(1/t)}= \phi^-x^L e^{Q(1/t)}\cdot St_d$ for some $St_d \in \GL_n(\Cx)$ on the overlap of the regions (the shaded area in the figure). The matrix $St_d$ is independent of the  choice of $d^+$ and $d^-$.}
\begin{defin} The matrix $St_d$ is called a {\em Stokes matrix in the direction $d$.}\end{defin}
\begin{ex} (Euler equation)
\begin{itemize} 
\item Eigenvalues of $Y' = AY - YA_0$: $\{\frac{1}{x}, 0\}$
\item Singular direction of $Y' = AY - YA_0$: $d = \pi$  
\item $\hat{\phi}(x) = \left(\begin{array}{cc} 1&\hat{f}\\ -\frac{1}{x^2}&\hat{f}'\end{array}\right) \ \ \ \hat{f} = \sum_{n=0}^\infty (-1)^n x^{n+1}$
\end{itemize}
For $d+$ and $d^-$ close to the negative real axis, we have \[f^+ = \int_{d^+}\frac{1}{1+\zeta} e^{-\frac{\zeta}{x}}d\zeta \mbox{ and }
f^- = \int_{d^-}\frac{1}{1+\zeta} e^{-\frac{\zeta}{x}}d\zeta.\]
Since $ f^+ - f^- = 2\pi i \mbox{ Res}_{\zeta= -1} \left(\frac{e^{-\frac{\zeta}{x}}}{1+\zeta} \right)= 2\pi i$, we have 
 \[\left(\begin{array}{cc} 1&f^+\\ -\frac{1}{x^2}&(f^+)'\end{array}\right) = \left(\begin{array}{cc} 1&f^-\\ \-\frac{1}{x^2}&(f^-)'\end{array}\right) \left(\begin{array}{cc} 1&2\pi i\\ 0& 1\end{array}\right)\]
so  $St_{\pi} = \left(\begin{array}{cc} 1&2\pi i\\ 0& 1\end{array}\right)$
\end{ex}
The following result is due to Ramis \cite{ramis85a, ramis85b, ramis85c, martinet_ramis}. Another proof appears in   the work of Loday-Richaud \cite{loday_richaud}.  An exposition of this result and its proof appears in Chapter 8 of \cite{PuSi2003}. 
\begin{thm} The convergent differential Galois group $\DGal(K/\Cx(\{x\}))$ is the Zariski closure of the group generated by \begin{enumerate}
\item the formal differential Galois group $\DGal(K/\Cx((x)))$, and
\item The collection of Stokes matrices $\{St_d\}$ where $d$ runs over the set of singular directions of the polynomials $\{q_i-q_j\}$, $q_i,q_j$ eigenvalues of $Y' = AY$.
\end{enumerate}
\end{thm}
\begin{exs} (1) (Euler equation) 
\begin{itemize} 
\item The formal differential Galois group:  \[\{ \left(\begin{array}{cc} c&0\\ 0& 1\end{array}\right) \ | \ c\neq 0\}\]
\item Stokes matrix: $\left(\begin{array}{cc} 1&2\pi i\\ 0& 1\end{array}\right)$
\end{itemize}
Therefore the convergent differential Galois group is \[\{ \left(\begin{array}{cc} c&d\\ 0& 1\end{array}\right) \ | \ c\neq 0, d\in \Cx\}.\]
(2) (Airy equation)
\begin{itemize} 
\item The formal differential Galois group: \[D_\infty =  \{ \left(\begin{array}{cc} c&0\\ 0& c^{-1}\end{array}\right), \left(\begin{array}{cc} 0&c\\ -c^{-1}& 0\end{array}\right) \ | \ c\neq 0\}\]
\item The Stokes matrices: \[St_0 =  \left(\begin{array}{cc} 1&*\\\ 0& 1\end{array}\right), St_{\frac{2\pi}{3}} = \left(\begin{array}{cc} 1& 0\\ *& 1\end{array}\right), St_{\frac{4\pi}{3}} = \left(\begin{array}{cc} 1&*\\\ 0& 1\end{array}\right)\]
(see Example 8.15 of \cite{PuSi} for the calculation of the Stokes matrices.)
\end{itemize}
Therefore the convergent differential Galois group is  $\SL_2(\Cx)$
\end{exs}
We highly recommend the book \cite{ramis93} of Ramis, the papers of Loday-Richaud \cite{loday_richaud90, loday_richaud95} and the papers of Varadarajan \cite{varadarajan_mero, varadarajan} for further introductions to this analytic aspect of linear differential equations.  One can find references to the original papers in these works as well as in \cite{PuSi2003}.

\section{Algorithms}\label{mfssec4}  In this section we shall consider the question: Can we compute differential Galois groups and their properties? We shall not give a complete algorithm to determine the Galois groups although one exists - see \cite{Hrushovski}.  This general algorithm has not been implemented at present and, in its present form seems not very efficient.  We will  rather give a flavor of the various techniques that are presently implemented and used.\\[0.1in]
\noindent{\underline{\bf Categories motivate algorithms}}.  We will show how results in a categorical setting concerning representations of groups lead to algorithms to determine Galois groups.  We begin with an example from the Galois theory of polynomials.
\begin{ex}\label{galoisex1} Let \[P(Y) = Y^3 + bY + c = 0  \in \Qx[Y].\]
The usual Galois group $G$ of this equation is a subgroup of the full permutation group $\calS_3$ on 3 elements. We will describe how to compute $G$. If  $P(Y) = (Y + \alpha)(Y^2 + \beta Y + \gamma), \alpha, \beta, \gamma \in \Qx$, then $G$ is the Galois group of $Y^2 + \beta Y + \gamma$, which we assume we know how to compute.  Therefore we will assume that $P(Y)$ is irreducible over $\Qx$.  This implies that $G$ acts transitively on the roots $\{\rho_1, \rho_2, \rho_3\}$ of $P$.  From group theory we know that the only transitive subgroups of $\calS_3$ are $\calA_3$, the alternating group on 3 elements, and $\calS_3$. To decide which we have, consider the following associated polynomial

\[Z^2 +4b^3+27c^2 = (Z+\delta)(Z-\delta) \ \ \delta = (r_1-r_2)(r_2-r_3)(r_3-r_1).\]
One then has that $G = \calA_3$ if $Z^2 +4b^3+27c^2$ factors over $\Qx$ and 
 $G = \calS_3$ if $Z^2 +4b^3+27c^2$ irreducible over $\Qx$.
 \end{ex}
 In the above example we determined the Galois group by considering factorization properties of the polynomial and associated polynomials. Why does this work?\\[0.1in]
 To answer this question, let us consider a square-free polynomial $f(Y) \in \Qx[Y]$ with roots $r_1, \ldots r_n \in \bar{\Qx}$, the algebraic closure of $Q$.  Let $G$ be the Galois group of $f$ over $\Qx$. The group $G$ acts as a group of permutations on the set $S= \{r_1, \ldots , r_n\}$.   We have that $G$ acts transitively on $S$ if and only if the polynomial $f$ is irreducible.  More generally, $G$ leaves a subset $\{r_{i_1}, \ldots , r_{i_t}\}$ invariant if and only if the polynomial $g(Y) = \prod_{j=1}^t (Y-r_{i_j}) \in \Qx[Y]$.  Therefore the orbits of 
 $G$ in the set $S$ correspond to irreducible factors of $f$.  In fact, for any $h \in \Qx[Y]$ having all its roots in the splitting field of $f$, $G$-orbits in the set of roots correspond to irreducible factors of $h$.  In the above example, we were able to distinguish between groups because they have different orbits in  certain ``well-selected'' sets and we were able to see this phenomenon via factorization properties of associated polynomials. The key is that, in general,
\begin{itemize}
\item[ ] {\em one can distinguish between finite groups by looking at the sets on which they act.}
\end{itemize}
 This strange looking statement can be formalized in the following way (see Appendix B1 of \cite{PuSi2003} for details). For any finite set, let $\Perm(S)$ be the group of permutations of $S$. One constructs a category $\Perm_G$ of $G$-sets whose objects are  pairs $(S, \rho)$, where $\rho:G \rightarrow \Perm(S)$ is the map defining an action of $G$ on $S$.  The morphisms of this category are just maps between $G$-sets that commute with the action of $G$. One has the forgetful functor $\omega:\Perm_G \rightarrow \Sets$ which takes $(S,\rho)$ to $S$.  One then has the formalization of the above statement as: one can recover the group from the  pair $(\Perm_G, \omega)$. A consequence of this is that given two groups $H\subset G$ there is a set $S$ on which $G$ (and therefore $H$) acts so that the orbit structures with respect to these two groups are different. \\[0.1in]
 These  facts lead to the following strategy to compute Galois groups of a polynomial $f \in \Qx[X]$.  For each pair of groups $H\subset G \subset \calS_n$ we construct a polynomial $f_{H,G} \in \Qx[X]$ whose roots  can be expressed in terms of the roots of $f$ and such that $H$ and $G$ can be distinguished by their orbit structures on the roots of $f_{H,G}$ (by the above facts, we know that there is a set that distinguishes $H$ from $G$  and one can show that this set can be constructed from the roots of $f$). The factorization properties of $f_{H,G}$ will then distinguish $H$ from $G$.  This is what was done in Example~\ref{galoisex1}.\\[0.1in]
 To generalize the above ideas to calculating differential Galois groups of linear differential equations, we must first see what determines a linear algebraic group.  In fact, just as finite groups are determined by their {\em permutation} representations, linear algebraic groups are determined by their {\em linear} representations.  To make this more precise, let $G$ be a linear group algebraic group defined over an algebraically closed field $C$.  I will denote by $\Rep_G$ the category of linear representations of $G$.  The objects of this category are pairs $(V,\rho)$ where $V$ is a finite dimensional vector space over $C$ and $\rho:G \rightarrow \GL_n(V)$ is a representation of $G$.  The morphisms $m:(V_1,\rho_1) \rightarrow (V_2,\rho_2)$ are linear maps $m:V_1\rightarrow V_2$ such that $m\circ\rho_1 = \rho_2\circ m$.  One can define subobjects, quotients, duals, direct sums and tensor products in obvious ways (see Appendix B2 and B3 of \cite{PuSi2003} for details).  One again has a forgetful functor $\omega:\Rep_G \rightarrow \Vect_C$ (where $\Vect_C $ is the category of vector spaces over $C$) given by $\omega(V,\rho) = V$.   In analogy to the finite group case, we have Tannaka's Theorem, that says that we can recover $G$ from the pair $(\Rep_G, \omega)$.  \\[0.1in]
Let  $\calT$ be a category with a notion of quotients, duals, direct sums, tensor products and a functor $\tilde{\omega}:\calT \rightarrow \Vect_C$.  An example of such a category is the category of differential modules $\calD$ over a differential field $k$ with algebraically closed constants and $\tilde{\omega}(M) = \ker(\dd,K\otimes M)$ (see the discussion following Definition~\ref{fundmatrixdef}) where $K$ is a differential field large enough to contain solutions of the differential equations corresponding to the modules in $\calD$. One can give axioms that guarantee that $\tilde{\omega}(\calT)$ ``='' $\omega(\Rep_G)$ for some linear algebraic group $G$ (see \cite{deligne_milne}, \cite{deligne_tannakian} or Appendix  B3 of \cite{PuSi2003}). Such a pair $(\calT, \tilde{\omega})$ is called a {\em neutral tannakian category}.  Applying this to a subcategory of $\calD$, one has:
\begin{thm}\label{tannakathm} Let $(k, \dd)$ be a differential field with algebraically closed constants, let $Y' = AY$ be a linear differential equation over $(k,\dd)$, let $M_A$ be the associated differential module and $K$ the associated Picard-Vessiot extension.  Let $\calT = \{\{M_A\}\}$ be the smallest subcategory of $\calD$ containing $M_A$ and closed under the operations of taking submodules, quotients, duals, direct sums , and tensor products.  For $N \in \{\{M_A\}\}$, let $\tilde{\omega}(N) = \ker(\dd,N\otimes K)$.\\[0.05in]
Then $(\calT, \tilde{\omega})$ is a neutral tannakian category and the image of $\tilde{\omega}$ is the category of vector spaces on which $\DGal(K/k)$ acts.
\end{thm}
This theorem and Tannaka's Theorem have the following implications:
\begin{itemize}
\item One can construct all representations of $\DGal(K/k)$ by:
\begin{itemize}
\item using the constructions of linear algebra (submodules, quotients, duals, direct sums, tensor products) on $M_A$, and
\item taking the solution spaces of the corresponding differential equations.
\end{itemize}
\item $\DGal(K/k)$ is determined by knowing which subspaces of these solution spaces are $\DGal(K/k)$-invariant.
\end{itemize}
At first glance, it seems that one would need to consider an infinite number of differential modules but in many cases, it is enough to consider only a finite number of modules constructed from $M_A$. Regretably, there are groups (for example, any group of the form $C^*\oplus \ldots \oplus C^*$) for which on needs an infinite number of representations to distinguish it from other groups.  Nonetheless, there are many groups where a finite number of representations suffice. Furthermore, Proposition~\ref{invsubprop} states that the $\DGal$-invariant subspaces of the solution space of a differential module $N$ correspond to factors of the associated scalar equation $L_N(Y) = 0$. Therefore, in many cases, one can give criteria for a differential equation to have a given group as Galois group in terms of factorization properties of certain associated differential operators. We shall give examples for second order equations.  We begin with a definition.\\[0.1in]
Let $k$ be a differential field with algebraically closed constants $C$, let $L = \dd^2 - s \in k[\dd]$, and let $K$ be the Picard-Vessiot extension corresponding to $L(Y) = 0$. Let $\{y_1, y_2\}$ be a basis for the solution space $V$ of $L(Y) = 0$ in $K$ and let $V_m = C-\mbox{span of }\{y_1^m, y_1^{m-1}y_2, \ldots , y_2^m\}$.  One can see that $V_m$ is independent of the selected basis of $V$ and that it is invariant under the action of $\DGal(K/k)$.  As in the proof of Proposition~\ref{invsubprop} one can show that this implies that $V_m$ is precisely the solution space of a scalar linear differential equation over $k$.
\begin{defin} \label{scalarsym}The $m^{th}$ symmetric power $L^{\miprod m}$ of $L$ is the monic operator whose solution space is $V_m$.
\end{defin}
One has that the order of $L^{\miprod m}$ is $m+1$.  To see this it is enough to show that the $y_1^{m-i}y_2^i, \ 0\leq i \leq m$ are linearly independent over $C$.  Since any homogeneous polynomial in two variables factors over$C$, if we had  $0 = \sum c_iy_1^{m-i}y_2^i = \prod(a_iy_1+b_iy_2) $, then $y_1$ and $y_2$ would be linearly dependent over $C$.  One can calculate $L^{\miprod m}$ by starting with $y^m$ and formally differentiating it $m$ times.  One uses the relation $y'' -sy=0$ to replace $y^{(j)}, \ j>1$ with combinations of $y$ and $y'$.  One then has $m+2$ expressions $y^m, (y^m)', \ldots , (y^m)^{(m+1)}$ in the $m+1$ ``indeterminates'' $y^{m-i}(y')^i, \ 0 \leq i \leq m$.  These must be linearly dependent and one can find a dependence $(y^m)^{(m+1)} + b_m(y^m)^{(m)} + \ldots + b_0(y^m) = 0$.  One can show that $L^{\miprod m} = \dd^{m+1} + b_m \dd^m + \ldots + b_0$ (see Chapter 2.3 of \cite{PuSi2003}).

\begin{ex}$m=2$:\\
\parbox{1.8in}{$\begin{array}{ccc}
y^2 & = &y^2\\
(y^2)' & = & 2yy'\\
(y^2)'' & = & 2(y')^2 + 2sy^2\\
(y^2)''' & = & 8syy'+ 2s'y^2
\end{array}$} $\Rightarrow$ \ \ \  $L^{\miprod 2}(y) = y'''-4sy'-2s'y $
\end{ex}
The vector spaces $V_m$ correspond to representations of $\DGal(K/k)$ on the $m^{th}$ symmetric power of $V$  and are well understood for many groups \cite{fulton_harris} (the $m^{th}$ symmetric power of a vector space is the vector space with a {\em basis} formed by the  monomials of degree $m$ in a basis of $V$).   Using representation theory, one can show:
\begin{prop} Let $L(y) = y''-sy$ and let $G = \DGal(K/k)$. Assume $L$ and $L^{\miprod 2}$ are  irreducible.
\begin{itemize} 
\item If $L^{\miprod 3}$ is reducible then  $ \ G/\pm I \simeq A_4$.
\item If $L^{\miprod 3}$ is reducible and $L^{\miprod 4}$ is irreducible then $\ G/\pm I \simeq S_4$.  
\item If  $L^{\miprod 4}$ is reducible, and $L^{\miprod 6}$ is irreducible, then  $\Rightarrow \ G/\pm I \simeq A_5$. 
\item If $L^{\miprod 6}$is  irreducible then $ \ G\simeq \SL_2(C)$. 
\end{itemize}
\end{prop}
In order to implement these criteria, one needs good algorithms to factor linear operators.  Such algorithms exist and I refer to Chapter 4.2 of \cite{PuSi2003} for descriptions of some algorithms that factor linear operators as well as further references to the extensive literature. The calculation of symmetric powers and factorization of operators has been implemented in {\sc Maple} in the DEtools package.\\[0.1in]
One can develop similar criteria for higher order operators.  Given a differential operator $L$ of order $n$ with $\{y_1, \ldots, y_n\}$ being a basis of the solution space $V$ in some Picard-Vessiot extension, one can consider the $C$-space $V_m$ of 
homogeneous polynomials of degree $m$ in the $y_i$.  This space has dimension at most $\left(\begin{array}{c}m+n-1\\ n-1\end{array}\right)$ but may have dimension that is strictly smaller and so is not isomorphic to the $m^{th}$ symmetric power of $V$ (although it is isomorphic to a quotient).  This complicates matters but this situation can be dealt with by casting things directly in terms of differential modules, forming the symmetric powers of differential modules and developing algorithms  to find submodules of differential modules.  This is explained in Chapter 4.2 of \cite{PuSi2003}.  Examples of criteria for third order equations are given in Chapter 4.3.5 of \cite{PuSi2003} where further references are given as well.\\[0.2in]
\noindent\underline{\bf Solving $L(y) = 0$ in terms of exponentials, integrals and}\\ \underline{\bf algebraics}. I shall give a formal definition that captures the meaning of this term.
\begin{defin} Let $k$ be a differential field, $L \in k[\dd]$, and $K$ be the associated  Picard-Vessiot extension.  \\[0.05in]
1) A {\em liouvillian tower over $k$} is a tower of differential  fields $ k= K_0\subset \ldots \subset K_m$  with $K \subset K_m$ and, for each $i, 0\leq i < m$, $K_{i+1} = K_i(t_i)$, with
\begin{itemize} 
\item $t_i$ algebraic over $K_i$, or
\item $t_i' \in K_i$, i.e., $t_i = \int u_i, u_i \in K_i$, or
\item $t_i'/t_i \in K_i$, i.e., $t_i = e^{\int u_i} , u_i \in K_i$
\end{itemize}
2) An element of liouvillian tower is said to be {\em liouvillian over k}. \\[0.05in]3)  $L(y)= 0$ {\em  is solvable in terms of liouvillian functions} if the associated PV-extension $K$ lies in a liouvillian tower over $k$.
\end{defin}
Picard and Vessiot stated the following Galois theoretic criteria and this was given a formal modern proof by Kolchin \cite{kolchin48, DAAG}. 

\begin{thm}\label{liouvthm} Let $k,L,$ and $K$ be as above. $L(y) = 0$ is solvable in terms of liouvillian functions if and only if the identity component (in the Zarski topology) of $\DGal(K/k)$ is solvable.
\end{thm}
Recall that a group $G$ is solvable if there exists a tower of subgroups $G = G_0 \supset G_1 \supset \ldots \supset G_m = \{e\}$ such that $G_{i+1}$ is normal in $G_{i}$ and $G_{i+1}/G_i$ is abelian.
This above theorem depends on the 
\begin{thm} (Lie-Kolchin Theorem) Let $C$ be an algebraically closed field. A Zariski connected solvable group $G \subset \GL_n(C)$ is conjugate  to a groups of triangular matrices.  In particular,  $G$ leaves a one dimensional subspace invariant.
\end{thm}
I shall show how the Lie-Kolchin Theorem can be strengthened resulting in a strengthened version of Theorem~\ref{liouvthm} that leads to an algorithm to decide if a given linear differential equation (over $C(x)$) can be solved in terms of liouvillian functions. I begin with

\begin{prop} \label{propline} Let $k$ be a differential field, $L\in k[\dd]$, $K$ the corresponding Picard-Vessiot extensions and $G$ the differential Galois group of $K$ over $k$. The following are equivalent:\begin{enumerate}
\item  $L(y) = 0$ has a liouvillian solution $\neq 0$.
\item $G $ has a subgroup $H, |G:H| = m < \infty$ such that $H$ leaves a one dimensional subspace invariant.
\item  $L(y) = 0$ has a soln $z\neq 0$ such that $z'/z$ algebraic over $k$ of degree $\leq m$.
\end{enumerate}
\end{prop} 
\begin{proof}(Outline; see Chapter 4.3 of \cite{PuSi2003}) (iii) $\Rightarrow $(i): Clear.\\[0.05in]
(i) $\Rightarrow $(ii): One can reduce this to the case where all solutions are liouvillian. In this case, the Lie-Kolchin Theorem implies that the identity component $G^0$ of $G$ leaves a one dimensional subspace invariant.\\[0.05in]
(ii) $\Rightarrow $(iii): Let $V=Soln(L)$ and $v \in V$ span  an $H$-invariant line. We then have that  $\forall \sigma \in H, \exists c_\sigma \in C_k$ such that $\sigma(v) = c_\sigma v$.  This implies that $\forall \sigma \in H, \sigma(\frac{v'}{v}) = \frac{(cv)'}{cv} = \frac{v'}{v}$. Therefore, the Fundamental Theorem implies that $ \frac{v'}{v} \in E=$ fixed field of $H$.  One can show that $[E:k] =  |G:H| = m$ so $\frac{v'}{v}$ is algebraic over $k$ of degree at most $m$.\end{proof}
In order to use this result, one needs a bound on the integer $m$ that can appear in (ii) and (iii) above.  This is supplied by the following group theoretic result (see Chapter 4.3.1 of \cite{PuSi2003} for references).
\begin{lem} There exists a function $I(n)$ such that if $G \subset \GL_n(\Cx), H\subset G$ satisfies
\begin{enumerate}
\item $|G:H|<\infty$ and
\item $H$ leaves a one dimension subspace invariant
\end{enumerate}
then there exists a subsgroup $\tilde{H} \subset G$ such that
\begin{enumerate}
\item $|G:\tilde{H}| < I(n)$ and
\item $\tilde{H}$ leaves a one dimensional subspace invariant.
\end{enumerate}
\end{lem}
In general we know that $I(n) \leq n^{2n^2+2}$.  For small values of $n$ we have exact values, ({\em e.g.,} I(2) = 12, I(3) = 360). This clearly allows us to deduce the following corollary to   Proposition~\ref{propline}.
\begin{cor} Let $k, L, K$ be as in Propostion~\ref{propline} with the order of $L$ equal to $n$.  The following are equivalent 
\begin{enumerate}
\item  $L(y) = 0$ has a liouvillian solution $\neq 0$.
\item $G $ has a subgroup $H, |G:H| \leq I(n)$ such that $H$ leaves a one dimensional subspace invariant.
\item  $L(y) = 0$ has a soln $z\neq 0$ such that $z'/z$ algebraic $/k$ of deg $\leq I(n)$.
\end{enumerate}
\end{cor} 

This last result leads to several algorithms to decide if a differential equation $L(y) = 0$, $L \in C(x)[\dd], C$ a finitely generated subfield of $\Cx$, has a liouvillian solution.  The one presented in \cite{singer_liouvillian} (containing ideas going back to \cite{boulanger}) searches for a putative minimal polynomial $P(u) = a_mu^m + a_{m-1}u^{m-1} + \ldots + a_0, \ a_i \in \Cx[x]$ of an element $u = z'/z$ where $z$ is a solution of $L(y)=0$ and $m \leq I(n)$.  This algorithm shows how the degrees of the $a_i$ in $x$ can be bounded in terms of information calculated at each singular point of $L$.  Once one has degree bounds, the actual coefficients $a_i$ can be shown to satisfy a system of polynomial equations and one can (in theory) use various techniques ({\em e.g.,} Gr\"obner bases) to solve these.  Many improvements  and new ideas have been given  since then  (see Chapter 4 of \cite{PuSi2003}).  We shall present  criteria that form the basis of one method, describe what one needs to do to use this in general and give details for finding liouvillian solutions of second order differential equations.

\begin{prop} (1) Let $G$ be a subgroup of $\GL_n$.  There exists a subgroup $H \subset G, |G:H| \leq I(n)$ with $H$ leaving a one dimensional space invariant if and only if $G$ permutes a set of at most $I(n)$ one dimensional subspaces.\\[0.05in]
(2) Let $k,L,K,G$ be as in Propostion~\ref{propline} with the order of $L$ equal to $n$. Then $L(y) = 0$ has a liovuillian solution if and only if $G \subset \GL_n$ permutes a set of at most $I(n)$ one dimensional subspaces.
\end{prop}
\begin{proof} (1) Let $\ell$ be a one dimensional space left invariant by $H$.  The orbit of $\ell$ under the action of $G$ has dimension at most $|G:H|$.  Now, let $\ell$ be a one dimensional space whose orbit under $G$ is at most $I(n)$ and let $H$ be the stabiliizer of $\ell$.  We then have $|G:H| \leq I(n)$. (2) is an immediate consequence of (1) and the previous proposition. 
\end{proof}
To apply this result, we need the following definition:
\begin{defin} Let $V$ be a vector space of dimension $n$ and $t \geq 1$ be an integer.  The $t^{th}$ symmetric power $\Sym^t(V)$of $V$  is the quotient of the $t-fold$ tensor product $V^{\otimes t}$ by the subspace generated by elements  $v_1\otimes \ldots \otimes v_t - v_{\pi(1)}\otimes \ldots \otimes v_{\pi(t)}$, for all $ v_i \in V$ and $\pi$ a permutation.
\end{defin}
We denote by $v_1v_2\cdots v_t$ the image of $v_1\otimes \ldots \otimes v_t$ in $\Sym^t(V)$. If $e_1, \ldots ,e_n$ is a basis of $V$ and $1\leq t \leq n$, then $\{e_{i_1}e_{i_2}\cdots e_{i_t} \ | \ i_1 <i_2<\ldots  < i_t\}$ is a basis of $\Sym^t(V)$. Furthermore, if $G \subset \GL_n(V)$, then $G$ acts on $\Sym^t(V)$ as well via the formula $\sigma(v_1\cdots v_t) = \sigma(v_1)\cdots \sigma(v_t), \sigma \in G$. 
\begin{defin} (1) An element $w \in \Sym^t(V)$ is {\em decomposable} if $w = w_1\cdots w_t$ for some $w_i \in V$.\\[0.05in]
(2) A one dimensional subspace $\ell \subset \Sym^t(V)$ is {\em decomposable} if $\ell$ is spanned by a decomposable vector.
\end{defin} 
We note that if we fix a basis $\{e_1, \ldots , e_n\}$ of $V$ then an element \[\sum_{i_1<\ldots <i_t} c_{i_1\ldots i_t}e_{i_1}\cdots e_{i_t}\] is decomposable if and only if the $c_{i_1\ldots i_t}$ satisfy a system of equations called the {\em Brill equations} (\cite{GKZ}, p.~120,140).  
Using these definitions, it is not hard to show
\begin{prop} (1) A group $G \subset \GL_n(V)$ permutes a set of $t$ one dimensional subspaces if and only if $G$ leaves  a decomposable one dimensional subspace of $\Sym^t(V)$ invariant.\\[0.05in]
(2) Let $k,L,K,G$ be as in Propostion~\ref{propline} with the order of $L$ equal to $n$. Then $L(y) = 0$ has a liovuillian solution if and only if for some $t \leq I(n)$ $G$ leaves invariant a one dimensional subspace of $\Sym^t(\Soln(L))$.
\end{prop}
Note that $\Sym^t(V)$ is constructed using tensor products and quotients.  When we apply this construction to differential modules, we get
\begin{defin} Let $k$ be a differential field and $M$ a differential module over $k$.
The $t^{th}$ symmetric power $\Sym^t(M)$of $M$  is the quotient of the $t$-fold tensor product $M^{\otimes t}$ by the $k$-subspace generated by elements  $v_1\otimes \ldots \otimes v_t - v_{\pi(1)}\otimes \ldots \otimes v_{\pi(t)}$, for all $ v_i \in V$ and $\pi$ a permutation (note that this subspace is a differential submodule as well).
\end{defin}
 At the end of Section 1.2, I defined the solution space of a $k$-differential module $M$ in a differential field $K \supset k$ to be $\Soln_K(M) = \ker(\dd, K\otimes_k M)$.  From our discussion of tannakian categories, we have that \linebreak $\Soln_K(\Sym^t(M)) = \Sym^t(\Soln_K(M))$.  This latter fact suggests the following algorithm to decide if a linear differential equation has a liouvillian solution (we continue to work with scalar equations although everything generalizes easily to systems). Let $M_L$ be the differential module associated to the equation $L(y) = 0$.  For each $t \leq I(n)$
 \begin{itemize}
 \item Calculate $N_t = \Sym^t(M)$ and find all one dimensional submodules.
 \item Decide if any of these is decomposable as a subspace of $N_t$.
 \end{itemize}
The tannakian formalism implies that a one dimensional submodule of $N_t$ is decomposable if and only if its solution space is a decomposable subspace $G$-invariant subspace of $\Sym^t(\Soln_K(M_L))$. \\[0.1in]
Much work has been done on developing algorithms for these two step. The problem of finding one dimensional submodules of a differential module was essentially solved in the $19^{th}$ and early $20^{th}$ Centuries (albeit in a different language).  More recently work of Barkatou, Bronstein van Hoeij, Li, Weil, Wu, Zhang and others has produced good algorithms to solve this problem (see \cite{hoeij_weil}, \cite{hruw}, \cite{lswz} and Chapter 4.3.2 of \cite{PuSi2003} for references). As noted above, the second problem can be solved (in principal) using the Brill equations (see \cite{singer_ulmer_linearforms} and Chapter 4.3.2 of \cite{PuSi2003} for a fuller discussion).  Finally, one can modify the above to produce the liouvillian solutions as well. \\[0.1in]
Before I explicate these ideas further in the context of second order equations, I will say a few words concerning finding invariant one dimensional subspaces of solution spaces of linear scalar equations.  
\begin{prop} Let $k$ be  a differential field with algebraically closed constants $C$. $L \in k[\dd]$, $K$ the Picard-Vessiot extension of $k$ for $L(Y) = 0$ and $G$ the differential Galois group of $K$ over $k$.   An element $z$ in the  solution space $V$ of $L(Y) = 0$ spans a one dimensional $G$-invariant subspace if and only if 
  $u=z'/z$ is left fixed by $G$. In this case, $\dd-u$ is a right divisor of $L$.  Conversely if $\dd - u, u \in k$ is a right divisor of $L$, then there exists a solution $z \in K$ such that $z'/z = u$, in which case, $z$ spans a $G$-invariant space.
\end{prop}
\begin{proof} This follows from Proposition~\ref{invsubprop} and its proof applied to first order factors of $L$.
\end{proof}
Therefore, to find one dimensional $G$-invariant subspaces of the solution space of $L(Y) = 0$, we need to be able to find elements $u \in k$ such that $z=e^{\int u}$ satisfies $L(z) = 0$.  This is discussed in detail in Chapter 4.1 of \cite{PuSi2003} and I will give a taste of the idea behind a method to do this for $L \in \Cx(x)$ assuming $L$ has only regular singular points.  \\[0.2in]
Let $u \in \Cx(x)$ satisfy $L(e^{\int u}) = 0$ and let $x= \alpha$ be a pole of $u$. The asumption that $x_0$ is at worst a regular singular point implies that  $e^{\int u}$ is dominate by a power of $(x-\alpha)$ near $\alpha$. We must therefore have that $u$ has a pole of order at most 1 at $x = x_0$ and so $z = e^{\int u} = (x-\alpha)^a h(x)$ where $ a \in \Cx$ and $h$ is analytic near $\alpha \in \Cx$.  To determine $a$, we write
\[L(Y) = Y^{(n)} + (\frac{b_{n-1}}{(x-\alpha)} + \mbox{ h.o.t.})Y^{(n-1)}+ \ldots + (\frac{b_0}{(x-\alpha)^{n-1}}+ \mbox{ h.o.t.})Y.\]
Substituting $Y = (x-\alpha)^a(c_0+c_1(x-\alpha)+\mbox{ h.o.t.})$ and setting the coefficient of $(x-\alpha)^{a-n}$ equal to zero, we have
\[c_0(a(a-1)\ldots (a-(n-1)) + b_{n-1}(a(a-1)\ldots (a-(n-2)) +\ldots +b_0)) = 0\] or
 \[a(a-1)\ldots (a-(n-1)) + b_{n-1}(a(a-1)\ldots (a-(n-2)) +\ldots +b_0 = 0
\]
This latter equation is called the {\em indicial equation} at $x = \alpha$ and its roots are called the {\em exponents at $\alpha$}.  If $\alpha$ is an ordinary point, the $b_{n-1} = \ldots = b_0 = 0$ so $a \in \{0, \ldots , n-1\}$.  One can also define exponents at $\infty$.  We therefore have that 
\[ y = \prod_{\alpha_i = \mbox{ finite sing. pt. }}(x-\alpha_i)^{a_i} P(x)\]
where the $a_i$ are exponents at $\alpha_i$ and $-\sum a_i -\deg P$ is an exponent at $\infty$. The $a_i$ and the degree of $P$ are therefore determined up to a finite set of choices.  Note that $P$ is a solution of $\tilde{L}(Y) = \prod(x-\alpha_i)^{-a_i}L(\prod(x-\alpha_i)^{a_i}Y)$. Finding polynomial solutions of 
$\tilde{L}(Y) = 0 $ of a fixed degree  can be done by substituting a polynomial of that degree with undetermined coefficients and equating powers of $x$ to reduce this to a problem in linear algebra.\\[0.2in]
\noindent \underline{\bf Liouvillian Solutions of Second Order Equations.} The method \linebreak outlined above can be simplified for second order equations $L(Y) = Y'' - sY$ because of several facts that we summarize below (see Chapter 4.3.4 of \cite{PuSi2003}). The resulting algorithm is essentially the algorithm presented by Kovacic in \cite{kovacic86} but put in the context of the general algorithm mentioned above.  Kovacic's algorithm predated and motivated much of the  work on liouvillian solutions of general linear differential equation presented above. 
\begin{itemize}
\item It can be shown that the fact that no $Y'$ term appears implies that the differential Galois group must be a subgroup of $\SL_2$ (Exercise 1.35.5, p.~27 \cite{PuSi2003}).  
\item  An examination of the algebraic subgroups of $\SL_2$ implies that $\SL_2$ has no subgroup of finite index leaving a one dimensional subspace invariant and any proper algebraic subgroup of $\SL_2$ has a subgroup of index $1,2,4,6,$ or $12$ that leaves a one dimensional subspace invariant ({\em c.f.,} \cite{kovacic86}).
\item As shown in the paragraph following Definition~\ref{scalarsym}, the dimension of the solution space of $L^{\miprod t}$ is the same as the dimension of $\Sym^t(\Soln(L))$ and so these two spaces are the same.
\item Any element $z = \sum_{i=0}^t y_1^{t-i} y_2^i \in  \Sym^t(\Soln(L))$ can be written as a product $\prod_{i=0}^t(c_iy_1+d_iy_2)$ and so all elements of $\Sym^t(\Soln(L))$ are decomposable.
\end{itemize}
Combining these facts with the previous results we have

\begin{center}
  $L(y) = 0$ has a nonzero liouvillian solution
  
  $\Updownarrow$
  
The differential Galois group   $G$ permutes $1,2,4,6$ or $12$   one dimensional subspaces of  in $Sol(L)$
 
  $\Updownarrow$
  
 For $t =1,2,4,6,\mbox{ or } 12$, $G$ leaves invariant  a  line in $Sol(L^{\miprod t})$
 
  $\Updownarrow$
  
   For $t =1,2,4,6,\mbox{ or } 12$, $L^{\miprod t}$ has a solution $z$ such that $u = z'/z \in k$\end{center}
 So, to check if $Y'' - sY=0$ has a nonzero liouvillian solution (and therefore, by variation of parameters, only liouvillian solutions), one needs to check this last condition.  
 \begin{ex} \[L(Y) = Y'' + \frac{3(x^2-x+1)}{16(x-1)^2x^2}Y\]
 This is an equation with only regular singular points. One can show that $L(Y)=0$ has no nonzero solutions $z$ with $z'/z \in \Cx(x)$.  The second symmetric power is 
 \[L^{\miprod 2}(Y) = Y''+3/4\,{\frac { \left( {x}^{2}-x+1 \right)} {{x}^{2} \left( x-1 \right)^2 }}Y'-3/8\,{\frac { \left( 2\,{x}^{3}-3\,{x}^{2}+5\,x-2 \right)}{{x}^{3} \left( x-1\right)^3  
}} Y. \]
 The singular points of this equation are at $0,1$ and $\infty$.  The exponents there are\\[0.05in]
 \hspace*{.3in} At $0: \ \{1,\frac{1}{2}, \frac{3}{2}\}$\\
\hspace*{.3in} At $1: \ \{1,\frac{1}{2}, \frac{3}{2}\}$\\
\hspace*{.3in} At $\infty: \ \{-1,-\frac{1}{2}, -\frac{3}{2}\}$\\[0.1in]
For $Y = x^{a_0}(x-1)^{a_1}P(x)$, try $a_0 = \frac{1}{2}, a_1 = 1, \  \deg P = 0 $. One sees that  $Y= x^{\frac{1}{2}}(x-1)$ is a solution of $L^{\miprod 2}(y) = 0$. Therefore $L(Y) = 0$ is solvable in terms of liouvillian functions.
\end{ex}
For second order equations one can also easily see how to modify the above to find liouvillian solutions when they exist.  Assume that one has found a nonzero element $u \in k$ such that $L^{\miprod m}(Y)=0$ has a solution $z$ with $z'/z = u$.  We know that 
\[z = y_1\cdot \ldots \cdot y_m, \ \ \   y_1,\ldots,y_m \in Soln(L) \ . \]  Since $\sigma(z)\in C\cdot z$ for all $\sigma \in G$, we have that for any $\sigma \in Gal(L)$, there exists a permutation $\pi$ and constants $c_i$ such that $\sigma(y_i) = c_i y_{\pi(i)}$. Therefore, for $v_i = \frac{y_i'}{y_i}$ and $\sigma \in G$ we have 
$\sigma(v_i) = v_{\pi(i)}$, {\em i.e.,} $G$ permutes the $v_i$.  Let
\[P(Y)  = \prod_{i=1}^m(Y- v_i) = Y^m+a_{m-1}Y^{m-1}+ \frac{a_{m-2}}{2!}Y^{m-2}+\ldots + \frac{a_0}{m!}.\] The above reasoning implies that the $a_i \in k$ and  \[a_{m-1} = -(v_1+\ldots + v_m) = -(\frac{y_1'}{y_1}+\ldots + \frac{y_m'}{y_m}) = -\frac{(y_1\cdot \ldots\cdot y_m)'}{y_1\cdot \ldots\cdot y_m} = -\frac{z'}{z} = -u\]
The remaining $a_i$ can be calculated from $a_{m-1} = -u$ using the following fact from \cite{bmw}, \cite{kovacic86} and \cite{ulmer_weil} (see also Chapter 4.3.4 of \cite{PuSi2003}).
\begin{lem} Using the notation above, we have \[a_{i-1} = -a_i'-a_{m-i}a_i-(m-i)(i+1)sa_{i+1}, \ \ i=m-1,\ldots,0\] where $a_{-1} = 0.$
\end{lem}
In particular we can find a nonzero polynomial satisfied by $y'/y$ for some nonzero solution of $L(y) = 0$.  This gives a liouvillian solution and variation of parameters yields another.
\begin{ex} We continue with the example \[L(Y) = Y'' + \frac{3(x^2-x+1)}{16(x-1)^2x^2}Y\]
We have that $y = x^{\frac{1}{2}}(x-1)$ is a soln of $L^{\miprod 2}(Y) = 0$.
Let \[a_1 = -\frac{y'}{y} = -(\frac{1}{2x}+\frac{1}{x-1})= -\frac{3x-1}{2x(x-1)},\] then if $v$ is a root of
\[P(Y) = Y^2 - \frac{3x-1}{2x(x-1)}Y+ \frac{9x^2-7x+1}{16x^2(x-1)^2}\]
we have $y =e^{\int v}$
satisfies $L(y) = 0$.  This yields \[\sqrt[4]{x-x^2}\sqrt{1+\sqrt{x}}\mbox{ and} \sqrt[4]{x-x^2}\sqrt{1-\sqrt{x}}\]
as solutions. $DGal$= $D_4$ = symmetries of a square.
\end{ex}
Another approach to second order linear differential equations was discovered by Felix Klein.  It has been put in modern terms by Dwork and Baldassari and made more effective recently by van Hoeij and Weil (see \cite{HWKlein} for references).

\vspace{.1in}

\noindent\underline{\bf Solving in Terms of Lower Order Equations.}  Another way of defining the notion of ``solving in terms of liouvillian functions'' is to say that a 
linear differential equation can be solved in terms of solutions of first order equations $Y' + aY = b$ and algebraic functions. With this in mind it is natural to make the following definition.
\begin{defin} Let $k$ be a differential field, $L \in k[\dd]$ and $K$ the associated Picard-Vessiot extension.  The equation $L(Y) = 0$ is {\em solvable in terms of lower order linear equations} if there exists a tower $k = K_0 \subset K_1 \subset \ldots \subset K_m$ with $K \subset K_m$ and for each $i, 0\leq i < m$, $K_{i+1} = K_i<t_i>$ ($K_{i+1}$ is generated as a differential field by $t_i$ and $K_i$), with

\begin{itemize} 
\item $t_i$ algebraic over $K_i$, or\
\item  $ L_i(t_i) = b_i\mbox{ for some }L_i\in K_i[\dd] \ ord(L_i) <  ord(L),  \ b_i \in K_{i} $
\end{itemize}
\end{defin}
One has the following characterization of this property in terms of the differential Galois group.
\begin{thm}(\cite{singer_fano}) Let $k,L,K$ be as above with $\ord(L) = n$.  The equation $L(Y) = 0$ is \underline{not} solvable in terms of lower order linear equations if and only if the lie algebra of $\DGal(K/k)$ is simple and has no faithful representations of dimension less than $n$.
\end{thm}

If $\ord(L) = 3$, then an algorithm to determine if this equation is solvable in terms of lower order linear equations is given in \cite{singer_second}.  Refinements of this algorithm and extensions to higher order $L$ are given in \cite{hoeij_banach}, \cite{hoeij07}, \cite{nguyen}, \cite{nguyen_vdp}, \cite{person}. 
 \section{Inverse Problems}\label{mfssec5}  In this section we consider the following problem:\\[0.05in]
 {\em Given a differential field $k$, characterize those linear differential algebraic groups that appear as differential Galois groups of Picard-Vessiot extensions over $k$.}\\[0.05in]
  We begin with considering the inverse problem over fields of constants.\\[0.1in]
 \noindent \underline{\bf Differential Galois groups over $\Cx$.}  Let $L \in \Cx[\dd]$ be a linear differential operator with constant coefficients. All solutions of $L(Y) = 0, \ L \in \Cx[\dd]$ are of the form $\sum_iP_i(x)e^{\alpha_i x}$ where $P_i(x) \in \Cx[x]$, $\alpha_i \in \Cx, x' = 1$.  This implies that the associated Picard-Vessiot $K$ is a subfield of a field $E = \Cx(x, e^{\alpha_1 x}, \ldots , e^{\alpha_t x})$.  One can show that $\DGal(E/k) = (\Cx, +)\times((\Cx^*,\cdot)\times \ldots \times (\Cx^*, \cdot))$.  The group $\DGal(K/\Cx)$ is a quotient
  of this latter group and one can show that it therefore must be of the form 
$(\Cx, +)^a\times(\Cx^*,\cdot)^b, a= 0,1, \ b \in \Nx.$

\noindent One can furthermore characterize in purely group theoretic term the groups that appear as differential Galois groups over $\Cx$.  Note that \[(\Cx,+) \simeq \{\left( \begin{array}{cc} 1& a\\ 0 & 1 \end{array}\right) \ | \ a\in \Cx \}\]
and that all these elements are unipotent matrices ({\em i.e.}, $(A-I)^m=0$ for some $m\neq 0$). This motivates the following definition
\begin{defin} If $G$ is a linear differnetial algebraic group, the {\em unipotent radical $R_u$ of $G$} is the largest normal subgroup of $G$ all of whose elements are unipotent.
\end{defin}
Note that for a group of the form $G = \Cx\times(\Cx^*\times \ldots \times \Cx^*)$   we have that $R_u(G) =\Cx$.  Using facts about linear algebraic groups (see \cite{humphreys} or\cite{springer}) one can characterize those linear algebraic groups that appear as differential Galois groups over an  algebraically closed field of constants $C$ as
\begin{prop} A linear algebraic group $G$ is a differential Galois group of a Picard-Vessiot extension of $(\Cx,\dd), \dd c = \forall c\in \Cx$ if and only if $G$ is connected, abelian and $R_u(G) \leq 1$.
\end{prop}
All of the above holds equally well for any algebraically closed field $C$.\\[0.2in]
\noindent \underline{\bf Differential Galois groups over $\Cx(x), x'=1$.} The inverse problem for this field was first solved by C.~and M.~Tretkoff, who showed 
\begin{thm}\label{ttthm}\cite{tretkoff79} Any linear algebraic group is  a differential Galois group of a Picard-Vessiot extension of $\Cx(x)$.
\end{thm}
Their proof (which I will outline below) depends on the solution of Hilbert's $21^{st}$ Problem.  This problem has a weak and strong form. Let $\calS = \{\alpha_1, \ldots , \alpha_m, \infty \} \subset S^2$, $\alpha_0 \in S^2-\calS$ and $\rho:\pi_1(S^2-\calS, \alpha_0)\rightarrow \GL_n(\Cx)$ be a homomorphism.\\[0.1in]
\noindent \underline{Weak Form} Does there exist $A\in \Mn(\Cx(x))$ such that 
\begin{itemize}
\item $Y'=AY$ has only regular singular points, and
\item The monodromy representation of $Y'=AY$ is $\rho$.\end{itemize}
\noindent \underline{Strong Form} Do there exist $A_i \in \GL_n(\Cx)$ such that the monodromy rep of 
\[Y' = (\frac{A_1}{x-\alpha_1} + \ldots +  \frac{A_m}{x-\alpha_m})Y\]
is $\rho$.\\[0.1in]
We know that a differential equation $Y' = AY$ that has a regular singular point is locally equivalent to one with a simple pole.  The strong form of the problem insists that we find an equation that is {\em globally} equivalent to one with only simple poles.  A solution of the strong form of the problem of course yields a solution of the weak form.  Many special cases of the strong form were solved before a counterexample to the general case was found by Bolibruch (see \cite{beauville}, \cite{anosov_bolibruch} or Chapters 5 and 6 of \cite{PuSi2003} for a history and exposition of results and \cite{BMM2006} for more recent work and generalizations).  
\begin{enumerate}\item Let  $\gamma_1, \ldots ,\gamma_m$ be generators of $\pi_1(S^2-\calS, \alpha_0)$, each enclosing just one $\alpha_i$. If some $\rho(\gamma_i)$ is diagonalizable, then the answer is yes. (Plemelj)
\item  If all $\rho(\gamma_i)$ are sufficiently close to $I$, the answer is yes. (Lappo-Danilevsky)
\item If $n =2$, the answer is yes. (Dekkers)
\item  If $\rho$ is irreducible, the answer is yes. (Kostov, Bolibruch)
\item  Counterexample for $n = 3, |\calS| = 4$ and complete characterization for $n = 3,4$ (Bolibruch, Gladyshev)
\end{enumerate}
In a sense, the positive answer to the weak from follows from Plemelj's result above - one just needs to add an additional singular point $\alpha_{m+1}$ and let $\gamma_{m+1}$ be the identity matrix.   A modern approach to give a positive answer to the weak form was given  by R\"ohrl and later Deligne (see \cite{deligne_LNM}).  We now turn to \\[0.1in]
\noindent {\em Proof of Theorem~\ref{ttthm}} (outline, see \cite{tretkoff79} of Chapter 5.2 or \cite{PuSi2003} for details)  Let $G$ be a linear algebraic group.  One can show that there exist $g_1, \ldots, g_m \in G$ that generate a Zariski dense subgroup $H$.   Select points $\calS = \{\alpha_1, \ldots \alpha_m\}\subset S^2$ and define $\rho:\pi_1(S^2-\calS, \alpha_0) \rightarrow H$ via $\rho(\gamma_i) = g_i$. Using the solution of the weak form of Hilbert's $21^{st}$, there exist  $A \in \Mn(\Cx(x))$ such that $Y' = AY$ has regular singular points and monodromy gp $H$.  Schlesinger's Theorem,  Theorem~\ref{schlesinger}, implies that the differential Galois group of this equation   is the Zariski closure of $H$.\hfill $\square$\\[0.1in] 
The above result leads to the following two  questions:
\begin{itemize}
\item What is the minimum number of singular points (not necessarily regular singular points) that a linear differential system must have to realize a given group as its Galois group?
\item Can we realize all linear algebraic group as differential Galois groups over $C(x), x'=1$ where $C$ is an arbitrary algebraically closed  field?
\end{itemize}
To answer the first question, we first will consider\\[0.1in]
\noindent \underline{\bf Differential Galois groups over $\Cx(\{x\})$.}
To characterize which groups occur as differential Galois groups over  this field, we need the following definitions.
\begin{defin} Let $C$ be an algebraically closed field.\\[0.05in]
(1) A {\em torus} is a linear algebraic group isomorphic to $(C^*, \cdot)^r$ for some $r$.\\[0.05in]
(2) If $G$ is a linear algebraic group then we define $L(G)$ to be the group generated by all tori in $G$.  This is a normal, Zariski closed subgroup of $G$ (see \cite{ramis_inverse} and Chapter 11.3 of \cite{PuSi2003}).
\end{defin}
\begin{exs} (1) If $G$ is reductive and connected ({\em e.g.}, tori, $\GL_n, \SL_n$), then $L(G) = G$ (A group $G$ is reductive if $R_u(G)= \{I\}$). \\[0.05in]
(2) If  $G= (C,+)^r, r \geq 2$, then $L(G) = \{I\}$. 
\end{exs}
Ramis \cite{ramis_inverse} showed the following (see also Chapter 11.4 of \cite{PuSi2003}).
\begin{thm} A linear algebraic group $G$ is a differential Galois group of a Picard-Vessiot extension of $\Cx(\{x\})$ if and only if $G/L(G)$ is the Zariski closure of a cyclic group.
\end{thm}
Any linear algebraic group $G$ with $G/G^0$ cyclic and $G$ reductive satisfies these criteria while $G= (C,+)^r, r \geq 2$ does not.  Ramis showed that his characterization of local Galois groups in terms of formal monodromy, exponential torus and Stokes matrices yields a group of this type and conversely any such group can be realized as a local Galois group.  Ramis was furthermore able to use analytic patching techniques to get a global version of this result.
\begin{thm} A linear algebraic group $G$ is a differential Galois group of a Picard-Vessiot extension of $\Cx(x)$  with at most $r-1$ regular singular points and one (possibly) irregular singular point if and only if $G/L(G)$ is the Zariski closure of a group generated by $r-1$ elements.
\end{thm}
\begin{exs} (1) We once again have that any linear algebraic group is a differential Galois group over $\Cx(x)$ since such a group satisfies the above criteria for some $r$.\\[0.05in]
(2) $\SL_n$ can be realized as a differential Galois group of a linear differential equation over $\Cx(x)$ with only one singular point.\\[0.05in]
(3) One needs  $r$ singular points to realize the group $(\Cx,+)^{r-1}$ as a differential Galois group over $\Cx(x)$.
\end{exs}
\noindent \underline{\bf Differential Galois groups over $C(x)$, $C$ an algebraically closed} \\ \underline{\bf field.} The proofs of the above results depend on analytic techniques over $\Cx$ 
that do not necessarily apply to general algebraically closed fields $C$ of characteristic zero.  The paper \cite{singer_moduli} examines the question of the existence of ``transfer principles'' and shows that for certain groups $G$, the existence of an equation over $\Cx(x)$ having differential Galois group $G$ over implies the existence of an equation over $C(x)$ having differential Galois group $G$ over $C(x)$ but these results do not apply to all groups. Algebraic proofs for various groups over $C(x)$ (in fact over any differential field finitely generated over its algebraically closed field of constants) have been given over the years by Bialynicki-Birula (nilpotent groups), Kovacic (solvable groups), Mitschi/Singer(connected groups) (see \cite{MiSi:JA} for references).  Finally, Hartmann \cite{hartmann2005} showed that any linear algebraic group can be realized over $C(x)$ (see also \cite{oberlies}).  An exciting recent development is the announcement of a new proof by Harbater and Hartmann \cite{HH1,HH2} of this fact based on a generalization of algebraic patching techniques.  I will give two examples of a constructive technique (from \cite{MiSi:JA}) that allows one to produce explicit equations for connected linear algebraic groups. This technique is based on (see Proposition 1.31 of \cite{PuSi2003})
\begin{prop} Let $Y'=AY$ be a differential equation over a differential field $k$ with algebraically closed constants $C$.  Let $G,H$ be linear algebraic subgroups of $\GL_n(C)$ with lie algebras $\calG,\calH$. Assume $G$ is connected.\\[0.05in]
1) Let  $Y' = AY$ is a differential equation with   $A \in \calG\otimes_Ck$  and assume that $G$ is connected. Then the differential Galois group of this equation  is conjugate to a subgroup of $G$.\\[0.05in]
(2) Assume $k = C(x)$, that $A \in \calG\otimes_CC(x)$ and that the differential Galois group of $Y' = AY$ is  $H\subset G$, where $H$ is assumed to be connected.  Then there exists $B \in G(C(x))$ such that $\tilde{A} = B^{-1}AB - B^{-1}B \in \calH\otimes_CC(x)$, that is $Y'=AY$ is equivalent to an equation $Y' = \tilde{A}Y$ with $\tilde{A}\in \calG\otimes_CC(x)$ and the equivalence is given by an element of 
$G(C(x))$.
\end{prop}
These results allow us to formulate the following strategy to construct differential equations with differential Galois group over $C(x)$ a given connected group $G$.  Select $A \in \calG\otimes_CC(x)$ such that
\begin{enumerate}
\item the differential Galois group of $Y' = AY$ over $C(x)$ is connected, and
\item for any proper connected linear algebraic subgroup $H$ of $G$, there is no $B \in G(C(x))$ such that $B^{-1}AB - B^{-1}B' \in \calH\otimes_CC(x)$.
\end{enumerate}
\noindent We will assume for convenience that $C \subset \Cx$. To insure that the differential Galois group is connected, we will select an element $A \in \calG\otimes_CC[x]$, that is a matrix with {\em polynomial entries}.  This implies that there is a fundamental solution matrix whose entries are entire functions.  Since these functions and their derivatives generate the associated Picard-Vessiot extension $K$, any element of $K$ will be a function with at worst poles on the complex plane.  If $E, C(x)\subset E \subset K$ is the fixed field of $G^0$, then the elements of $E$ are algebraic functions that can only be ramified at the singular points of $Y' = AY$, that is, only at $\infty$.  But such an algebraic function must be rational so $E = C(x)$ and therefore $G = G^0$.\\[0.1in]
The next step is to select an $A\in \calG\otimes_CC[x]$ satisfying condition (ii) above. This is the heart of \cite{MiSi:JA} and we shall only give two examples.
\begin{ex}$G = C^*\times C^* =$ \[\{\left( \begin{array}{cc} a&0\\0&b\end{array}\right) \ | \ ab\neq 0\}\]
If $\calG$ is the lie algebra of $G$ then  $\calG\otimes_CC[x] =  $\[\{\left( \begin{array}{cc} f_1&0\\0&f_2\end{array}\right) \ | \ f_1,f_2 \in C[x]\}\]
The proper closed subgroups of $G$ are of the form
\[G_{m,n} = \{\left( \begin{array}{cc} c&0\\0&d\end{array}\right) \ | \ c^nd^m=1\}\]
$m,n$ integers, not  both zero.  The connected subgroups are those $G_{m,n}$ with $m,n$ relatively prime.  If $\calG_{m,n}$ is the lie algebra of $G_{m,n}$ then  $\calG_{m,n}\otimes_CC(x) =$\[\{\left( \begin{array}{cc} g_1&0\\0&g_2\end{array}\right) \ | \ mg_1+ng_2=0 \}\]
We need to find $A = \left( \begin{array}{cc} f_1&0\\0&f_2\end{array}\right), f_1,f_2 \in  C[x] $ such that for any $B=
\left( \begin{array}{cc} u&0\\0&v\end{array}\right), u,v\in C(x)$, $BAB^{-1} -BB^{-1}\notin\calG_{m,n}\otimes_CC(x)$, that is
\[m(f_1-\frac{u'}{u}) + n(f_2-\frac{v'}{v}) = 0 \Rightarrow m= n = 0.\]
It is sufficient to find $f_1,f_2 \in C[x]$ such that 
\[mf_1+nf_2 = \frac{h'}{h} \mbox{ for some } h \in C(x) \Rightarrow m = n = 0\]
Selecting $f_1 = 1$ and $f_2 = \sqrt{2}$ will suffice so
\[Y' = \left( \begin{array}{cc} 1&0\\0&\sqrt{2}\end{array}\right)Y\]
has Galois group $G$.
\end{ex}
\begin{ex} $ G = {\rm SL}_2$.  Its lie algebra is  ${\rm sl}_2$  = trace $0$ matrices. Let
\[ A = A_0+xA_1 = \left( \begin{array}{cc} 0&1\\ 1&0\end{array}\right) + x \left( \begin{array}{cc} 1&0\\0&-1\end{array}\right)\]
It can be shown that any proper subalgebra of ${\rm sl}_2$ is solvable and therefore leaves a one dimensional subspace invariant.  A  calculation shows that no $U'U^{-1}+UAU^{-1}$ leaves a line invariant (see \cite{MiSi:JA}) for details) and so this equation must have Galois group $\SL_2$.
\end{ex}

This latter example can be generalized to any  semisimple linear algebraic group $G$. Such a group can be realized as the differential Galois group of an equation of the form $Y' = (A_0+xA_1)Y$  where $A_0$ is a sum of generators of the root spaces and $A_1$ is a sufficiently general element of the Cartan subalgebra.  In general, we have
\begin{thm} Let $G$ be a connected lin alg gp. defined over alg closed field $C$.  One can construct $A_i \in \Mn(C)$ and $A_{\infty} \in C[x]$ such that for distinct $\alpha_i \in C$
\[Y' = (\frac{A_1}{x-\alpha_1} + \ldots +\frac{A_d}{x-\alpha_d} + A_{\infty})Y\]
has Galois group $G$.
\end{thm}

Furthermore, the number $d$ in the above result coincides with the number of generators of a Zariski dense subgroup of $G/L(G)$ and the degree of the polynomials in $A_\infty$ can be bounded in terms of the groups as well.  For non-connected groups \cite{MiSi:TOU} gives a construction to realize solvable-by-finite groups and \cite{CMS05} gives a construction to realize certain semisimple-by-finite groups. Finally, many linear algebraic groups can be realized as differential Galois groups of members of classical families of differential equations \cite{BBH, BH,  katz_calculations, katz_exponential, katz_rigid, duval_mitschi, mitschi89, mitschi96}.\\[0.1in]
To end this section, I will mention a general inverse problem.  We know that any Picard-Vessiot ring is the coordinate ring of a torsor for the differential Galois group. 
One can ask if every coordinate ring of a torsor for a linear algebraic group can be given the structure of a PIcard-Vessiot ring.  To be more precise, let $k$ be a differential field withe derivation $\dd$ and algebraically closed constants $C$.  Let $G$ be a linear algebraic group defined over $C$ and $V$ a $G$-torsor defined over $k$.\\[0.1in]
{\em Does there exist a derivation on $R=k[V]$ extending $\dd$ such that $R$ is a Picard-Vessiot ring with differential Galois group $G$ and such that the Galois action of $G$ on $R$ corresponds to the action  induced by $G$ on the torsor $V$?} \\[0.1in]
When $k = C(x)$, this has can be answered affirmatively due to the work of \cite{hartmann}.  For a general $k$, finitely generated over $C$, I do not know the answer. For certain groups and fields, Juan and Ledet have given positive answers, see \cite{J2, JL1, JL2, JL3}. \\[0.1in]
Finally, the inverse problem over $\Rx(x)$ has been considered by \cite{dyck_inverse}.

\section{Families of Linear Differential Equations}\label{mfssec6}
In this section we will consider a family of parameterized linear differential equations and discuss how the differential Galois group depends on the parameter.  I will start by describing the situation for polynomials and the usual Galois groups. Let $C$ be a field and $\overline{C}$ its algebraic closure\\[0.1in]
 Let $G \subset \calS_m$ = fixed group of permutations and 
${\calP(n,m, G)} =$
\[ \{ P = \sum_{i=0}^n\sum_{j=0}^m a_{i,j}x^iy^j \ | a_{i,j}  \in C, Gal(P/C(x)) = G\}\subset C^{(n+1)(m+1)}\]
To describe the structure of $\calP(n,m, G)$, we need the following definition
\begin{defin} A set $S \subset \overline{C}^N$ is {\em $C$-constructible} if it is the finite union of sets \[\{a \in \overline{C}^N\ | \ f_1(a) = \ldots =  f_t(a) = 0, g(a) \neq 0\}\]
where $f_i, g \in C[X_1,\ldots ,X_N]$.

\end{defin}
For example, any Zariski closed set defined over $C$ is $C$-constructible and a subset $S \subset \overline{C}$ is $\overline{C}$-constructible if and only if it is finite or cofinite.  In particular, $\Qx$ is not $\Cx$-constructible in $\Cx$.
We have the following results (see \cite{vdDR} and the references given there)
\begin{thm}$\calP(m,n,G)$ is $\Qx$-constructible.\end{thm}
From this we can furthermore deduce
\begin{cor}\label{finiteinverse} Any finite group $G$ is a Galois group over $\overline{\Qx}(x)$.\end{cor}
\begin{proof} (Outine)  For $G \subset \calS_m$, one can construct an $m$-sheeted normal covering of the Riemann Sphere with this as the group of deck transformations, \cite{tret71}.  The Riemann Existence Theorem implies that this Riemann Surface is an algebraic curve and that the the Galois group of its function field over $\Cx(x)$ is $G$. Therefore for some $n$, $\calP(n,m, G)$ has a point in $C^{(n+1)(m+1)}$.  Hilbert's Nullstellensatz implies it will have a point in $\overline{\Qx}^{(n+1)(m+1)}$ and therefore that $G$ is a Galois group over $\overline{\Qx}(x)$.
\end{proof}
\noindent\underline{\bf Differential Galois Groups of Families of Linear Differential}\linebreak \underline{\bf  Equations.}
I would like to apply the same strategy to linear differential equations. We begin by fixing a linear algebraic group $G \subset \GL_m$ and set ${ \calL(m,n,G)} = $
\[ \{ L = \sum_{i=0}^n\sum_{j=0}^m a_{i,j}x^i\dd^j \ | a_{i,j}  \in \Cx, Gal(L/\Cx(x)) = G\}\subset \Cx^{(n+1)(m+1)}\]
Regrettably, this set is not necessarily a $\Qx$ or even a $\Cx$-constructible set.
\begin{ex}\label{famex1} We have seen that the Galois group of ${ \frac{dy}{dx} - \frac{\alpha}{x} y}$ is $ \GL_1(\Cx)$ if and only if $\alpha \notin \Qx$.  If $\calL(1,1,GL_1)$ were $\Cx$-constructible then 
\[\calL(1,1,GL_1) \cap \{ x\dd - \alpha  \ | \alpha \in \Cx\} = \Cx \backslash \Qx\] would be constructible.  Therefore, even the  set $\calL(1,1,GL_1)$ is not $\Cx$-constructible.
\end{ex}
\begin{ex}\label{famex2} Consider the family \[\frac{d^2y}{dx^2} -( \alpha_1 +\alpha_2)\frac{dy}{dx}  +\alpha_1\alpha_2 y = 0 \hspace{.3in} G = \Cx^* \oplus \Cx^* \]

This equation has solution $\{e^{\alpha_1 x}, e^{\alpha_2 x}\}$. Assume $\alpha_2 \neq 0$.  
If $\alpha_1/\alpha_2 \notin \Qx \mbox{ then }DGal = \Cx^* \oplus \Cx^*$ while if  $\alpha_1/\alpha_2 \in \Qx \mbox{ then } DGal = \Cx^*$. Therefore, $\calL(2,0,\Cx^* \oplus \Cx^*)$ is not  $\Cx$-constructible.
\end{ex}

Recall that  for $L \in \Cx(x)[\dd] $ and   $x_0$ a singular point $ L(y) = 0$ has a fundamental set of solutions:
\[y_i = (x-\alpha)^{\rho_i} e^{P_i(1/t)}(\sum_{j=0}^{s_i} b_{ij}(\log(x-x_0)^j)\]
where $\rho_i \in \Cx$,  $t^\ell = (x-x_0)$ for some $\ell \leq n!$, $P_i$ are polynomials without constant term, $b_{ij} \in \Cx[[x-x_0]]$.  
\begin{defin} The set $\calD_{x_0}= \{ \rho_1, \ldots , \rho_n, P_1, \ldots , P_n\}$ is called the {\em local data of $L$} at $x_0$.
\end{defin}
In Example~\ref{famex1}, $0$ and $\infty$ are the singular points and at both of these $\{\alpha\}$ is the local data.  In Example~\ref{famex2}, $\infty$ is the only singular point and the local data is $\{P_1 = \frac{\alpha_1}{x}, P_2 = \frac{\alpha_2}{x}\}$.  In both cases, the parameters allow us to vary the local data.  One would hope that by fixing the local data, parameterized families of linear differential equations with fixed Galois group $G$ would be constructible.  This turns out to be true for many {\em but not all} groups $G$.
\begin{defin} Let $\calD =  \{ \rho_1, \ldots , \rho_r, P_1, \ldots , P_s\}$ be a finite set with $\rho_i \in \Cx$ and $P_i$ polynomials without constant terms. We define \[ \calL(m,n,\calD,G) = \{L \in \calL(m,n,G) \ | \ \calD_{a} \subset \calD \mbox{ for all sing pts } a  \in S^2\}\]
\end{defin} 
Note that we do not fix the position of the singular points; we only fix the local data at putative singular points.  Also note that $\calL(m,n,\calD) 
= \{ L= \sum_{i=0}^n\sum_{j=0}^m a_{i,j}x^i\dd^j \ | a_{i,j}  \in \Cx, \calD_{x_0} \subset \calD, \mbox{ for all sing pts } x_0\}$ is a constructible set.  To describe the linear algebraic groups for which $\calL(m,n,\calD,G)$ is a constructible set, we need the following definitions. 
\begin{defin} Let $G$ a linear algebraic group  defined over $C$, an algebraically closed field  and let $G^0$ be the identity component.\\[0.05in]
1) A {\em character $\chi$ } is a polynomial homomorphism $\chi:G \rightarrow C^*$.\\[0.05in]
2) {\em $KerX(G^0)$} is defined to be the intersection of kernels of all characters $\chi: G^0 \rightarrow C^*$.
\end{defin}
\begin{exs} 1) If $G$ is finite, then $KerX(G^0)$ is trivial.\\[0.05in]
2) If $G^0$ is semisimple or unipotent, then $KerX(G^0) = G^0$.
\end{exs}
One can show \cite{singer_moduli} that $KerX(G^0)$ is the smallest normal subgroup such that $G^0/KerX(G^0)$ is a torus.  In addition it is the subgroup generated by all unipotent elements of $G^0$. Furthermore, $KerX(G^0)$ is normal in $G$ and we have
\[1\rightarrow G^0/KerX(G^0)\rightarrow G/KerX(G^0)\rightarrow G/G^0 \rightarrow 1\]
Finally, $G^0/KerX(G^0)$ is abelian, so we have that  $G/G^0$ acts on $G^0/KerX(G^0)$.

\begin{thm}\label{diffcons}\cite{singer_moduli} Assume that the action of $G/G^0$  on $G^0/KerX(G^0)$ is trivial. Then $\calL(m,n,\calD,G)$ is constructible.
\end{thm}
\begin{exs} The groups $G$ that satisfy the hypothesis of this theorem include all finite groups, all connected groups and all groups that are either semisimple or unipotent.  An example of a group that does not satisfy the hypothesis is the group
$\Cx^*\rtimes \{1,-1\}$ where $-1$ sends $c$ to $c^{-1}$. \end{exs} 
For groups satisfying the hypothesis of this theorem we can prove a result similar to Corollary~\ref{finiteinverse}.

\begin{cor}  Let $G$ be as in Theorem~\ref{diffcons}, defined over $\overline{\Qx}$.  Then $G$ can be realized as a differential Galois group over $\overline{\Qx}(x)$.
\end{cor}
	\begin{proof} (Outline) Let $G$ be as in Theorem~\ref{diffcons}, defined over $\overline{\Qx}$. We know that $G$ can be realized as the differential Galois group of a linear differential equation over $\Cx(x)$, {\em e.g.,} Theorem~\ref{ttthm}. A small modification of Theorem~\ref{ttthm} allows one to assume that the local data of such an equation is defined over $\overline{\Qx}$ (see \cite{singer_moduli}). Therefore for some $n,m, \calD$, $\calL(m,n,\calD, G)$ has a point in ${\Cx}^{ (n+1)(m+1)}$.  The Hilbert Nullstellensatz now implies that $\calL(m,n,\calD, G)$ has a point in $\overline{\Qx}^{(n+1)(m+1)}$. Therefore  $G \subset \GL_m$ is a differential Galois gp over $\overline{\Qx}(x)$. \end{proof}
The above proof ultimately depends on analytic considerations.  We again note that Harbater and Hartman \cite{HH1,HH2} have given a direct algebraic proof showing all linear algebraic groups defined over $\overline{\Qx}$ can be realized as a differential Galois group over $\overline{\Qx}(x)$..\\[0.1in]
The   hypotheses of Theorem~\ref{diffcons} are needed.  In fact one can construct (see \cite{singer_moduli}, p.~384-385)) a two-parameter family of second order linear differential equations $L_{a,t}(Y)=0$ such that 
\begin{itemize}
\item $L_{a,t}(Y)=0$ has $4$ regular singular points $0,1,\infty, a$ and exponents {\em independent of the parameters $a,t$},
\item for all values of the parameters, the differential Galois group is a subgroup of $\Cx^*\rtimes \{1,-1\}$ where $-1$ sends $c$ to $c^{-1}$.
\item there is an elliptic curve $ \calC :y^2 = f(x)$, $f$ cubic with $(f,f')=1$ such that if $(a,y(a))$ is a point of finite order on $\calC$ then there exists a unique $t$ such that the differential Galois group of $L_{a,t}(Y)=0$ is finite and conversely, the differential Galois group being finite implies that $(a,y(a))$ is a point of finite order on $\calC$.
\end{itemize}
One already sees in the Lam\'e equation 
\[L_{n,B,e}(y) = f(x) y''+\frac{1}{2}f'(x)y' -(n(n+1)x+B)y = 0\]
where $f(x) = 4x(x-1)(x-e)$, that the finiteness of the differential Galois group depends on relations between the fourth singular point $e$ and the other parameters.  For example, Brioschi showed  that  if $n+\frac{1}{2} \in \Zx, $ then there exists a  $P \in \Qx[u,v]$ such that $L_{n,B,e}(y)=0$ has finite differential Galois group if and only if $P(e, B) = 0$.  Algebraic solutions of the Lam\'e equation are a continuing topic of interest, see \cite{bald81,bald87,dwork90,lict,maier2004,BvdW2004,LPU}.\\[0.1in]
Finally, the results of \cite{singer_moduli} have been recast and generalized in 
 \cite{berken02, berken04}. In particular the conditions of Theorem~\ref{diffcons} are shown to be necessary and sufficient and existence and properties of moduli spaces are examined (see also Chapter 12 of \cite{PuSi2003}).\\[0.2in]
 \noindent\underline{\bf Parameterized Picard-Vessiot Theory.} In the previous paragraphs, I examined how differential Galois groups depend algebraically on parameters that may appear in families of differential equations.  Now I will  examine the differential dependence and give an introduction to \cite{CaSi} were  a Galois theory is developed that measures this.  To make things concrete, let us consider parameterized differential equations of the form
 \begin{eqnarray}\label{parameqn1}\frac{\der Y}{\der x}& = &A(x, t_1, \ldots , t_m) Y \ \  \ \  A\in \Mn(\Cx(x, t_1, \ldots , t_m)).\end{eqnarray}
If $x=x_0, 
t_1 = \tau_1, \ldots , t_m = \tau_m$
is a point where the entries of $A$ are analytic, standard existence theory yields a solution $Y = (y_{i,j}(x, t_1, \ldots , t_m)),$ with  \ \ $y_{i,j}$ analytic near  $x=x_0,  t_1 = \tau_1, \ldots , t_m = \tau_m$.  Loosely speaking, we want to define a Galois group that is the group of transformations that preserve all algebraic relations among $x,t_1, \ldots , t_m, $ the $y_{i,j}$  and  the derivatives of the $y_{i,j}$ and their derivatives with respect to {\em all the variables $x,t_1, \ldots ,t_m$}.\\[0.1in]
 To put this in an algebraic setting, we let $k = \Cx(x, t_1, \ldots , t_m)$ be a (partial) differential field with derivations $\Delta = \{\der_0, \der_1, \ldots , \der_m\}, \der_0 = \frac{\der}{\der x}, 
\ \der_i = \frac{\der}{\der t_i} \ i=1, \ldots m $.  We let \[K = k(y_{1,1}, \ldots , y_{n,n},  \ldots ,\der_1^{n_1}\der_2^{n_2}\cdots \der_m^{n_m}y_{i,j}, \ldots) = k\langle y_{1,1}, \ldots ,y_{n,n}\rangle_\Delta.\]  Note that the fact that the $y_{i,j}$ appear as entries of a solution of (\ref{parameqn1}) implies that this field is 
stable under the derivation $\dd_0$. We shall define the parameterized Picard-Vessiot group  (PPV-group) $\DGal_\Delta (K/k)$ to be the group of $k$-automorphisms $\sigma$ of $K$ that commute with all $\der_i, \ i= 0, \ldots , m$.
\begin{ex} \label{paramex1}Let $n=1, m = 1, k = \Cx(x,t), \Delta = \Delta = \{\der_x, \der_t\} $ and \[\frac{\der y}{\der x} = \frac{t}{x} y.\] This equation has $y = x^t = e^{t\log x}$ as a solution.  Differentiating with respect to $t$ and $x$ shows that $K = \Cx(x,t, x^t, \log x)$.   Let $\sigma \in DGal_\Delta(K/k)$.\\[0.1in]
\noindent {\em Claim 1}: $\sigma(x^t) = a x^t, \  a \in K,  \ \der_x(a) = 0$ (and so $a \in \Cx(t)$). To see this note that $\der_x \sigma(x^t) = \frac{t}{x} \sigma(x^t)$ so $\der_x(\sigma(x^t)/x^t) = 0$.\\[0.05in]
\noindent {\em Claim 2}: $\sigma(\log x) = \log x + c, \ c \in \Cx$. To see this, note that  $\der_x \sigma(\log x) = \frac{1}{x}$ so $\der_x(\sigma(\log x) -\log x) = 0$. Since $\der_t(\log x) = 0$ we also have that $\der_t(\sigma(\log x) -\log x) = 0$.\\[0.05in]
\noindent {\em Claim 3}: $\sigma(x^t) = a x^t, \der_x a = 0, \ \der_t a = c a, \ c\in \Cx$.  To see this note that $\sigma(\der_tx^t) = \sigma(\log x \  x^t) = (\log x +c)ax^t = a\log x \  x^t + cax^t$ and  $\der_t\sigma(x^t) = \der_t(ax^t) = (\der_ta)x^t + a\log x \  x^t$.\\[0.05in]
Therefore, \begin{eqnarray*}
\DGal_\Delta(K/k) &=& \{ a \in k^* \ | \ \der_xa = 0 , \ \der_t(\frac{\der_t a}{a}) = 0 \}\\
 & = &\{ a \in k^* \ | \ \der_xa = 0 , \ \der_t^2(a)a - (\der_t(a))^2 = 0 \}
\end{eqnarray*}

\end{ex}
\vspace{.1in}

\noindent Remarks: 
  (1) Let $k_0 = \ker \dd_x = \Cx(t)$. This is a $\dd_t$-field. The PPV-group $\DGal_\Delta(K/k)$ can be identified with a certain subgroup of $\GL_1(k_0)$, that is, with a group of  matrices whose entries satisfy a differential equations.  This is an example of a linear differential algebraic group (to be defined more fully below).\\[0.05in]
 (2) In fact, using partial fraction decompositions, one can see that 
 \[ \{ a \in k^* \ | \ \der_xa = 0 , \ \der_t(\frac{\der_t a}{a}) = 0 \} = \Cx^*\]
  This implies that for any $\sigma \in \DGal_\Delta(K/k), \ \sigma(\log x) = \log x$.  If we want a Fundamental Theorem for this new Galois theory, we will want to have enough automorphisms so that any element $ u \in K\backslash k$ is moved by some automorphism.  This is not the case here.  \\[0.05in]
  (3) If we specialize $t$ to be an element $\tau \in \Cx$, we get a linear differential equation whose monodromy group is generated by the map $y \mapsto e^{2\pi \tau}y$.  One can think of the map $ y \mapsto e^{2\pi t}y$ as the ``parameterized monodromy map''.  I would like to say that this map lies in the PPV-group and that this map generates a group that is, in some sense, dense in the PPV-group.  We cannot do this yet. \\[0.1in]
Remarks (2) and (3) indicate that the PPV group does not have enough elements.  When we considered the Picard-Vessiot theory, we insisted that the constants be algebraically closed.  This guaranteed  that (among other things)   the differential Galois group (defined by polynomial equations)  had enough elements.  Since our PPV-groups will be defined by differential equations, it seems natural to require that the field of $\dd_0$-constants is ``differentially closed''.  More formally,
\begin{defin} Let $k_0$ be a differential field with derivations $\Delta_0 = \{\der_1, \ldots , \der_m\}$. We say that
$k_0$ is {\em differentially closed} if any system of polynomial differential equations \[f_1(y_1, \ldots , y_n) = \ldots
= f_r(y_1, \ldots , y_n) = 0, g(y_1, \ldots , y_n) \neq 0\] with coefficients in $k_0$ that has a solution in some $\Delta_0-$differential extension field has a solution in $k_0$.
\end{defin}
Differentially closed fields play the same role in differential algebra as algebraically closed fields play in algebraic geometry and share many (but not all) of the same general properties and, using Zorn's Lemma, they are not hard to construct. They have been studied by Kolchin (under the name of constrained closed differential fields \cite{constrained}) and by many logicians \cite{marker2000}.   We need  two more definitions before we state the main facts concerning parameterized Picard-Vessiot theory.
\begin{defin} Let $k_0$ be a differential field with derivations $\Delta_0 = \{\der_1, \ldots , \der_m\}$.\\[0.05in]
(1) A set $X \subset k_0^n$ is {\em Kolchin closed} if it is the zero set  of a system of differential polynomials  over $k_0$. \\[0.05in]
(2) A {\em linear differential algebraic group} is a subgroup $G \subset\GL_n(k_0)$ that is Kolchin closed in $\GL_n(k_0) \subset k_0^{n^2}$.
\end{defin}
Kolchin closed sets form the closed sets of a topology called the Kolchin topology. I will therefore use topological language (open, dense, etc) to describe their properties. Examples and further description of linear differential algebraic groups are given after the following definition and result.
\begin{defin} Let $k$ be a differential field with derivations $\Delta = \{\der_0, \der_1, \ldots , \der_m\}$ and let
\begin{eqnarray}\label{PVex2}
\der_0Y = AY, \ \ A \in \Mn(k)
\end{eqnarray}
A {\em Parameterized Picard-Vessiot extension (PPV-extension)} of $k$ for (\ref{PVex2}) is a $\Delta$-ring $R \supset k$ such that $R$ is a simple $\Delta$-ring ({\em i.e.}, its only $\Delta$-ideals are $(0)$ and $R$) and $R$ is generated as a $\Delta$-ring by the entries of some $Z \in \GL_n(R)$ and $\frac{1}{\det Z}$ where $\der_0Z = AZ$. A {\em PPV-field} is defined to be the quotient field of a PPV-ring.
\end{defin}
As in the usual Picard-Vessiot theory, such rings exist and are domains.  Under the assumption that $k_0 = \ker \dd_0$ is a differentially closed $\Delta_0 = \{\der_1, \ldots , \der_m\}$-field, PPV-rings can be shown to be unique (up to appropriate isomorphism).  

\begin{thm} Let $k$ be a differential field with derivations $\Delta = \{\der_0, \der_1, \ldots , \der_m\}$ and assume that $k_0 = \ker \dd_0$ is a differentially closed $\Delta_0 = \{\der_1, \ldots , \der_m\}$-field. Let $K$ be the PPV-field for an equation \[\der_0Y = AY, \ \ A \in \Mn(k).\] \\[0.05in]
(1) The group $\DGal_\Delta(K/k)$ (called the PPV-group) of $\Delta$-differential $k$-automorphisms of $K$ has the structure of a linear differential algebraic group.\\[0.05in]
(2) There is a Galois correspondence between $\Delta$-subfields $E$ of $K$ containing $k$ and Kolchin $\Delta_0$ closed subgroups $H$ of $\DGal_\Delta(K/k)$ with normal subgroups $H$ corresponding to field $K$ that are again PPV-extensions of $k$. 
\end{thm}

In the context of this theorem we can recast the above example. Let $k_0, \Cx(t) \subset k_0$ be a differentially closed $\dd_t$-field and $k = k_0(x)$ be  a $\Delta = \{\dd_x, \dd_t\}$ field where $\dd_x(k_0) = 0, \dd_x(x) = 1, \dd_t(x) = 0$.  One can show that the field $K = k<x^t>_\Delta = k(x^t, \log )$ is a PPV-field for the equation $\dd_xY = \frac{t}{x}Y$ with PPV group as described above.  Because $k_0$ is differentially closed, this PPV group has enough elements so that the Galois correspondence holds. Furthermore the element $e^{2\pi it} \in k_0$ is an element of this group and, in fact, generates a Kolchin dense subgroup.\\[0.1in]
Remarks: (1) In general, let $\dd_xY = AY$ is a parameterized equation with $A \in \Cx(x,t_1, \ldots, t_m)$.  For each value of $\bar{t}$ of $t = (t_1, \ldots, t_m)$ at which the entries of $A$ are defined and generators $\{\gamma_i\}$ of the Riemann Sphere punctured at the poles of $A(x, \bar{t}_1, \ldots \bar{t}_m)$, we can define monodromy matrices $\{M_i(\bar{t})\}$. For all values of $t$ in a sufficiently small open set, the entries  of $M_i$ are analytic functions of $t$ and one can show that these matrices belong to the PPV-group.  If the equation $\dd_xY = AY$ has only regular singular points, then one can further show that these matrices generate a Kolchin dense subgroup of the PPV group.\\[0.05in]
(2) One can consider parameterized families of linear partial differential equations as well. Let $k$ be a $\Delta = \{\dd_0, \ldots , \dd_s, \dd_{s+1}, \ldots ,\dd_m\}$ and let 
\[\dd_iY = A_iY, \ A_i\in \Mn(k), \ i = 0, \ldots, s\]
be a system of differential equations with $\dd_i(A_j) +A_iA_j = \dd_j(A_i) + A_jA_i$.  One can develop a parameterized differential Galois theory for these as well. \\[0.05in]
(3) One can develop a theory of parameterized liouvillian functions and show that one can solve parameterized linear differential equations in terms of these if and only if the PPV group is solvable-by-finite.\\[0.1in]
To clarify and apply the PPV theory, we will describe some facts about linear differential algebraic groups.\\[0.1in]
\noindent\underline{\em Linear differential algebraic groups.}  These groups were introduced by Cassidy \cite{cassidy1}. A general theory of differential algebraic groups was initiated by Kolchin \cite{kolchin_groups} and developed by Buium, Hrushovsky, Pillay, Sit and others (see \cite{CaSi} for references). We give here some examples and results. Let $k_0$ be a differentially closed $\Delta_0 = \{\dd_1, \ldots ,\dd_m\}$ differential field.
\begin{exs}\label{diffgpexs} (1) All linear algebraic groups are linear differential algebraic groups.\\[0.05in]
(2) Let $C = \cap_{i=1}^m \ker\der_i$ and let $G(k_0)$ be a linear algebraic 
 group defined over $k_0$.  Then $G(C)$ is a linear {\em differential} algebraic group (just add 
 $\{\der_1y_{i,j} = \ldots =\der_my_{i,j} = 0\}_{i,j=1}^n$ to the defining equations!).  \\[0.05in]
 (3) Differential subgroups of \[G_a(k_0) = (k_0,+)
 = \{\left(\begin{array}{cc}  1&z\\0&1\end{array}\right) \ | \ z\in k \}. \]
  The linear differential  subgroups  of this group are all of the form
 \[G_a^{\{L_1, \ldots , L_s\}} = \{ z \in k_0 \ | \ L_1(z)= \ldots =L_s(z) = 0\}\]
 \noindent where the $L_i$ are linear differential operators in $k_0[\der_1, \ldots, \der_m]$.   When $m=1$ we may always take $s=1$. \\[0.2in]
 For example, if $m=1,$ \[G_a^\der = \{z \in k_0 \ | \ \der z = 0\} = G_a(C)\]
(4) Differential subgroups of $G_m(k_0) = (k_0^*,\times)
 = \GL_1(k_0)$.  Assume  $m=1, \Delta_0 = \{\der\}$ (for the general case, see \cite{cassidy1}).  The linear differential  subgroups are:
  \begin{itemize}
\item finite and cyclic, or
\item $G_m^{L} = \{ z \in k_0^* \ | \ L(\frac{\der z}{z})=  0\}$ where the $L$ is a linear differential operator in $k_0[\der]$.
\end{itemize}
For example the PPV group of Example~\ref{paramex1},  $G_m^\der = \{z \in k_0 \ | \ \der (\frac{\der z}{z}) = 0\}$ is of this form. 
The above characterization follows from the exactness of
\[(1)\longrightarrow G_m(C) \longrightarrow G_m(k_0) \stackrel{z\mapsto \frac{\der z}{z}}{\longrightarrow}G_a(k_0) \longrightarrow (0)\]
and the characterization of differential subgroups of $G_a$.
Note that in the usual theory of linear algebraic groups there is no surjective homomorphism from $G_m$ to $G_a$.\\[0.05in]
(5) A result of Cassidy \cite{cassidy1} states: if  $m = 1$, $\Delta_0= \{\dd\}$ and $H$ is a Zariski dense proper differential subgroup of $\SL_n(k_0)$, then there is an element  $g \in \SL_n(k_0)$ such that $gHg^{-1} = \SL_n(C)$, where $C = \ker \dd$.  This has been generalized by Cassidy and, later, Buium to the following.
Let $H$ be a Zariski dense proper differential subgroup of $G$,  a simple linear algebraic group, defined over $\Qx$.  Then there  exists $\Dx = \{D_1, \ldots D_s \}\subset k\Delta_0 = k-\mbox{span of } \Delta_0$ such that
  \begin{enumerate}
  \item $[D_i,D_j] = 0$ for $1\leq i,j\leq s$ and
  \item There exists a $ g \in G(k) \mbox{ such that } gHg^{-1} = G(C_\Dx)$, where $C_\Dx = \cap_{i=1}^s \ker D_i$.
 \end{enumerate}
\end{exs}

\vspace{0.15in}

\noindent \underline{\em Inverse Problem.} In analogy to the Picard-Vessiot theory, the question I will discuss is: Given a differential field $k$, which linear differential algebraic groups appear a PPV-groups over $k$?  Even in simple cases the complete answer is not known and there are some surprises.\\[0.1in]
Let $\Delta = \{\dd_x, \dd_t\}, \Delta_0 = \{\dd_x\}$, $k_0$ a $\dd_t$-differentially closed $\dd_t$-field and $k = k_0(x)$ with $\dd_0(k_0) = 0, \dd_x(x) = 1$, an $\dd_t(x) = 0$. Let $C\subset k_0$ be the $\Delta$-cnstants of $k$. I will show that all differential subgroups of $G_a = (k_0, +)$ appear as PPV-groups over this field but that the full group $G_a$ does not appear as a PPV-group \cite{CaSi}.  \\[0.1in]
One can show that the task of describing the differential subgroups of $G_a$ that appear as PPV-groups  is equivalent to describing which linear differential algebraic groups appear as PPV-groups for equations of the form 
\[\frac{dy}{dx} = a(t,x), \ a(t,x) \in k_0(x)\]
Using partial fraction decomposition, one sees that a solution of such an equation is of the form
\[y = R(t,x)+\sum_{i=1}^s a_i(t) \log(x-b_i(t)), R(t,x) \in k_0(x), a_i(t), b_i(t) \in k_0\]
and the associated PPV-extension is of the form $K = k(y)$.  If $H$ is the PPV-group of $K$ over $k$ and $\sigma \in H$, then one shows that 
\[\sigma(y) = R(t,x)+\sum_{i=1}^s a_i(t) \log(x-b_i(t)) + \sum_{i=1}^sc_ia_i(t)\]
for some $c_i \in C$.  Let $L\in k_0[\frac{d}{dt}]$ have solution space spanned by the $a_i(t)$ (note that $L \neq 0$). For any $\sigma \in H$,  $\sigma$ can be identified with an element of 
    \[H = \{z\in G_a(k_0) \ | \ L(z)=0.\}\]
    Conversely, any element of $H$ can be shown to induce a differential automorphism of $K$, so $H$ is the PPV-group.  Therefore $H$ must be a \underline{proper} subgroup of $G_a$ and not all of $G_a$.  One can furthermore show that all proper differential algebraic subgroups of $G_a$ can appear in this way. \\[0.1in]
 A distinguishing feature of proper subgroups of $G_a$ is that they contain a finitely generated Kolchin dense subgroup whereas $G_a$ does not have such a subgroup.  A possible conjecture is that a linear differential algebraic group $H$ is a PPV-group over $k$ if and only if $H$ has a finitely generated Kolchin dense subgroup.  If $H$ has such a finitely generated subgroup, a parameterized solution of Hilbert's $22^{nd}$ problem  shows that one can realize this group as a PPV-group.  If one could establish a parameterized version of Ramis's theorem that the formal monodromy, exponential torus and Stokes generate a a dense subgroup of the differential Galois group (something of independent interest), the other implication of this conjecture would be true.\\[0.1in]
I will now construct a field over which one can realize $G_a$ as a PPV-group.  Let $k$ be as above and let $F = k(\log x, x^{t-1}e^{-x})$. The {\em incomplete Gamma function}  \[\gamma(t,x) = \int_0^x s^{t-1}e^{-s} ds\]  satisfies
\[\frac{d \gamma}{dx} = x^{t-1}e^{-x} \]
over $F$.  The PPV-group over $k$ is $G_a(k_0)$ because $\gamma, \frac{d\gamma}{dt}, \frac{d^2\gamma}{dt^2}, \ldots $ are algebraically independent over
$k$ \cite{JRR} and so there are no relations to preserve. Over $k= k_0(x)$, $\gamma(t,x)$ satisfies
\[\frac{d^2\gamma}{dx^2}- \frac{t-1-x}{x} \  \frac{d\gamma}{dx} = 0\]
and the PPV-group is 
\begin{eqnarray*}
H& =& \{ \left( \begin{array}{cc} 1&a\\0&b \end{array}\right) \ | \ a \in k_0, b\in k_0^*, \der_t(\frac{\der_tb}{b} )=0\}\\
  &=& G_a(k_0) \rtimes G_m^{\der_t}, \ G_m^{\der_t} = \{ b\in k_0^* \ | \ \der_t(\frac{\der_t b}{b}) = 0\end{eqnarray*}
 
\vspace{0.15in}
\noindent\underline{\em Isomonodromic families.}  The parameterized Picard-Vessiot theory can be used to characterize isomonodromic families of linear differential equations with regular singular points, that is, families of such equations where the monodromy is independent of the parameters.  
\begin{defin} Let $k$ be a $\Delta = \{\dd_0, \ldots, \dd_m\}$-differential field and let $A \in \gl_n(k)$.  We say that $\dd_0Y=AY$ is {\em completely integrable} if there exist $A_i \in \gl_n(k), i=0,1,\ldots ,m$ with $A_0 = A$ such that \[\dd_jA_i - \dd_iA_j = A_jA_i - A_i A_j \mbox{ for all } i,j = 0, \ldots m\]
\end{defin}
The nomenclature is motivated by the fact that the latter conditions on the $A_i$ are the usual integrability conditions on the system of differential equations $\dd_iY = AY, i = , \ldots ,m$.  
\begin{prop}\cite{CaSi} Let $k$ be a $\Delta = \{\dd_0, \ldots, \dd_m\}$-differential field and let $A \in \gl_n(k)$.  Assume that $k_0 = \ker \dd_0$ is $\Delta_0=\{\dd_1, \ldots, \dd_m\}$-differentially closed field and let $K$ be the PPV-field of $k$ for $\dd_0Y=AY$ with $A\in \gl_n(k)$. Let $C = \cap_{i=0}^m \ker\dd_i$.\\[0.05in]
There exists a linear algebraic group $G$, defined over $C$, such that \\ $\DGal_\Delta(K/k)$ is conjugate to $G(C)$ over $k_0$ if and only if $\dd_0Y=AY$ is completely integrable.  If this is the case, then $K$ is a Picard-Vessiot extension corresponding to this integrable system.
\end{prop}
Let $k = \Cx(x, t_1, \ldots , t_n)$ and $A \in \gl_n(k)$.  For $t = (t_1, \ldots , t_m)$ in some sufficiently small open set $\calO \subset \Cx^{n}$, one can select generators $\{\gamma_i(t_i, \ldots, t_m)\}$ of the fundamental group of the Riemann Sphere minus the (parameterized) singular points  such that the monodromy matrices $\{M_i(t_1, \ldots, t_m)\}$   depend analytically on the parameters.  One says that the differential equation $\dd_xY = AY$ is {\em isomonodromic} if for some sufficiently small open set $\calO$, the $M_i$ are independent of $t$.  In \cite{CaSi} we deduce the following corollary from the above result.
\begin{cor}  Let $k, A, \calO$ be as above and assume that $\dd_xY = AY$ has only regular singular points for $t \in \calO$.  Then $\dd_xY=AY$ is isomonodromic if and only if the PPV-group is conjugate to a linear algebraic group $G(\Cx) \subset \GL_n(\Cx).$
\end{cor}

\vspace{0.1in}

\noindent \underline{\em Second order  parameterized equations}. A classification of the linear differential subgroups of $\SL_2$ allows one to classify parameterized systems of second order linear differential equations. To simplify the discussion we will consider equations of the form 
\[\frac{\dd Y}{\dd x} = A(x,t)Y\]
that depend on only one parameter and satisfy $A \in \sll_n(\Cx(x,t))$.  \\[0.1in]
\noindent Let $H$ be a linear differential algebraic subgroup of $\SL_2(k_0)$ (where $k_0$ is a $\dd_t$-differentially closed field containing $\Cx(t)$ and $\ker \dd_t = \Cx$)  and $\bar{H}$ its Zariski closure.  If $\bar{H} \neq \SL_2$, then one can show that $\bar{H}$ and therefore $H$ is solvable-by-finite \cite{kovacic86,cassidy1}.  Furthermore, Cassidy's result (Example~\ref{diffgpexs}(3)) states that if  $\bar{H} = \SL_2$ but $H \neq \SL_2$, then $H$ is conjugate to $\SL_2(\Cx)$.  From this discussion and the discussion in the previous paragraphs one can conclude the following.
\begin{prop} Let $k$ and $\frac{\dd Y}{\dd x} = AY$ be as above.  Assume that for some small open set of values of $t$ that this equation has only regular singular points.  Then either 
\begin{itemize}
\item the PPV group is $\SL_2$, or
\item the equation is solvable in terms of parameterized exponentials, integrals and algebraics, or
\item the equation is isomonodromic.
\end{itemize}
\end{prop}
The notion of `solvable in terms of  parameterized exponentials, integrals and algebraics' is analogous to the similar notion in the usual Picard-Vessiot theory and is given a formal definition in \cite{CaSi}, where it is also shown that this notion is equivalent to the PPV-group being solvable-by-finite. 
  
\section{Final Comments}

I have touched on only a few of the aspects of the Galois theory of linear differential equations.  I will indicate here some of the other aspects and give pointers to some of  the current literature.\\[0.1in]
Kolchin himself generalized this theory to a theory of {\em strongly normal} extension fields \cite{DAAG}.  The Galois groups in this theory are arbitrary algebraic groups. Recently Kovacic \cite{kov03,kov06} has recast this theory in terms of groups schemes and differential schemes.  Umemura \cite{u1,u2,u3,u4,u5,u6} has developed a Galois theory of differential equations for general nonlinear equations.  Instead of Galois groups, Umemura uses Lie algebras to measure the symmetries of differential fields.  Malgrange \cite{malgrange_galois1, malgrange_galois2} has proposed a Galois theory of differential equations where the role of the Galois group is replaced by certain groupoids.  This theory has been expanded and applied by Cassale \cite{cas1,cas2,cas3}. Pillay \cite{pillaygalois2, pillaygalois3, pillaygalois1, pillay2004, pillaygalois4} develops a Galois theory where the Galois groups can be arbitrary {\em differential} algebraic groups and, together with Bertrand, has used these techniques to generalize Ax's work on Schanuel's conjecture (see \cite{PilBert, bert06}). Landesman \cite{Landesman} has generalized Kolchin's theory of strongly normal extensions to a theory of strongly normal parameterized extensions.  The theory of Cassidy and myself presented in Section~\ref{mfssec6} can be viewed as a special case of this theory and many of the results of \cite{CaSi} can be derived from the theories of Umemura and Malgrange as well. \\[0.1in]
 A Galois Theory of linear {\em difference} equations has also been developed, initally by Franke (\cite{franke63} and several subsequent papers), Bialynicki-Birula \cite{BB2} and more fully by van der Put and myself \cite{PuSi}. Andr\'e \cite{andre_galois} has presented a Galois theory that treats both the difference and differential case in a way that allows one to see the differential case as a limit of the difference case.   Chatzidakis, Hrushovski and Kamensky have used model theoretic tools to develop Galois theory of difference equations (see \cite{ChHr1}, \cite{kamensky} and \cite{CHS07} for a discussion of connections with the other theories). Algorithmic issues have been considered by Abramov, Barkatou, Bronstein, Hendriks, van Hoeij, Kauers, Petkovsek, Paule, Schneider, Wilf, M. Wu and Zeilberger and there is an extensive literature on the subject for which I will refer to MathSciNet and the ArXiV. Recently  work by Andr\'e, Di Vizio, Etingof, Hardouin, van der Put, Ramis,  Reversat, Sauloy, and Zhang,  has been done on understanding the Galois theory of $q$-difference equation (see \cite{divi1, divi2,divi3, etingof, vdp_reversat, sauloy_galois, sau1, sau2, sau3, sau4, hard_it}).  Fields with both derivations and automorphisms have been studied from a model theoretic point of view by Bustamente \cite{bustamente} and  work of Abramov,  Bronstein, Li, Petkovsek, Wu, Zheng and myself \cite{abramov_petkovsek, AbBrPe95,BronPet, BrPe94, blw, lswz,  wu05} consider algorithmic questions for mixed differential-difference systems. The Picard-Vessiot theory of linear differential equations has been used by Hardouin \cite{hardouin06} to study the differential properties of solutions of difference equations. A Galois theory, with linear differential algebraic groups as Galois groups, designed to measure the differential relations among solutions of  difference equation,  has been recently developed and announced in \cite{sing_oberwol}. \\[0.1in]
 Finally,  I have not touched on the results and enormous literature concerning 
 the hypergeometric equations and their generalizations as well as arithmetic properties of differential equations.  For a taste of the current research, I refer the reader to \cite{sst} and \cite{hyper07} for the former and \cite{dwork_theory} for the latter.  
  













\newcommand{\SortNoop}[1]{}\def\cprime{$'$}



\end{document}